\documentclass[a4paper]{amsart}
\usepackage{amsmath} 
\usepackage{amssymb}
\usepackage{latexsym}
\usepackage[mathscr]{eucal}
\usepackage[varg]{txfonts}

\theoremstyle{plain}
\newtheorem{thm}{Theorem}[subsection]
\newtheorem{lem}[thm]{Lemma}
\newtheorem{prop}[thm]{Proposition}

\theoremstyle{plain}
\theoremstyle{definition}

\theoremstyle{remark}

\font\es=eufm10

\def\gC{\mbox{\es {C}}}

\def\dfrac#1#2{\displaystyle \frac{#1}{#2}}
\def\Aut{\mbox{\rm {Aut}}}
\def\det{\mbox{\rm {det}}}
\def\diag{\mbox{\rm {diag}}}

\def\Iso{\mbox{\rm {Iso}}}

\def\Ker{\mbox{\rm {Ker}}}

\def\tr{\mbox{\rm {tr}}}

\def\ov{\overline}
\def\wti{\widetilde}
\def\ti{\tilde}

\def\dfrac#1#2{\displaystyle \frac{#1}{#2}}

\def\C{\mbox{\boldmath $C$}}

\def\H{\mbox{\boldmath $H$}}

\def\R{\mbox{\boldmath $R$}}

\def\Z{\mbox{\boldmath $Z$}}

\def\sR{\mbox{\boldmath $\scriptstyle{R}$}} 
\def\sC{\mbox{\boldmath $\scriptstyle{C}$}}

\def\0{\mbox{\boldmath {0}}}    
\def\1{\mbox{\boldmath {1}}}      
\def\2{\mbox{\boldmath {2}}}      
\def\3{\mbox{\boldmath {3}}}      
\def\4{\mbox{\boldmath {4}}}      
\def\5{\mbox{\boldmath {5}}}      
\def\6{\mbox{\boldmath {6}}}      
\def\7{\mbox{\boldmath {7}}}      
\def\8{\mbox{\boldmath {8}}}      
\def\9{\mbox{\boldmath {9}}}      
\def\a{\mbox{\boldmath $a$}}
\def\b{\mbox{\boldmath $b$}}

\def\i{\mbox{\boldmath $i$}}

\def\m{\mbox{\boldmath $m$}}

\def\x{\mbox{\boldmath $x$}}
\def\y{\mbox{\boldmath $y$}}

\usepackage{epic,eepic}

\begin{document}

\title[ Realizations of globally exceptional $\mathbb{Z}_2 \times \mathbb{Z}_2$- symmetric spaces]
{Realizations of globally exceptional
$\mathbb{Z}_2 \times \mathbb{Z}_2$- symmetric spaces}

\dedicatory{Dedicated to Professor Ichiro Yokota on the occasion of his eighty-eighth birthday}
 
\author[Toshikazu Miyashita]{Toshikazu Miyashita}



\begin{abstract}
In \cite{Andreas}, a classification is given of the exceptional $\mathbb{Z}_2 \times \mathbb{Z}_2$-symmetric spaces $G/K$, where $G$ is an exceptional compact Lie group or $S\!pin(8)$, and moreover the structure of $K$ is determined as Lie algebra. 
 In the present article, 
we give a pair of commuting involutive  automorphisms (involutions) $\tilde{\sigma}, \tilde{\tau}$ of $G$ concretely and determine the structure of group $G^{\sigma} \cap G^{\tau}$ corresponding to Lie algebra $\mathfrak{g}^\sigma \cap \mathfrak{g}^\tau$, where $G$ is an exceptional compact Lie group.
Thereby, we realize exceptional $\mathbb{Z}_2 \times \mathbb{Z}_2$-symmetric spaces, globally.
 

\end{abstract}

\maketitle


\section{Introduction}

According to the article \cite{Andreas}, the notion of $\varGamma$-symmetric spaces introduced by Lutz \cite{Lutz}, is a generalization of the classical notion of a symmetric space, where $\varGamma$ is a finite abelian group. (As for the definition of $\varGamma$-symmetric space, see \cite{Bahturin}.) 
In the case $\varGamma=\mathbb{Z}_2$ this is the classical definition of symmetric spaces, 
and in the case $\varGamma=\mathbb{Z}_2 \times \mathbb{Z}_2$ we say that this is $\mathbb{Z}_2 \times \mathbb{Z}_2$-symmetric space. Now, the definition of $\mathbb{Z}_2 \times \mathbb{Z}_2$-symmetric space in \cite{Andreas} is as follows.
\vspace{1mm}

\noindent {\bf Definition.}\, A homogeneous space $G/K$ is $\mathbb{Z}_2 \times \mathbb{Z}_2$-symmetric space if there are $\tilde{\sigma}, \tilde{\tau} \in \Aut(G)\setminus \{\text{id}_G \}$ such that ${\tilde{\sigma}}^2={\tilde{\tau}}^2=\text{id}_G, \tilde{\sigma} \ne \tilde{\tau}$ and $\tilde{\sigma}\tilde{ \tau}=\tilde{\tau} \tilde{\sigma}$ such that $(G^\sigma \cap G^\tau)_0 \subseteq K \subseteq G^\sigma \cap G^\tau$, where $G^\sigma$(resp.$G^\tau$) is a fixed points subgroup of $G$ by $\tilde{\sigma}$ (resp.$\tilde{\tau}$) and $(G^\sigma \cap G^\tau)_0$ is a connected component containing $1$ of $G^\sigma \cap G^\tau$. (Hereafter $\text{id}_G$ is abbreviated as $1$.) 
\vspace{1mm}

The main purpose of this article is to give a pair of different involutive  automorphisms $\tilde{\sigma}, \tilde{\tau}$ in $G$ and to determine the structure of the group $G^\sigma \cap G^\tau$corresponding to $\mathfrak{k}=\mathfrak{g}^\sigma \cap \mathfrak{g}^\tau$ in the second column of Table 1, where $G$ is a simply connected compact exceptional Lie group $G_2, F_4, E_6, E_7$ or $E_8$. 
Thereby, we realize exceptional $\mathbb{Z}_2 \times \mathbb{Z}_2$-symmetric spaces, globally. We call those spaces "Globally exceptional $\mathbb{Z}_2 \times \mathbb{Z}_2$-symmetric spaces".
Moreover we confirm all types $(G/G^\sigma\!, G/G^\tau \!, G/G^{\sigma\tau})$ of $\mathbb{Z}_2 \times \mathbb{Z}_2$-symmetric spaces determined by Andreas kollross, globally. For example,
since it follows from the triple group isomorphisms  $(E_6)^{\lambda\gamma}\! \cong \!S\!p(4)/\Z_2, (E_6)^{\lambda\gamma\sigma}\!\cong \! (E_6)^{\lambda\gamma} \!\cong \! S\!p(4)/\Z_2, (E_6)^{(\lambda\gamma)(\lambda\gamma\sigma)}\!=\!(E_6)^{\sigma} \! \cong (U(1) \times S\!pin(10))/\Z_4$  
 that $E_6/(E_6)^{\lambda\gamma}, E_6/(E_6)^{\lambda\gamma\sigma}, E_6/(E_6)^{(\lambda\gamma)(\lambda\gamma\sigma)}$ are the symmetric spaces of type EI, EI, EIII, respectively. Then 
the globally $\mathbb{Z}_2 \times \mathbb{Z}_2$-symmetric space of this type
is called type EI-EI-EIII, and denote EI-EI-EIII by  abbreviated form EI-I-III. 
In addition, when $\sigma$ and $\tau$ are conjugate in $G$, we give explicitly  the element $\delta \in G$ such that $\sigma=\delta\tau\delta^{-1}$ except for three cases in $E_8$.

This article is closely in connection with the preceding articles \cite{M.Y.01}, \cite{M.01}, 
\cite{M.02}, \cite{Miyashitatoshikazu}, \cite{Yokotaichiro0}, \cite{Yokotaichiro1}, 
\cite{Yokotaichiro2} and \cite{Yokotaichiro3}, and may be a continuation of those in some sense. 

J-S.H and J.U \cite{Jing-Song Huang} classified the Klein four subgroups $\Gamma$ of $\Aut(\mathfrak{u}_0)$ for each compact Lie algebra $\mathfrak{u}_0$ by calculating the symmetric subgroups $\Aut(\mathfrak{u}_0)^\theta$ ($\theta \in \Aut(\mathfrak{u}_0)$ is a involutive automorphism) and their involution classes, and determined the fixed point subgroup $\Aut(\mathfrak{u}_0)^\Gamma$. 
In general, suppose a group $G$ is simply connected, we have $\Aut(G) \cong \Aut(\mathfrak{g})$ (\cite{MS}. $\mathfrak{g}$ is the Lie algebra of $G$), moreover when the center $z(G)$ of $G$ is trivial, it is well known that $G \subset \Aut(G)$. Since the exceptional compact Lie groups $G=G_2, F_4, E_8$ are simply connected and these $z(G)$ are trivial , we see that $G \subset \Aut(G) \cong \Aut(\mathfrak{g})$. Hence, for $G=G_2, F_4, E_8$, our results of $G^\sigma \cap G^\tau$ in Table 1 are realized as the subgroups of the results of fixed point subgroups of Klein four subgroups in exceptional case of \cite{Jing-Song Huang}.

In \cite{Jing-Song Huang}, they had approached the ends by using root system of $\mathfrak{u}_0$. On the other hand, we define
the mappings between groups explicitly, and give the proofs of isomorphism of group by using homomorphism theorem as elementary approach. 
The author would like to say that this is one of features about this article.
\vspace{1mm}

For $G=G_2, F_4, E_6, E_7$, and $E_8$, our results are as follows.
\vspace{-3mm}

\begin{center}
\begin{small}
\begin{tabular}{|l|c|l|c|l|}
\hline
{}&&&&\\[-5pt]
\hspace{5mm}Type & $\mathfrak{g}$ & \hspace{8mm}$\mathfrak{k} (\cong \mathfrak{g}^\sigma \cap \mathfrak{g}^\tau$) & Involutions & \hspace{20mm}$G^\sigma \cap G^\tau$ 
\\[3pt]
\hline \hline 
{}&&&&\\[-6pt]
G-G-G & $\mathfrak{g}_2$ & $\i\R \oplus \i\R$ & $\gamma, \gamma_{{}_{\scriptscriptstyle {H}}}$ & $(U(1) \times U(1))/\Z_2 \rtimes \{1, \gamma_{{}_{\scriptscriptstyle {C}}}\}$ 
\\[2pt]
\hline\hline
{}&&&&\\[-6pt]
FI-I-I & $\mathfrak{f}_4$ & $\mathfrak{u}(3)\oplus \i\R$ &$\gamma, \gamma_{{}_{\scriptscriptstyle {H}}}$& $(U(1) \times U(1) \times S\!U(3))/\Z_3 \rtimes \{1, \gamma_{{}_{\scriptscriptstyle {C}}}\}$ 
\\[2pt]
\hline
{}&&&&\\[-6pt]
FI-I-II &  $\mathfrak{f}_4$ & $\mathfrak{sp}(2) \oplus \mathfrak{sp}(1) \oplus \mathfrak{sp}(1)$ &$\gamma, \gamma\sigma$&$(S\!p(1) \times S\!p(1) \times S\!p(2))/\Z_2$ 
\\[2pt]
\hline
{}&&&&\\[-6pt]
FII-II-II & $\mathfrak{f}_4$ & $\mathfrak{so}(8)$ &$\sigma, \sigma'$&$S\!pin(8)$
\\[2pt]

\hline\hline
{}&&&&\\[-6pt]
EI-I-II & $\mathfrak{e}_6$ &  $\mathfrak{so}(6) \oplus \i\R$ &$\lambda\gamma, \lambda\gamma\gamma_{{}_{\scriptscriptstyle {C}}}$ & $(U(1) \times S\! O(6))/\Z_2 \rtimes \{1,\gamma_{{}_{\scriptscriptstyle {H}}}\}$ 
\\[2pt]
\hline
{}&&&&\\[-6pt]
EI-I-III & $\mathfrak{e}_6$ & $\mathfrak{sp}(2) \oplus \mathfrak{sp}(2)$ &$\lambda\gamma, \lambda\gamma\sigma$&$(S\!p(2) \times S\!p(2))/\Z_2 \rtimes \{1, \rho \}$
\\[2pt]
\hline
{}&&&&\\[-6pt]
EI-II-IV & $\mathfrak{e}_6$ & $\mathfrak{sp}(3) \oplus \mathfrak{sp}(1)$ &$\lambda\gamma, \gamma$&$(S\!p(1) \times S\!p(3))/\Z_2$
\\[2pt]
\hline
{}&&&&\\[-6pt]
EII-II-II & $\mathfrak{e}_6$ &  $\mathfrak{su}(3) \oplus \mathfrak{su}(3) \oplus \i\R \oplus \i\R$ & $\gamma, \gamma_{{}_{\scriptscriptstyle {H}}}$ & $(U(1) \times U(1) \times S\!U(3) \times S\!U(3))/\Z_3 \rtimes \{1, \gamma_{{}_{\scriptscriptstyle {C}}}\}$
\\[2pt]
\hline
{}&&&&\\[-6pt]
EII-II-III & $\mathfrak{e}_6$ &  $\mathfrak{su}(4) \oplus \mathfrak{sp}(1) \oplus \mathfrak{sp}(1) \oplus \i\R$ &$\gamma, \sigma\gamma$ & $(S\!p(1) \times S\!p(1) \times U(1) \times S\!U(4))/(\Z_2 \times \Z_4)$
\\[2pt]
\hline
{}&&&&\\[-6pt]
EII-III-III & $\mathfrak{e}_6$ & $\mathfrak{su}(5) \oplus  \i\R \oplus \i\R$ &$\gamma, \gamma_{{}_{\scriptscriptstyle {H}}} \rho_{{}_{\scriptscriptstyle {2}}}$&$(U(1) \times U(1) \times S\!U(5))/(\Z_2 \times \Z_5)$
\\[2pt]
\hline
{}&&&&\\[-6pt]
EIII-III-III & $\mathfrak{e}_6$ & $\mathfrak{so}(8) \oplus  \i\R \oplus \i\R$ & $\sigma, \sigma'$ &$(U(1) \times U(1) \times S\!pin(8))/(\Z_2 \times \Z_4)$
\\[2pt]
\hline
{}&&&&\\[-6pt]
EIII-IV-IV & $\mathfrak{e}_6$ & $\mathfrak{so}(9)$ &$\lambda,\sigma$&$S\!pin(9)$ 
\\[2pt]
\hline\hline

{}&&&&\\[-6pt]
EV-V-V & $\mathfrak{e}_7$ & $\mathfrak{so}(8) $ &$\lambda\gamma, \iota\gamma_{{}_{\scriptscriptstyle {C}}}$&$S\!O(8)/\Z_2 \times \{1, -1 \}$ 
\\[2pt]
\hline
{}&&&&\\[-6pt]
EV-V-VI & $\mathfrak{e}_7$ & $\mathfrak{su}(4)\oplus \mathfrak{su}(4) \oplus \i\R $ &$\lambda\gamma, \lambda\gamma\sigma$&$(U(1) \times S\!U(4) \times S\!U(4))/(\Z_2 \times \Z_4) \rtimes \{1, \varepsilon \}$ 
\\[2pt]
\hline
{}&&&&\\[-6pt]
EV-V-VII & $\mathfrak{e}_7$ & $\mathfrak{sp}(4) $ &$\lambda\gamma, \iota\lambda\gamma$&$S\!p(4)/\Z_2 \times \{1, -1 \}$ 
\\[2pt]
\hline
{}&&&&\\[-6pt]
EV-VI-VII & $\mathfrak{e}_7$ & $\mathfrak{su}(6)\oplus \mathfrak{sp}(1) \oplus \i\R $ &$\lambda\gamma, \gamma$ &$(U(1) \times S\!U(2) \times S\!U(6))/\Z_{24} $
\\[2pt]
\hline
{}&&&&\\[-6pt]
EVI-VI-VI & $\mathfrak{e}_7$ & $\mathfrak{so}(8)\oplus \mathfrak{so}(4) \oplus \mathfrak{sp}(1) $ &$\gamma, -\sigma$&$(S\!U(2) \times S\!pin(4) \times S\!pin(8))/(\Z_2 \times \Z_2)$ 
\\[1pt]
& $\mathfrak{e}_7$ & \hspace{1.5mm}$\mathfrak{u}(1) \oplus \mathfrak{su}(6) \oplus \i\R$&$\gamma, \gamma_{{}_{\scriptscriptstyle {H}}}$&$(U(1) \times U(1) \times S\!U(6))/\Z_3 \rtimes \{1, \gamma_{{}_{\scriptscriptstyle {C}}} \}$
\\[2pt]
\hline
{}&&&&\\[-6pt]
EVI-VII-VII & $\mathfrak{e}_7$ & $\mathfrak{so}(10)\oplus \i\R \oplus \i\R $ &$-\sigma, \iota$ &$(U(1) \times U(1) \times S\!pin(10))/\Z_{12}$ 
\\[2pt]
\hline
{}&&&&\\[-6pt]
EVII-VII-VII & $\mathfrak{e}_7$ & $\mathfrak{f}_4$ &$\iota, \lambda$&$F_4 \times \{1, -1 \}$
\\[1pt]
\hline\hline
{}&&&&\\[-6pt]
EVIII-VIII-VIII & $\mathfrak{e}_8$ & $\mathfrak{so}(8)\oplus \mathfrak{so}(8)$ &$\sigma, \sigma'$&$(S\!pin(8) \times S\!pin(8))/(\Z_2 \times \Z_2)$ 
\\[2pt]
\hline
{}&&&&\\[-6pt]
EVIII-VIII-IX & $\mathfrak{e}_8$ & $\mathfrak{su}(8)\oplus \i\R$ &$\lambda_{\omega}\gamma, \lambda_{\omega} \gamma\upsilon$&$(S\!O(2) \times S\!U(8))/\Z_4 \rtimes \{1, \rho_\upsilon \}$
\\[2pt]
\hline
{}&&&&\\[-6pt]
EVIII-IX-IX & $\mathfrak{e}_8$ & $\mathfrak{so}(12)\oplus \mathfrak{sp}(1) \oplus \mathfrak{sp}(1)$ &$\sigma, \upsilon$&$(S\!U(2) \times S\!U(2) \times S\!pin(12))/\Z_4$
\\[2pt]
\hline
{}&&&&\\[-6pt]
EIX-IX-IX & $\mathfrak{e}_8$ & $\mathfrak{e}_6 \oplus \i\R \oplus \i\R$ &$\upsilon, \iota_{\omega}$&$(S\!O(2) \times U(1) \times E_6)/\Z_6 \rtimes \{1, \nu \}$ 
\\[2pt]
\hline
\end{tabular}
\end{small}
\vspace{1mm}

Table 1.\label{table1} Globally exceptional $\mathbb{Z}_2 \times \mathbb{Z}_2$-symmetric spaces 
\end{center}
\vspace{1mm}



\noindent {\bf Remark.}\,\,In the forth column, we omit a sign $\sim$, for example $\tilde{\gamma}$ is denoted by $\gamma$. In the fifth column 
, a sign $\rtimes$ means semi-direct product of groups, for example $(U(1) \times S\!O(6))/\Z_2 \rtimes \{1,\gamma_{\scriptscriptstyle {H}}\}$. 




\section{Preliminaries} 
  
We give the definitions of the simply connected compact exceptional Lie groups used in this article, and we state general notes for notation.
\subsection{Cayley algebra and compact Lie group of type $G_2$} 

Let $\mathfrak{C}=\{e_0 =1, e_1, e_2, e_3, e_4, e_5, e_6, e_7 \}_{\sR}$ be the division Cayley algebra. In $\mathfrak{C}$, since the multiplication and the inner product are well known, these are omitted.
\vspace{1mm}

The simply connected compact Lie group of type $G_2$ is given by
$$
   G_2 =\{\alpha \in \Iso_{\sR}(\mathfrak{C})\,|\, \alpha(xy)=(\alpha x) (\alpha y) \}\vspace{1mm}.
$$ 
\subsection{Exceptional Jordan algebra and compact Lie group of type $F_4$} 

Let  
$\mathfrak{J}(3,\mathfrak{C} ) = \{ X \in M(3, \mathfrak{C}) \, | \, X^* = X \}$ be the 
exceptional Jordan algebra. 
In $\mathfrak{J}(3,\mathfrak{C} )$, the Jordan multiplication $X \circ Y$, the 
inner product $(X,Y)$ and a cross multiplication $X \times Y$, called the Freudenthal multiplication, are defined by
$$
\begin{array}{c}
      X \circ Y = \dfrac{1}{2}(XY + YX), \quad (X,Y) = \tr(X \circ Y),
\vspace{1mm}\\
       X \times Y = \dfrac{1}{2}(2X \circ Y-\tr(X)Y - \tr(Y)X + (\tr(X)\tr(Y) 
- (X, Y))E), 
\end{array}$$
respectively, where $E$ is the $3 \times 3$ unit matrix. Moreover, we define the trilinear form $(X, Y, Z)$, the determinant $\det \,X$ by
$$
(X, Y, Z)=(X, Y \times Z),\quad \det \,X=\dfrac{1}{3}(X, X, X),
$$
respectively, and briefly denote $\mathfrak{J}(3, \mathfrak{C})$
by $\mathfrak{J}$.
\vspace{1mm}

The simply connected compact Lie group of type $F_4$ is given by
\begin{eqnarray*}
   F_4 \!\!\! &=& \!\!\! \{\alpha \in \Iso_{\sR}(\mathfrak{J}) \, | \, \alpha(X \circ Y) = \alpha X \circ \alpha Y \}
\\[1mm]
       &=&\!\!\!  \{\alpha \in \Iso_{\sR}(\mathfrak{J}) \, | \, \alpha(X \times Y) = \alpha X \times \alpha Y \}. 
\end{eqnarray*}
Then we have naturally the inclusion $G_2 \subset F_4$. 
\subsection{Complex exceptional Jordan algebra and Compact Lie group of type $E_6$} 
Let $\mathfrak{J}(3,\mathfrak{C})^C = \{ X \in M(3, \mathfrak{C})^C \, | \, X^* = X \}$ be the complexification of the exceptional Jordan algebra $\mathfrak{J}$. In $\mathfrak{J}(3,\mathfrak{C})^C$, as in $\mathfrak{J}$, we can also define the multiplication $X \circ Y, X \times Y$, the inner product $(X, Y)$, the trilinear forms $(X, Y, Z)$ and the determinant $\det \, X$ in the same manner, and those have the same properties. The  $\mathfrak{J}(3,\mathfrak{C} )^C$ is called the complex exceptional Jordan algebra, and briefly denote $\mathfrak{J}(3, \mathfrak{C})^C$ by $\mathfrak{J}^C$. 
\vspace{1mm}

The simply connected compact Lie group of type $E_6$ is given by
$$
E_6 = \{\alpha \in \Iso_C(\mathfrak{J}^C) \, | \, \det\, \alpha X = \det\, X, \langle \alpha X, \alpha Y \rangle = \langle X, Y \rangle \},  
$$
where the Hermite inner product $\langle X, Y \rangle$ is defined by $(\tau X, Y)$\,($\tau$ is a complex conjugation in $\mathfrak{J}^C$: $\tau(X+iY)=X-iY, \,X, Y \in \mathfrak{J}$).

\noindent Then we have naturally the inclusion $G_2 \subset F_4 \subset E_6$.
\subsection{$C$-vector space and compact Lie group of type $E_7$}

We define a $C$-vector space $\mathfrak{P}^C$, called the Freudenthal $C$-vector space, 
by
$$
          \mathfrak{P}^C = \mathfrak{J}^C \oplus \mathfrak{J}^C \oplus C \oplus C 
$$
with the Hermite inner product
$$
      \langle P, Q \rangle = \langle X, Z \rangle + \langle Y, W \rangle + 
(\tau\xi)\zeta + (\tau\eta)\omega
$$
for $P = (X, Y, \xi, \eta), \, Q = (Z, W, \zeta, \omega) \in \mathfrak{P}^C$. For $\phi 
\in {\mathfrak{e}_6}^C$, $A, B \in \mathfrak{J}^C$ and $\nu \in C$, we define a $C$-linear mapping  $\varPhi(\phi, A, B, \nu) : \mathfrak{P}^C \to \mathfrak{P}^C$ by
$$
\varPhi(\phi, A, B, \nu) \begin{pmatrix}X \vspace{1.5mm}\cr
                                Y \vspace{1.5mm}\cr
                                \xi \vspace{.5mm}\cr
                                \eta
                              \end{pmatrix}
                                =  \begin{pmatrix}\phi X - \dfrac{1}{3}\nu X + 2B 
\times Y + \eta A \vspace{-0.5mm}\cr
                                2A \times X - {}^t\phi Y + \dfrac{1}{3}\nu Y + 
\xi B \vspace{1mm}\cr
                                (A, Y) + \nu\xi \vspace{1.5mm}\cr
                                (B, X) - \nu\eta
                                    \end{pmatrix},
$$
where ${}^t\phi \in {\mathfrak{e}_6}^C$ is the transpose of $\phi$ with respect to the 
inner product $(X, Y)$: $({}^t\phi X, Y) = (X, \phi Y)$ (${\mathfrak{e}_6}^C$ is the complex Lie algebra of type $E_6$).

\noindent Moreover, for $P = (X, Y, 
\xi, \eta), \, Q = (Z, W, \zeta, \omega) \in \mathfrak{P}^C$, we define a $C$-linear 
mapping $P \times Q: \mathfrak{P}^C \to \mathfrak{P}^C$\vspace{-1mm} by 
$$
P \times Q = \varPhi(\phi, A, B, \nu), \quad
\left\{
\begin{array}{l}
    \, \phi = - \dfrac{1}{2}(X \vee W + Z \vee Y)
\vspace{1mm}\\
     A =  - \dfrac{1}{4}(2Y \times W - \xi Z - \zeta X)
\vspace{1mm}\\
     B = \;\;\, \dfrac{1}{4}(2X \times Z - \eta W - \omega Y)
\vspace{1mm}\\
   \, \nu = \;\;\, \dfrac{1}{8}((X, W) + (Z, Y) - 3(\xi \omega + \zeta \eta)),
\end{array}
\right.$$
where $X \vee W \in {\mathfrak{e}_6}^C$ is defined by $(X \vee W)U = \dfrac{1}{2}(W, U)X$ 
+ $\dfrac{1}{6}(X, W)U - 2W \times (X \times U)$ \,for $U \in \mathfrak{J}^C$.
\vspace{1mm}

The simply connected compact Lie group of type $E_7$ is given by 
$$
      E_7 = \{\, \alpha \in \Iso_C(\mathfrak{P}^C) \,|\, \alpha (P \times Q) 
\alpha^{-1} = \alpha P \times \alpha Q, \langle \alpha P, \alpha Q \rangle = 
\langle P, Q \rangle \}.
$$
Then we have naturally the inclusion $G_2 \subset F_4 \subset E_6 \subset E_7$.
\subsection {$C$-vector space and compact Lie group of type $E_8$}

We define a $C$-vector space ${\mathfrak{e}_8}^C$ by      
$$
    {\mathfrak{e}_8}^C = {\mathfrak{e}_7}^C \oplus \mathfrak{P}^C \oplus \mathfrak{P}^C \oplus C \oplus C \oplus C, 
$$
with the Lie bracket $[R_1, R_2]$, $R_k=(\varPhi_k, P_k, Q_k, r_k, s_k, t_k), k=1,2$, defined by
$$
  [(\varPhi_1, P_1, Q_1, r_1, s_1, t_1), (\varPhi_2, P_2, Q_2, r_2, s_2, t_2)] = 
  (\varPhi, P, Q, r, s, t), $$
$$
\left\{\begin{array}{l}
     \varPhi = {[}\varPhi_1, \varPhi_2] + P_1 \times Q_2 - P_2 \times Q_1
\vspace{1mm} \\
     Q = \varPhi_1P_2 - \varPhi_2P_1 + r_1P_2 - r_2P_1 + s_1Q_2 - s_2Q_1 
\vspace{1mm} \\
     P = \varPhi_1Q_2 - \varPhi_2Q_1 - r_1Q_2 + r_2Q_1 + t_1P_2 - t_2P_1
\vspace{1mm} \\
     r = - \dfrac{1}{8}\{P_1, Q_2\} + \dfrac{1}{8}\{P_2, Q_1\} + s_1t_2 - s_2t_1\vspace{1mm} \\
     s = \,\,\, \dfrac{1}{4}\{P_1, P_2\} + 2r_1s_2 - 2r_2s_1
\vspace{1mm} \\
     t = - \dfrac{1}{4}\{Q_1, Q_2\} - 2r_1t_2 + 2r_2t_1, 
\end{array} \right. 
$$
where ${\mathfrak{e}_7}^C$ is the complex Lie algebra of type $E_7$, $\{P, Q \} = (X, W) -(Y, Z) + \xi\omega - \eta\zeta, P = (X, Y, \xi, \eta),\,Q = (Z, W, \zeta, \omega) \in \vspace{0.5mm}\mathfrak{P}^C$. 
Then ${\mathfrak{e}_8}^C$ becomes the complex simple Lie algebra of type $E_8$.
\vspace{1mm}

Here, we define a $C$-linear transformation $\lambda_{\omega}$ of ${\mathfrak{e}_8}^C$
by
$$
       \lambda_{\omega}(\varPhi, P, Q, r, s, t) = (\lambda\varPhi\lambda^{-1}, \lambda Q, - \lambda P, -r, -t, -s), 
$$
where $\lambda$ of the right hand side is the $C$-linear transformation of $\mathfrak{P}^C$ and is defined in Section 3.4.

\noindent Moreover, the complex conjugation in ${\mathfrak{e}_8}^C$ is denoted by $\tau$:     
$$
    \tau(\varPhi, P, Q, r, s, t) = (\tau\varPhi\tau, \tau P, \tau Q, \tau r, \tau s, \tau t),$$
where $\tau$ in the right hand side is the usual complex conjugation in the complexification.
 
\noindent Then we define a Hermitian inner product $\langle R_1, R_2 \rangle$ in ${\mathfrak{e}_8}^C$ by 
$$
    \langle R_1, R_2 \rangle =-\frac{1}{15}B_8(\tau \lambda_{\omega} R_1, R_2),     
$$
where $B_8$ is the Killing form of ${\mathfrak{e}_8}^C$ (as for $B_8$, see \cite [Section $E_8$]{Yokotaichiro0} in detail).
 
The simply connected compact Lie group $E_8$ are given by 
$$
   E_8 = \{\alpha \in \Iso_C({\mathfrak{e}_8}^C) \,|\, \alpha[R_1, R_2] = [\alpha R_1, \alpha R_2], \langle \alpha R_1, \alpha R_2 \rangle =\langle R_1, R_2 \rangle \}. 
$$

\noindent Then we have naturally the inclusion $G_2 \subset F_4 \subset E_6 \subset E_7 \subset E_8$.
\vspace{1mm}

Now, we state general notes of this article for notation.
Let $G$ be a group.
For $\delta \in G$, $\tilde{\delta}$ denotes the inner automorphism induced by $\delta$: $\tilde{\delta}(g)=\delta g \delta^{-1}, g \in G$, then $G^{\tilde{\delta}}=\{g \in G \,|\, \tilde{\delta}(g)=g \}$. Hereafter $G^{\tilde{\delta}}$ 
will be also written by $G^\delta$. 
For $\alpha, \beta \in G$, when $\alpha$ and $\beta$ are conjugate in $G$, we denote it by $\alpha \sim \beta$. Besides, we almost use the same notations as \cite{Yokotaichiro0}.


\section{Globally exceptional symmetric spaces of type I}

In Table 2 below, the list of left half is classification of exceptional symmetric spaces that was found by \'{E}lie Cartan, on the other hand the list of right half is the results of group realizations corresponding to those. The structures of the groups $G^\varrho$ below are well-known fact, however the explicit forms of involutive inner automorphisms $\varrho$ are seldom known fact, so we write all in the following Table 2. The definitions of  $\varrho$ are written in the each section of this chapter. We remark that as in Table 1 we omit a sign $\sim$ in the fifth column.

\vspace{2mm}

\begin{center}
\begin{tabular}[t]{|l|c|l|l|c|l|}
\hline
{}&&&&&\\[-6pt]
Type & $\mathfrak{g}$ & \hspace{4mm}$\mathfrak{k} (\cong \mathfrak{g}^\varrho$) &$G$ & Involution $\varrho$ & \hspace{9mm}$K=G^\varrho$
\\[2pt]
\hline \hline 
{}&&&&&\\[-7pt]
G & $\mathfrak{g}_2$ & $\mathfrak{sp}(1) \oplus \mathfrak{sp}(1)$& $G_2$&$\gamma$& $(S\!p(1) \times S\!p(1))/\Z_2$
\\[1pt]
\hline\hline
{}&&&&&\\[-7pt]
FI & $\mathfrak{f}_4 $ & $\mathfrak{sp}(3) \oplus \mathfrak{sp}(1)$& $F_4$& $\gamma$&$(S\!p(1) \times S\!p(3))/\Z_2$
\\[1pt]
\hline
{}&&&&&\\[-7pt]
FII & $\mathfrak{f}_4 $ & $\mathfrak{so}(9)$& $F_4$& $\sigma$&$S\!pin(9)$
\\[1pt]
\hline\hline
{}&&&&&\\[-7pt]
EI & $\mathfrak{e}_6$ & $\mathfrak{sp}(4)$ & $E_6$& $\lambda\gamma$&$S\!p(4)/\Z_2$
\\[1pt]
\hline
{}&&&&&\\[-7pt]
EII & $\mathfrak{e}_6$ & $\mathfrak{su}(6) \oplus \mathfrak{sp}(1)$& $E_6$& $\gamma$&$(S\!p(1) \times S\!U(6))/\Z_2$
\\[1pt]
\hline
{}&&&&&\\[-7pt]
EIII & $\mathfrak{e}_6$ & $\mathfrak{so}(10) \oplus i\R$& $E_6$& $\sigma$&$(U(1) \times S\!pin(10))/\Z_4$
\\[1pt]
\hline
{}&&&&&\\[-7pt]
EIV & $\mathfrak{e}_6$ & $\mathfrak{f}_4$& $E_6$& $\lambda$&$F_4$
\\[1pt]
\hline 
\end{tabular}

\begin{tabular}[t]{|l|c|l|l|c|l|}
\hline
{}&&&&&\\[-6pt]
Type & $\mathfrak{g}$ & \hspace{4mm}$\mathfrak{k} (\cong \mathfrak{g}^\varrho$) &$G$ & Involution $\varrho$ & \hspace{9mm}$K=G^\varrho$
\\[2pt]
\hline \hline
{}&&&&&\\[-7pt]
EV & $\mathfrak{e}_7$ & $\mathfrak{su}(8)$& $E_7 $& $\lambda\gamma$&$S\!U(8)/\Z_2$
\\[1pt]
\hline
{}&&&&&\\[-7pt]
EVI & $\mathfrak{e}_7$ & $\mathfrak{so}(12) \oplus \mathfrak{sp}(1)$& $E_7$& $\gamma$&$(S\!p(1) \times S\!pin(12))/\Z_2$
\\[1pt]
\hline
{}&&&&&\\[-7pt]
EVII & $\mathfrak{e}_7$ & $\mathfrak{e}_6 \oplus i\R$& $E_7$& $\iota$&$(U(1) \times E_6 )/\Z_3$
\\[1pt]
\hline\hline 
{}&&&&&\\[-7pt]
EVIII & $\mathfrak{e}_8$ & $\mathfrak{so}(16)$ & $E_8$& $\lambda_{\omega}\gamma$&$S\!s(16)$
\\[1pt]
\hline
{}&&&&&\\[-7pt]
EIX & $\mathfrak{e}_8$ & $\mathfrak{e}_7 \oplus \mathfrak{sp}(1)$& $E_8$& $\upsilon$&$(S\!p(1) \times E_7)/\Z_2$
\\[1pt]
\hline
\end{tabular}
\vspace{3mm}

Table 2.\label{Table 2} Globally exceptional symmetric spaces of type I
\end{center}
\vspace{2mm}

In this chapter, each proof of the theorem is based on \cite{Yokotaichiro0}, and so whereas we omit the
detail, refer to the each section in \cite{Yokotaichiro0}, which is written in every proof, for details.


\subsection{Type G }

Let $\mathfrak{C}=\H \oplus \H e_4$ be Cayley devision algebra, where $\H$ is the field of quaternion number and $e_4$ is one of basis in $\mathfrak{C}$. 
\vspace{2mm}


We define an $\R$-linear transformation $\gamma$ of $\mathfrak{C}$ by 
$$
 \gamma(a+be_4)=a-be_4, \,\, a+be_4 \in \H \oplus \H e_4 = \mathfrak{C}.
$$
Then we have  $\gamma \in G_2, \,\, \gamma^2 =1$. Hence $\gamma$ induces involutive inner  automorphism $\tilde{\gamma}$ of $G_2: \tilde{\gamma}(\alpha)=\gamma\alpha\gamma, \alpha \in G_2$. 
\vspace{2mm}

Now, the structure of the group $(G_2)^\gamma$ is as follows. 

\begin{thm}\label{G}{\bf[G]}\,\,\, The group $(G_2)^\gamma$ is isomorphism to the group $(S\!p(1) \times S\!p(1))/\Z_2${\rm:}  $(G_2)^\gamma \cong  (S\!p(1) \times S\!p(1))/\Z_2, $ $ \Z_2=\{ (1,1), (-1,-1) \}$.
\end{thm}
\begin{proof}We define a mapping $\varphi_{{}_{\text{G}}}: S\!p(1) \times S\!p(1) \to (G_2)^\gamma$ by 
$$
   \varphi_{{}_{\text{G}}}(p, q)(m+a e_4)=qm \ov{q}+(pa \ov{q}) e_4, \,\,\,m+a e_4 \in \H \oplus \H e_4 =\mathfrak{C}.
$$
This mapping induces the required isomorphism (see \cite [Section 1.10]{Yokotaichiro0}).
\end{proof}

\subsection{Types FI and FII}

Let $\mathfrak{J}$ be the exceptional Jordan algebra. An element $X \in \mathfrak{J}$ has the form 
$$
X = \begin{pmatrix}
                   \xi_1 & x_3 & \ov{x}_2 \\
                   \ov{x}_3 & \xi_2 & x_1 \\ 
                   x_2 & \ov{x}_1 & \xi_3
\end{pmatrix},\,\, \xi_k \in \R,\, x_k \in \mathfrak{C}, k=1, 2, 3.
$$
Hereafter, in $\mathfrak{J}$, we use the following nations:
$$
\begin{array}{ccc}
E_1 = \left(\begin{array}{ccc}
              1 & 0 & 0 \\
              0 & 0 & 0 \\
              0 & 0 & 0
             \end{array}
             \right), & \,\,\,
E_2 = \left(\begin{array}{ccc}
              0 & 0 & 0 \\
              0 & 1 & 0 \\
              0 & 0 & 0
             \end{array}
             \right), & \,\,\,
E_3 = \left(\begin{array}{ccc}
              0 & 0 & 0 \\
              0 & 0 & 0 \\
              0 & 0 & 1
             \end{array}
             \right),    
\vspace{1mm}\\
F_1 (x) = \left(\begin{array}{ccc}
              0 &      0 & 0 \\
              0 &      0 & x \\
              0 & \ov{x} & 0
             \end{array}
             \right), & \,\,\,
F_2(x) = \left(\begin{array}{ccc}
              0 & 0 & \ov{x} \\
              0 & 0 & 0 \\
              x & 0 & 0
             \end{array}
             \right), & \,\,\,
F_3 (x) = \left(\begin{array}{ccc}
              0 & x & 0 \\
         \ov{x} & 0 & 0 \\
              0 & 0 & 0
             \end{array}
             \right).
\end{array}
$$
\indent We correspond such $X \in \mathfrak{J}$ to an element $M +\a \in \mathfrak{J}(3, \H) \oplus \H^3$ such that  
$$
     \begin{pmatrix}
                   \xi_1 & m_3 & \ov{m}_2 \\
                   \ov{m}_3 & \xi_2 & m_1 \\ 
                   m_2 & \ov{m}_1 & \xi_3
\end{pmatrix} 
        + (\a_1, \a_2, \a_3),
$$
where $x_k = m_k + a_ke_4 \in \H \oplus \H e_4 = \gC, k=1,2,3$. Then $\mathfrak{J}(3, \H) \oplus \H^3$ has the Freudenthal multiplication and the inner product 
$$
\begin{array}{c}
    (M + \a) \times (N + \b) = \Big(M \times N - \dfrac{1}{2} (\a^*\b + \b^*\a)\Big) - \dfrac{1}{2}(\a N + \b M), 
\vspace{1mm}\\
        (M + \a, N + \b) = (M, N) + 2(\a, \b), 
\end{array}
$$
where $(\a, \b) = \dfrac{1}{2}(\a\b^* + \b\a^*)$, corresponding those of $\mathfrak{J}$, that is, $\mathfrak{J}$ is isomorphic to $\mathfrak{J}(3, \H) $ $\oplus \,\H^3$ as algebra. From now on, we identify $\mathfrak{J}$ with $\mathfrak{J}(3, \H) \oplus \H^3$.
\vspace{2mm}


We define $\R$-linear transformations $\gamma, \sigma$ of $\mathfrak{J}$ by
$$
\gamma X
        = \begin{pmatrix} \xi_1 & \gamma x_3 & \ov{\gamma x_2} \\
                         \ov{\gamma x_3} & \xi_2 & \gamma x_1 \\
                          \gamma x_2 & \ov{\gamma x_1} & \xi_3   \end{pmatrix}
,\,\,\,\,\sigma X
        = \begin{pmatrix} \xi_1 & -x_3 & -\ov{x}_2 \\
                         -\ov{x}_3 & \xi_2 & x_1 \\
                          -x_2 & \ov{x}_1 & \xi_3   \end{pmatrix},\,\,X \in \mathfrak{J},
$$
respectively, where $\gamma$ of right hand side is the same one as $\gamma \in G_2$. Then we have that $\gamma, \sigma \in F_4, \gamma^2=\sigma^2 =1$. Hence $\gamma, \sigma$ induce involutive inner automorphisms $\tilde{\gamma}, \tilde{\sigma}$ of $F_4{\rm :}\,\tilde{\gamma}(\alpha)=\gamma\alpha\gamma, \tilde{\sigma}(\alpha)=\sigma\alpha\sigma, \alpha \in F_4$.
\vspace{2mm}

Now, the structures of groups $(F_4)^\gamma$ and $(F_4)^\sigma$ are as follows.

\begin{thm}\label{FI}{\bf[FI]}\,\,\,The group $(F_4)^\gamma$ is isomorphic to the group $(S\!p(1) \times S\!p(3))/\Z_2${\rm:}  $(F_4)^\gamma \cong (S\!p(1) \times S\!p(3))/\Z_2, \,$ $\Z_2 =\{(1, E), (-1, -E) \}$.
\end{thm}
\begin{proof}
We define a mapping $\varphi_{{}_{\text{F1}}}: S\!p(1) \times S\!p(3) \to (F_4)^\gamma$ by
$$
\varphi_{{}_{\text{F1}}}(p, A)(M+\a)=AMA^* +p\a A^*,\,\,\, M+\a \in \mathfrak{J}(3, \H) \oplus \H^3=\mathfrak{J}.
$$
This mapping induces the required isomorphism (\cite[Section 2.11]{Yokotaichiro0}).
\end{proof}








\begin{thm}\label{FII}{\bf[FII]}\,\,\,The group $(F_4)^\sigma$ is isomorphic to the group $S\!pin(9)${\rm:}$(F_4)^\sigma \!\cong \!S\!pin(9)$.
\end{thm}
\begin{proof}
From \cite[Thorem 2.7.4]{Yokotaichiro0}
, we have $(F_4)_{E_1} \cong S\!pin(9)$, so by proving that $(F_4)^\sigma \cong (F_4)_{E_1}$ (\cite[Thorem 2.9.1]{Yokotaichiro0}) we have the required isomorphism (see \cite[Sections 2.7, 2.9 ]{Yokotaichiro0} in detail).
\end{proof}
\vspace{1mm}

\subsection{Types EI, EII, EIII and EIV}
Let $\mathfrak{J}^C$ be the complex exceptional Jordan algebra. 
The complex conjugation $\tau$ of $\mathfrak{J}^C$ satisfies the equalities: 
$$
\tau(X \circ Y)=\tau X \circ \tau Y,\,\,\tau(X \times Y)=\tau X \times \tau Y, \,\,X, Y \in \mathfrak{J}^C.
$$ 
\indent Here, we define an involutive automorphism $\lambda$ of $E_6$ by 
$$
\lambda(\alpha)={}^t \alpha ^{-1},\, \alpha \in E_6.
$$
Then, from the definition of transpose: $({}^t \alpha X, Y)=(X, \alpha Y)$, we see that 
$$
\lambda(\alpha)=\tau\alpha\tau,\, \alpha \in E_6.
$$
 
Let the $C$-linear transformations $\gamma, \sigma$ of $\mathfrak{J}^C$ be the complexification of $\gamma \in G_2 \subset F_4, \sigma \in F_4$.
Then we have that $\gamma, \sigma \in E_6, \gamma^2=\sigma^2=1$. Hence, as in $F_4$, the group $E_6$ has involutive inner automorphisms $\tilde{\gamma}, \tilde{\sigma}$ induced by $\gamma, \sigma$: $\tilde{\gamma}(\alpha)=\gamma\alpha\gamma, \tilde{\sigma}(\alpha)=\sigma\alpha\sigma, \alpha \in E_6$. 
\vspace{2mm}

Now, the structures of the groups $(E_6)^{\lambda\gamma}, (E_6)^{\gamma}, (E_6)^{\sigma}$ and $(E_6)^{\lambda}$ are as follows.

\begin{thm}\label{EI}{\bf[EI]}\,\,\,The group $(E_6)^{\lambda\gamma}$ is isomorphic to the group $S\!p(4)/\Z_2${\rm:}\,$(E_6)^{\lambda\gamma}\cong (E_6)^{\tau\gamma} \cong S\!p(4)/\Z_2, $ $\Z_2 =\{E, -E \}$.
\end{thm}
\begin{proof}
We define a mapping $\varphi_{{}_{\text{E1}}}:S\!p(4) \to (E_6)^{\tau\gamma}$ by 
$$
\varphi_{{}_{\text{E1}}}(P)X=g^{-1}(P(gX)P^*),\,\, X \in \mathfrak{J}^C,
$$ 
where $g: \mathfrak{J}^C \to \mathfrak{J}{(4, \H)^C}_{\!0}$ is the $C$-linear isomorphism. 
This mapping induces the required isomorphism (see \cite[Section 3.12 ]{Yokotaichiro0}). 
\end{proof}
\noindent {\bf Remark.} From $\lambda(\alpha)=\tau\alpha\tau$ and $\tau\gamma=\gamma\tau$, we see that $\lambda(\gamma\alpha\gamma)=\tau(\gamma\alpha\gamma)\tau, \, \alpha \in E_6$. 

\begin{thm}\label{EII}{\bf[EII]}\,\,The group $(E_6)^\gamma$ isomorphic to the group $(S\!p(1) \times S\!U(6))/\Z_2${\rm:}
$(E_6)^\gamma \!\cong \!(S\!p(1) \times S\!U(6))/\Z_2,$ $ \Z_2 =\{(1, E), (-1, -E) \}$, where $S\!U(6)\!=\{U \in M(6, C)\,|\,(\tau\,{}^t U) U$ $=1, \det U=1) \}$.
\end{thm}
\begin{proof}
We define a mapping $\varphi_{{}_{\text{E2}}}:S\!p(1) \times S\!U(6) \to (E_6)^\gamma $ by
$$
\varphi_{{}_{\text{E2}}}(p, U)(M+\a)={k_J}^{-1}(U(k_J M){}^t U)+p\a k^{-1}(\tau \,{}^t U), M+\a \in \mathfrak{J}(3, \H)^C \oplus (\H^3)^C=\mathfrak{J}^C,
$$
where both of $k_J:\mathfrak{J}(3, \H)^C \to \mathfrak{S}(6, C)$ and $k:M(3, \H)^C \to M(6, C)$ are the $C$-linear isomorphisms.
This mapping induces the required isomorphism (see \cite[Section 3.11 ]{Yokotaichiro0}).
\end{proof}

\begin{thm}\label{EIII}{\bf[EIII]}\,\,\,The group $(E_6)^\sigma$ is isomorphic to the group $(U(1) \times S\!pin(10))/\Z_4${\rm:}\,
$(E_6)^\sigma \!\cong (U(1) \times S\!pin(10))/\Z_4,$ $ \Z_4=\{ (1, \phi(1)), (-1, \phi(-1)), (i, \phi(-i)), (-i, \phi(i)) \}$.
\end{thm}
\begin{proof}
We define a mapping $\varphi_{{}_{\text{E3}}}:U(1) \times S\!pin(10) \to (E_6)^\sigma $ by
$$
\varphi_{{}_{\text{E3}}}(\theta, \delta)=\phi_1(\theta)\delta,
$$
where $\phi_1(\theta)\!:\! \mathfrak{J}^C \!\to \! \mathfrak{J}^C$ is the $C$-linear mapping.
This mapping induces the required isomorphism (see \cite[Section 3.10 ]{Yokotaichiro0}).
\end{proof}

\begin{thm}\label{EIV}{\bf[EIV]}\,\,\,The group $(E_6)^\lambda$ is isomorphic to the group $F_4${\rm:}$(E_6)^\lambda \!=\!(E_6)^\tau \!\cong F_4$.
\end{thm}
\begin{proof}
From the explanation at the beginning of this section, we have $(E_6)^\lambda \cong (E_6)^\tau$, so by proving $(E_6)^\tau \cong F_4$ we have the required isomorphism (see \cite[Section 3.7 ]{Yokotaichiro0}).
\end{proof}
\vspace{1mm}

\subsection{Types EV, EVI and  EVII}
Let $\mathfrak{P}^C$ be the Freudenthal $C$-vector space. We define $C$-linear transformations  $\lambda, \gamma, \sigma$ and $\iota$ of $\mathfrak{P}^C$ by 
\begin{eqnarray*}
\gamma(X, Y, \xi, \eta)\!\!\!&=&\!\!\!(\gamma X, \gamma Y, \xi, \eta),
\\[1mm]
\sigma(X, Y, \xi, \eta)\!\!\!&=&\!\!\!(\sigma X, \sigma Y, \xi, \eta),
\\[1mm]
\lambda(X, Y, \xi, \eta)\!\!\!&=&\!\!\!(Y, -X, \eta, -\xi),
\\[1mm]
\iota(X, Y, \xi, \eta)\!\!\!&=&\!\!\!(-iX, iY, -i\xi, i\eta),\,(X, Y, \xi, \eta) \in \mathfrak{P}^C,
\end{eqnarray*}
where $ i \in C$ and $\gamma, \sigma$ of the right hand side are the same ones as $\gamma \in G_ 2 \subset F_4 \subset E_6, \sigma \in F_4 \subset E_6$. 
Then we have that $\gamma, \sigma, \lambda, \iota \in 
E_7$ and $\gamma^2=\sigma^2=1, \lambda^2=\iota^2=-1$. Hence, as 
in $E_6$, the group $E_7$ has involutive inner automorphisms 
$\tilde{\gamma}, \tilde{\sigma}$ induced by $\gamma, \sigma$:\,$\tilde{\gamma}(\alpha)=\gamma\alpha\gamma, \tilde{\sigma}(\alpha)=\sigma\alpha\sigma, \alpha \in E_7$. Moreover, since $ -1 \in 
z(E_7)$ (the center of $E_7$), $\lambda,\iota$ induce involutive inner automorphisms $\tilde{\lambda}, \tilde{\iota}$ of $E_7$:\,  $\tilde{\lambda}(\alpha)=\lambda\alpha\lambda^{-1},\tilde{\iota}(\alpha)=\iota\alpha\iota^{-1}, \alpha \in E_7$. 
\vspace{2mm}

Now, the structures of the groups $(E_7)^{\lambda\gamma}, (E_7)^{\gamma}$ and $(E_7)^{\iota}$ are as follows. 

\begin{thm}\label{EV}{\bf[EV]}\,\,\,The group $(E_7)^{\lambda\gamma}$ is isomorphic to the group $S\!U(8)/\Z_2${\rm:}$(E_7)^{\lambda\gamma}\cong   (E_7)^{\tau\gamma} \cong S\!U(8)/\Z_2, $ $ \Z_2 =\{E, -E \}$.
\end{thm}
\begin{proof}
We define a mapping $\varphi_{{}_{\text{E5}}}:S\!U(8) \to (E_7)^{\tau\gamma}$ by
$$
  \varphi_{{}_{\text{E5}}}(A)P=\chi^{-1}(A(\chi P){}^t A),\, P \in \mathfrak{P}^C,
$$
where $\chi\!:\!\mathfrak{P}^C \!\to \! \mathfrak{S}(8, \C)^C$ is the $C$-linear isomorphism.
This mapping induces the required isomorphism (see \cite[Section 4.12 ]{Yokotaichiro0}).
\end{proof}
\noindent {\bf Remark.} Since $\lambda\gamma$ is conjugate to $\tau\gamma$ in $E_6$, it is also in $E_7$. 

\begin{thm}\label{EVI}{\bf [EVI]}\,\,\,The group $(E_7)^{\gamma}$ is isomorphic to the group $(S\!U(2) \times S\!pin(12))/\Z_2${\rm:}\, 
$(E_7)^{\gamma} \cong (E_7)^{-\sigma} \!= (E_7)^{\sigma} \cong (S\!U(2) \times S\!pin(12))/\Z_2, \Z_2=\{(E,1), (-E, -\sigma) \}$.
\end{thm}
\begin{proof}
We define a mapping $\varphi_{{}_{\text{E6}}}:S\!U(2) \times S\!pin(12) \to (E_7)^{\sigma}$ by
$$
 \varphi_{{}_{\text{E6}}}(A, \beta)=\phi_2(A)\beta,    
$$
where $\phi_2(A)\!:\! \mathfrak{P}^C \!\! \to \!\! \mathfrak{P}^C$ is the $C$-linear mapping.
This mapping induces the required isomorphism (see \cite[Section 4.11]{Yokotaichiro0}).
\end{proof}
\noindent {\bf Remark.} As for the fact that  $\gamma$ is conjugate to $-\sigma$ in $E_7$, see \cite[Proposition 4.3.5 (3)]{Yokotaichiro2}.

\begin{thm}\label{EVII}{\bf[EVII]}\,\,\,The group $(E_7)^{\iota}$ is isomorphic to the group $(U(1) \times E_6)/\Z_3${\rm:}\, $(E_7)^{\iota} \cong (U(1) \times E_6)/\Z_3, $ $ \Z_3 =\{(1, 1), (\omega, \phi(\omega^2)), (\omega^2, \phi(\omega)) \}$, where $\omega \in C, \omega^3  =1, \omega \ne 1$.
\end{thm}
\begin{proof}
We define a mapping $\varphi_{{}_{\text{E7}}}:U(1) \times E_6 \to (E_7)^{\iota}$ by
$$
\varphi_{{}_{\text{E7}}}(\theta, \beta)=\phi(\theta)\beta,
$$
where $\phi(\theta)\!:\!\mathfrak{P}^C \! \to \! \mathfrak{P}^C$ is the $C$-linear mapping.
This mapping induces the required isomorphism (see \cite[Section 4.10 ]{Yokotaichiro0}).
\end{proof} 
\vspace{-2mm}

\subsection{Types EVIII and  EIX}

Let ${\mathfrak{e}_8}^C$ be $248$ dimensional $C$-vector space. We define  $C$-linear transformations $\lambda_{\omega}, \gamma$ and $\upsilon$ of ${\mathfrak{e}_8}^C$ by 
\begin{eqnarray*}
\lambda_{\omega}(\varPhi, P, Q, r, s, t)\!\!\!&=&\!\!\!(\lambda \varPhi\lambda^{-1}, \lambda Q, -\lambda P, -r, -t, -s ),
\\[1mm]
\gamma(\varPhi, P, Q, r, s, t)\!\!\!&=&\!\!\!(\gamma\varPhi\gamma, \gamma P, \gamma Q, r, s, t),
\\[1mm]
\upsilon(\varPhi, P, Q, r, s, t)\!\!\!&=&\!\!\!(\varPhi, -P, -Q, r, s, t), \,(\varPhi, P, Q, r, s, t) \in {\mathfrak{e}_8}^C,
\end{eqnarray*}
where $\lambda, \gamma$ of right hand side are same ones $\lambda \in E_7, \gamma \in G_2 \subset F_4 \subset E_6 \subset E_7$. Then we have that $\lambda_{\omega}, \gamma, \upsilon \in E_8$ and ${ \lambda_{\omega}}^2=\gamma^2=\upsilon^2=1$. Hence  $\lambda_{\omega}, \gamma, \upsilon$ induce involutive inner automorphisms ${\tilde{\lambda}}_{\omega}, \tilde{\gamma},  \tilde{\upsilon}$ of $E_8$: ${\tilde{\lambda}}_{\omega}(\alpha)=(\lambda_{\omega})\alpha( \lambda_{\omega}), \tilde{\gamma}(\alpha)=\gamma\alpha\gamma,  \tilde{\upsilon}(\alpha)=\upsilon\alpha\upsilon, \alpha \in E_8$. (Remark. $\lambda_{\omega}$ is nothing but $\lambda\lambda'$ defined in \cite[Section 5.5 ]{Yokotaichiro0}.)
\vspace{2mm}

Now, the structures of the groups $(E_8)^{\lambda_{\omega} \gamma}$ and $(E_8)^\upsilon$ are as follows.
\begin{thm}\label{EVIII}{\bf [EVIII]}\,\,\,The group $(E_8)^{\lambda_{\omega} \gamma}$ is isomorphic to the group $S\!s(16)$\,{\rm:}\, $(E_8)^{\lambda_{\omega} \gamma} \cong S\!s(16) $.
\end{thm}
\begin{proof}
Since the homomorphism between $(E_8)^{\lambda_{\omega} \gamma}$ and $S\!s(16)$ is not found in $E_8$ defined in Section 2.5 until now, we omit this proof (see \cite[Section 5.8 ]{Yokotaichiro0}).
\end{proof}

\begin{thm}\label{EIX}{\bf [EIX]}\,\,\,The group $(E_8)^\upsilon$ is isomorphic to the group $(S\!U(2) \times E_7)/\Z_2$\,{\rm:}\,
$(E_8)^\upsilon \cong (S\!U(2) \times E_7)/\Z_2, $  $ \Z_2 =\{(E, 1), (-E, -1) \}$.
\end{thm}
\begin{proof}
We define a mapping $\varphi_{{}_{\text{E9}}}:S\!U(2) \times E_7 \to (E_7)^{\upsilon}$ by 
$$
  \varphi_{{}_{\text{E9}}}(A, \beta)=\phi_3(A)\beta, 
$$
where $\phi_3 (A)\!:\! {\mathfrak{e}_8}^C \!\to \! {\mathfrak{e}_8}^C$ is the $C$-linear transformation.
This mapping induces the required isomorphism (see \cite[Section 5.7 ]{Yokotaichiro0}. $\phi_3$ is nothing but $\varphi_3$ defined in \cite[Theorem 5.7.4]{Yokotaichiro0}).
\end{proof}

\section {Globally exceptional $\mathbb{Z}_2 \times \mathbb{Z}_2$- symmetric spaces}

In this chapter, for $G=G_2, F_4, E_6, E_7$ or $E_8$,  
we determine the type $(G/G^\sigma, G/G^\tau,$ $ G/G^{\sigma\tau})$ of globally 
exceptional $\mathbb{Z}_2 \times \mathbb{Z}_2$-symmetric space and the 
structure of group $G^\sigma \cap G^\tau$ 
by giving a pair of involutive inner automorphisms $\tilde{\sigma}$ and $\tilde{\tau}$ of $G$. Most of fundamental $K$-linear transformations and involutive automorphisms used later are defined in previous chapter, where $K=\R , C$, and others are defined each times.

Even if some proofs of this chapter  are similar to ones of the preceding articles
\cite{M.Y.01},\cite{M.01}, \cite{M.02},
\cite{Miyashitatoshikazu},\cite{Yokotaichiro0}, 
\cite{Yokotaichiro1}, \cite{Yokotaichiro2} and \cite{Yokotaichiro3}, 
we rewrite in detail again as much as possible. As mentioned in Tables 1,2, we also omit a sing $\sim$ for the elements of $\mathbb{Z}_2 \times \mathbb{Z}_2$.
\vspace{3mm}

$\bullet$\,\,{\boldmath $[G_2]$}\,\,We study one type in here.
\subsection{Type G-G-G}
In this section, we give a pair of involutive inner automorphisms $\tilde{\gamma}$ and ${\tilde{\gamma}}_{\scriptscriptstyle {H}}$, where 
${\tilde{\gamma}}_{\scriptscriptstyle {H}}$ is induced by an $\R$-linear transformation $\gamma_{\scriptscriptstyle {H}}$ defined below.
 \vspace{1mm}
 
We define  $\R$-linear transformations $\gamma_{{}_{\scriptscriptstyle {H}}}, \gamma_{\scriptscriptstyle {C}}$ of $\H$ by
\begin{eqnarray*}
\gamma_{{}_{\scriptscriptstyle {H}}}(a+be_2)\!\!\!&=&\!\!\! a -b e_2,
\\
\gamma_{{}_{\scriptscriptstyle {C}}}(a+be_2)\!\!\!&=&\!\!\!\ov{a}+\ov{b} e_2,
\,a+be_2 \in \C \oplus \C e_2 =\mathfrak{\H}.
\end{eqnarray*}

\noindent Then $\gamma_{{}_{\scriptscriptstyle {H}}}, \gamma_{\scriptscriptstyle {C}}$ are naturally extended to  $\R$-linear transformations $\gamma_{\scriptscriptstyle {H}}, \gamma_{\scriptscriptstyle {C}}$ of $\mathfrak{C}$ as follows:
\begin{eqnarray*}
\gamma_{{}_{\scriptscriptstyle {H}}}(x+ye_4)\!\!\!&=&\!\!\!\gamma_{{}_{\scriptscriptstyle {H}}} x +(\gamma_{{}_{\scriptscriptstyle {H}}} y)e_4,
\\
\gamma_{{}_{\scriptscriptstyle {C}}}(x+ye_4)\!\!\!&=&\!\!\!\gamma_{{}_{\scriptscriptstyle{C}}} x +(\gamma_{{}_{\scriptscriptstyle {C}}} y)e_4,
\,x+ye_4 \in \H \oplus \H e_4 =\mathfrak{C}.
\end{eqnarray*}
Needless to say, we have that 
$\gamma_{{}_{\scriptscriptstyle {H}}},
\gamma_{{}_{\scriptscriptstyle {C}}} \!\in G_2, 
\gamma_{{}_{\scriptscriptstyle {H}}}^2\!
=\!\gamma_{{}_{\scriptscriptstyle {C}}}^2\!=\!1, 
\gamma\gamma_{{}_{\scriptscriptstyle 
{H}}}=\gamma_{{}_{\scriptscriptstyle {H}}} \gamma, 
\gamma\gamma_{{}_{\scriptscriptstyle {C}}}
=\gamma_{{}_{\scriptscriptstyle {C}}} \gamma$.
Hence 
$\gamma_{{}_{\scriptscriptstyle {H}}}, \gamma_{{}_{\scriptscriptstyle 
{C}}}$ induce involutive inner automorphisms 
${\tilde{\gamma}}_{{}_{\scriptscriptstyle {H}}}, 
{\tilde{\gamma}}_{{}_{\scriptscriptstyle {C}}}$ of $G_2$: 
${\tilde{\gamma}}_{{}_{\scriptscriptstyle 
{H}}}(\alpha)=\gamma_{{}_{\scriptscriptstyle 
{H}}}\alpha\gamma_{{}_{\scriptscriptstyle {H}}}, 
{\tilde{\gamma}}_{{}_{\scriptscriptstyle 
{C}}}(\alpha)=\gamma_{{}_{\scriptscriptstyle 
{C}}}\alpha\gamma_{{}_{\scriptscriptstyle {C}}}, \alpha \in G_2$.
\begin{lem}\label{lem:G} In $G_2$, we have the following facts.

{\rm (1)}\, $\gamma$ is conjugate to both of $\gamma_{{}_{\scriptscriptstyle {H}}}$ and $ \gamma \gamma_{{}_{\scriptscriptstyle {H}}}${\rm :}
$\gamma \sim \gamma_{{}_{\scriptscriptstyle {H}}}, \gamma \sim \gamma \gamma_{{}_{\scriptscriptstyle {H}}}$.

{\rm (2)}\, $\gamma$ is conjugate to both of $\gamma_{{}_{\scriptscriptstyle {C}}}$ and $ \gamma \gamma_{{}_{\scriptscriptstyle {C}}}${\rm :}
$\gamma \sim \gamma_{{}_{\scriptscriptstyle {C}}}, \gamma \sim \gamma \gamma_{{}_{\scriptscriptstyle {C}}}$.
\end{lem}
\begin{proof}
(1)\, We define $\R$-linear isomorphisms $\delta_1, \delta_2 : \mathfrak{C} \to \mathfrak{C}$ by 
\begin{eqnarray*}
&& \hspace{-5mm}\delta_1 : 1 \mapsto 1,\, e_1 \mapsto e_1,\, e_2 \mapsto e_4, e_3 \mapsto e_5,\,e_4 \mapsto e_2,\,e_5 \mapsto e_3, \, e_6 \mapsto -e_6, \, e_7 \mapsto -e_7,
\\
&&\hspace{-5mm} \delta_2 : 1 \mapsto 1,\, e_1 \mapsto e_1,\, e_2 \mapsto -e_6, e_3 \mapsto -e_7,\,e_4 \mapsto -e_4,\,e_5 \mapsto -e_5, \, e_6 \mapsto -e_2, \, e_7 \mapsto -e_3,
\end{eqnarray*}
where $1$ and $e_k, k=1,2,\dots, 7$ are the basis of $\mathfrak{C}$, respectively.
Then we see $\delta_1, \delta_2 \in G_2, {\delta_1}^2 = {\delta_2}^2 =1$. Hence, by straightforward computation, we have  $\delta_1 \gamma=\gamma_{{}_{\scriptscriptstyle {H}}} \delta_1, \delta_2 \gamma=(\gamma\gamma_{{}_{\scriptscriptstyle {H}}}) \delta_2$, that is, $\gamma \sim \gamma_{{}_{\scriptscriptstyle {H}}}, \gamma \sim \gamma \gamma_{{}_{\scriptscriptstyle {H}}}$ in $G_2$. 
\vspace{1mm}

(2)\,We define $\R$-linear transformations $
\delta_3, 
\delta_4: \mathfrak{C} \to \mathfrak{C}$ by 
\begin{eqnarray*}
 &&\vspace{-3mm}\delta_3:1 \mapsto 1,\, e_1 \mapsto e_4,\, e_2 \mapsto e_2, e_3 \mapsto e_6,\,e_4 \mapsto e_1,\,e_5 \mapsto -e_5, \, e_6 \mapsto e_3, \, e_7 \mapsto -e_7,
 \\
&&\vspace{-3mm}\delta_4:1 \mapsto 1,\, e_1 \mapsto e_5,\, e_2 \mapsto e_2, e_3 \mapsto -e_7,\,e_4 \mapsto -e_4,\,e_5 \mapsto e_1, \, e_6 \mapsto -e_6, \, e_7 \mapsto -e_3,
\end{eqnarray*}
respectively.
Then as in (1) above, we have that $\delta_3, \delta_4 \in G_2, {\delta_3}^2={\delta_4}^2=1, 
\delta_3 \gamma=\gamma_{\scriptscriptstyle {C}}\delta_3, \delta_4 \gamma=(\gamma\gamma_{\scriptscriptstyle {C}})\delta_4$, that is, $
\gamma \sim \gamma_{\scriptscriptstyle {C}}, \gamma \sim \gamma\gamma_{\scriptscriptstyle {C}}$ in $E_6$.
\end{proof}

We have the following proposition which is the direct result of Lemma 4.1.1.

\begin{prop}
The group $(G_2)^\gamma$ is isomorphic to both of the groups $(G_2)^{\gamma_{{}_{\scriptscriptstyle {H}}}}$ and $(G_2)^{\gamma \gamma_{{}_{\scriptscriptstyle {H}}}}${\rm :} $(G_2)^\gamma \cong (G_2)^{\gamma_{{}_{\scriptscriptstyle {H}}}} \cong (G_2)^{\gamma \gamma_{{}_{\scriptscriptstyle {H}}}}$.
\end{prop}

From the result of type G in Table 2 and Proposition 4.1.2 , we have the following theorem.


\begin{thm} For $\mathbb{Z}_2 \times \mathbb{Z}_2=\{1,\gamma \} \times \{1, \gamma_{{}_{\scriptscriptstyle {H}}} \}$, the $\mathbb{Z}_2 \times \mathbb{Z}_2$-symmetric space is of type $(G_2/(G_2)^\gamma, G_2/(G_2)^{\gamma_{{}_{\scriptscriptstyle {H}}}}, G_2/(G_2)^{\gamma \gamma_{{}_{\scriptscriptstyle {H}}}})=(G_2/(G_2)^\gamma, G_2/(G_2)^\gamma, G_2/(G_2)^\gamma)$ , that is, {\rm (G, G, G)}, abbreviated as {\rm G}.
\end{thm}
Here, we prove lemma needed and make some preparations for theorem below.


\begin{lem}
The mapping $\varphi_{{}_{\rm G}}: S\!p(1) \times S\!p(1) \to (G_2) ^\gamma$ of Theorem 3.1.1 satisfies the equalities {\rm :}
\vspace{-7mm}

\begin{eqnarray*}
&&{\rm (1)}\, \,
\gamma_{{}_{\scriptscriptstyle {H}}}=\varphi_{{}_{\rm G}}(e_1 , e_1),
\,\,
\gamma_{{}_{\scriptscriptstyle {C}}}=\varphi_{{}_{\rm G}}(e_2 , e_2).
\\[0mm]
&&{\rm (2)}\,\,
\gamma_{{}_{\scriptscriptstyle {H}}}\varphi_{{}_{\rm G}}(p, q)\gamma_{{}_{\scriptscriptstyle {H}}}=\varphi_{{}_{\rm G}}(\gamma_{{}_{\scriptscriptstyle {H}}}p, \gamma_{{}_{\scriptscriptstyle {H}}}q).
\end{eqnarray*}
\end{lem}
\begin{proof}
The proof of (1) is omitted (see \cite[Lemma 1.3.3]{Yokotaichiro1} in detail). The equality of (2) is the direct result of (1).
\end{proof}

\vspace{1mm}

Consider a group $\mathcal{Z}_2 = \{1, \gamma_{{}_{\scriptscriptstyle {C}}} \}$. Then the group $\mathcal{Z}_2 = \{1, \gamma_{{}_{\scriptscriptstyle {C}}} \}$ acts on the group $U(1) \times U(1)$ by
$$
  \gamma_{{}_{\scriptscriptstyle {C}}}(a, b)=(\gamma_{{}_{\scriptscriptstyle {C}}} a, \gamma_{{}_{\scriptscriptstyle {C}}} b)
$$ 
and let $(U(1) \times U(1)) \rtimes \mathcal{Z}_2$ be the semi-direct product of $U(1) \times U(1)$ and $\mathcal{Z}_2$ with this action. 

\vspace{2mm}

Now, we determine the structure of the group $(G_2)^\gamma \cap (G_2)^{\gamma_{{}_{\scriptscriptstyle {H}}}}$.

\begin{thm} We have that $(G_2)^\gamma \cap (G_2)^{\gamma_{{}_{{}_{\scriptscriptstyle {H}}}}}\cong (U(1) \times U(1))/\Z_2 \rtimes \mathcal{Z}_2,
\Z_2= \{(1, 1), $ $ (-1, -1) \}, \, \mathcal{Z}_2 = \{1, \gamma_{{}_{\scriptscriptstyle {C}}} \}$.
\end{thm}
\begin{proof}
We define a mapping $\varphi_{415}: (U(1) \times U(1)) \rtimes \{1, \gamma_{\scriptscriptstyle {C}} \} \to (G_2)^\gamma \cap (G_2)^{\gamma_{{}_{\scriptscriptstyle {H}}}}$ by
\begin{eqnarray*}
  \varphi_{415}(a, b, 1)\!\!\!&=&\!\!\!\varphi_{{}_\text{G}}(a,b),
\\
   \varphi_{415}(a, b, \gamma_{{}_{\scriptscriptstyle {C}}})\!\!\!&=&\!\!\!\varphi_{{}_\text{G}}(a,b)\, \gamma_{{}_{\scriptscriptstyle {C}}},
\end{eqnarray*}
where $\varphi_{\text{G}}$ is defined in Theorem 3.1.1. From $\gamma\gamma_{{}_{\scriptscriptstyle {C}}}=\gamma_{{}_{\scriptscriptstyle {C}}}\gamma, \gamma\gamma_{{}_{\scriptscriptstyle {H}}}=\gamma_{{}_{\scriptscriptstyle {H}}}\gamma$ and Lemma 4.1.4 (1)
, we have $\varphi_{415}(a, b, 1), \varphi_{415}(a, b, \gamma_{{}_{\scriptscriptstyle {C}}}) \in (G_2)^\gamma \cap (G_2)^{\gamma_{{}_{\scriptscriptstyle {H}}}}$. Hence $\varphi_{415}$ is well-defined. Using $(ae_2)c=(a\ov{c})e_2, a, c \in U(1)$, we can confirm that  $\varphi_{415}$ is a homomorphism. Indeed, we show that the case of $\varphi_{415}(a, b, \gamma_{{}_{\scriptscriptstyle {C}}})\, \varphi_{415}(c, d, 1)=\varphi_{415}((a, b, \gamma_{{}_{\scriptscriptstyle {C}}})(c,d,1))$ as example. For the left hand side of this equality, we have that
\begin{eqnarray*}
\varphi_{415}(a, b, \gamma_{{}_{\scriptscriptstyle {C}}})\, \varphi_{415}(c, d, 1)\!\!\!&=&\!\!\!(\varphi_{\text{G}}(a,b)\,\gamma_{\scriptscriptstyle {C}})\varphi_{\text{G}}(c,d)
\\
\!\!\!&=&\!\!\!(\varphi_{\text{G}}(a,b)\,\varphi_{\text{G}}(e_2,e_2))\varphi_{\text{G}}(c,d)
\\
\!\!\!&=&\!\!\!\varphi_{\text{G}}((a e_2)c,(b e_2)d)
\\
\!\!\!&=&\!\!\!\varphi_{\text{G}}((a \ov{c})e_2,(b \ov{d}) e_2).
\end{eqnarray*}
On the other hand, for the right hand side of same one, we have that
\begin{eqnarray*}
\varphi_{415}((a, b, \gamma_{{}_{\scriptscriptstyle {C}}})(c,d,1))\!\!\!&=&\!\!\!\varphi_{415}((a,b)\gamma_{{}_{\scriptscriptstyle {C}}}(c,d), \gamma_{{}_{\scriptscriptstyle {C}}})
\\
\!\!\!&=&\!\!\!\varphi_{\text{G}}(a(\gamma_{{}_{\scriptscriptstyle {C}}}c), b(\gamma_{{}_{\scriptscriptstyle {C}}}d))\gamma_{{}_{\scriptscriptstyle {C}}}
\\
\!\!\!&=&\!\!\!\varphi_{\text{G}}(a\ov{c}, b\ov{d})\,\varphi_{\text{G}}(e_2,e_2)
\\
\!\!\!&=&\!\!\!\varphi_{\text{G}}((a\ov{c})e_2, (b\ov{d})e_2)
\end{eqnarray*}
that is, $\varphi_{415}(a, b, \gamma_{{}_{\scriptscriptstyle {C}}})\, \varphi_{415}(c, d, 1)=\varphi_{415}((a, b, \gamma_{\scriptscriptstyle {C}})(c,d,1))$. Similarly, the other cases are shown.
 
We shall show that $\varphi_{415}$ is surjection. Let $\alpha \in (G_2)^\gamma \cap (G_2)^{\gamma_{{}_{\scriptscriptstyle {H}}}}$. Since $ (G_2)^\gamma \cap (G_2)^{\gamma_{{}_{\scriptscriptstyle {H}}}} \subset (G_2)^\gamma$, there exist $p, q \in S\!p(1)$ such that $\alpha=\varphi_{{}_\text{G}}(p, q)$ (Theorem 3.1.1). Moreover, from 
$ \alpha=\varphi_{{}_\text{G}}(p, q) \in (G_2)^{\gamma_{{}_{\scriptscriptstyle {H}}}}$, that is, $\gamma_{{}_{\scriptscriptstyle {H}}}\varphi_{{}_\text{G}}(p, q)\gamma_{{}_{\scriptscriptstyle {H}}}=\varphi_{{}_\text{G}}(p, q)$, we have $\varphi_{{}_\text{G}}(\gamma_{{}_{\scriptscriptstyle {H}}}p , \gamma_{{}_{\scriptscriptstyle {H}}} q)=\varphi_{{}_\text{G}}(p, q)$ (Lemma 4.1.4 (2)). Hence it follows that  
$$
\left \{
         \begin{array}{l}
                         \gamma_{{}_{\scriptscriptstyle {H}}} p  = p
                         \vspace{3mm}\\
                         \gamma_{{}_{\scriptscriptstyle {H}}} q  = q
         \end{array}\right.\qquad \text{or}\qquad 
\left \{         
          \begin{array}{l}
                          \gamma_{{}_{\scriptscriptstyle {H}}} p  = -p
                         \vspace{3mm}\\
                         \gamma_{{}_{\scriptscriptstyle {H}}} q  = -q.
         \end{array}\right. 
$$

\noindent In the former case, we see that $p, q \in U(1)$, then set $p=a, q=b, a, b \in U(1)$. Hence we have that $\alpha=\varphi_{{}_\text{G}}(a, b)=\varphi_{415}(a, b, 1)$. 
In the latter case, we can find the explicit form of $p, q$ as follow: $p=p_2 e_2 +p_3 e_3=(p_2 +p_3 e_1)e_2, q=q_2 e_2 +q_3 e_3=(q_2+q_3 e_1)e_2, p_k, q_k \in \R, k=2,3$, 
 that is, $p, q \in U(1)\, e_2 = \{ u e_2 \,|\, u \in U(1) \}$. Then, set $p=ae_2, q=be_2, a, b \in U(1)$, by using $\gamma_{\scriptscriptstyle {C}}=\varphi_{{}_{\rm G}}(e_2 , e_2)$ (Lemma 4.1.4 (1)), we have that 
$$
\alpha=\varphi_{{}_\text{G}}(ae_2, be_2)=\varphi_{{}_\text{G}}(a, b)\varphi_{{}_\text{G}}(e_2, e_2)=\varphi_{{}_\text{G}}(a, b)\gamma_{{}_{\scriptscriptstyle {C}}}=\varphi_{415}(a, b, \gamma_{{}_{\scriptscriptstyle {C}}}).
$$
Thus $\varphi_{415}$ is surjection. 

From $\Ker\,\varphi_{{}_\text{G}}=\{(1,1), (-1, -1) \}$, we can easily obtain that $\Ker\,\varphi_{415}=\{(1, 1, 1), (-1, -1,$ $ 1) \} \cong (\Z_2, 1)$. 

Therefore we have the required isomorphism
$$
(G_2)^\gamma \cap (G_2)^{\gamma_{{}_{\scriptscriptstyle {H}}}}\cong (U(1) \times U(1))/\Z_2 \rtimes \vspace{-3mm}\mathcal{Z}_2.
$$
\end{proof}

$\bullet$\,\,{\boldmath $[F_4]$} \,\,We study three types in here.
\subsection{Type FI-I-I}
In this section, we give a pair of involutive inner automorphisms $\tilde{\gamma}$ and ${\tilde{\gamma}}_{{}_{\scriptscriptstyle {H}}}$.
\vspace{1mm}

We define $\R$-linear transformations $\gamma_{\scriptscriptstyle {H}}, \gamma_{\scriptscriptstyle {C}}$ of $\mathfrak{J}$ by
$$
\gamma_{\scriptscriptstyle {H}} X
        = \begin{pmatrix} \xi_1 & \gamma_{{}_{\scriptscriptstyle {H}}} x_3 & \ov{\gamma_{{}_{\scriptscriptstyle {H}}} x_2} \\
\ov{\gamma_{{}_{\scriptscriptstyle {H}}} x_3} & \xi_2 & \gamma_{{}_{\scriptscriptstyle {H}}} x_1 \\
\gamma_{{}_{\scriptscriptstyle {H}}} x_2 & \ov{\gamma_{{}_{\scriptscriptstyle {H}}} x_1} & \xi_3   \end{pmatrix},\,\,
\gamma_{{}_{\scriptscriptstyle {C}}} X
        = \begin{pmatrix} \xi_1 & \gamma_{{}_{\scriptscriptstyle {C}}} x_3 & \ov{\gamma_{{}_{\scriptscriptstyle {C}}} x_2} \\
\ov{\gamma_{{}_{\scriptscriptstyle {C}}} x_3} & \xi_2 & \gamma_{{}_{\scriptscriptstyle {C}}} x_1 \\
\gamma_{{}_{\scriptscriptstyle {H}}} x_2 & \ov{\gamma_{{}_{\scriptscriptstyle {C}}} x_1} & \xi_3   \end{pmatrix},\,\,
X \in \mathfrak{J},
$$
where $\gamma_{{}_{\scriptscriptstyle {H}}}, \gamma_{{}_{\scriptscriptstyle {C}}}
$ 
of right hand side are the same ones as $\gamma_{{}_{\scriptscriptstyle {H}}}, \gamma_{{}_{\scriptscriptstyle {C}}} \in G_2$. Then we have that  $\gamma_{{}_{\scriptscriptstyle {H}}}, \gamma_{{}_{\scriptscriptstyle {C}}} \in F_4, {\gamma_{{}_{\scriptscriptstyle {H}}}}^2={\gamma_{{}_{\scriptscriptstyle {C}}}}^2=1$. Hence   $\gamma_{{}_{\scriptscriptstyle {H}}}, \gamma_{{}_{\scriptscriptstyle {C}}}$ induce involutive inner automorphisms 
${\tilde{\gamma}_{{}_{\scriptscriptstyle {H}}}}, {\tilde{\gamma}_{{}_{\scriptscriptstyle {C}}}}$ of $F_4$: ${\tilde{\gamma}_{{}_{\scriptscriptstyle {H}}}}(\alpha)=\gamma_{\scriptscriptstyle {H}} \alpha \gamma_{{}_{\scriptscriptstyle {H}}}, {\tilde{\gamma}_{{}_{\scriptscriptstyle {C}}}}(\alpha)=\gamma_{\scriptscriptstyle {C}} \alpha \gamma_{{}_{\scriptscriptstyle {C}}}, \alpha \in F_4$. (Remark. In $F_4$, we use $\gamma_{{}_{\scriptscriptstyle {C}}}$, however we do not use ${\tilde{\gamma}_{{}_{\scriptscriptstyle {C}}}}$.)
\vspace{1mm}

\noindent Moreover, using the inclusion $G_2 \subset F_4$, the $\R$-linear transformations $\delta_1, \delta_2$ defined in Lemma 4.1.1 are naturally extended to $\R$-linear transformations of $\mathfrak{J}$ as follows:
$$
\delta_k X=\begin{pmatrix} \xi_1 & \delta_k x_3 & \ov{\delta_k x_2} \\
\ov{\delta_k x_3} & \xi_2 & \delta_k x_1 \\
\delta_k x_2 & \ov{\delta_k x_1} & \xi_3   \end{pmatrix}
,\,\,X \in \mathfrak{J}, \,\,k=1,2.
$$
Then we see $\delta_1, \delta_2 \in F_4, {\delta_1}^2={\delta_2}^2 =1$. As in $G_2$, since we easily see that $\gamma \sim \gamma_{{}_{\scriptscriptstyle {H}}}, \gamma \sim \gamma\gamma_{{}_{\scriptscriptstyle {H}}}$ in $F_4$, we have the following proposition.

\begin{prop}
The group $(F_4)^\gamma$ is isomorphic to both of the groups $(F_4)^{\gamma_{{}_{\scriptscriptstyle {H}}}}$ and $(F_4)^{\gamma \gamma_{{}_{\scriptscriptstyle {H}}}}$\,{\rm :}\, $(F_4)^\gamma \cong (F_4)^{\gamma_{{}_{\scriptscriptstyle {H}}}} \cong (F_4)^{\gamma \gamma_{{}_{\scriptscriptstyle {H}}}}$.
\end{prop}
\vspace{1mm}

From the result of type FI in Table 2 and Proposition 4.2.1,
we have the following theorem.

\begin{thm}For $\mathbb{Z}_2 \times \mathbb{Z}_2=\{1,\gamma \} \times \{1, \gamma_{\scriptscriptstyle {H}} \}$, the $\mathbb{Z}_2 \times \mathbb{Z}_2$-symmetric space is of type $(F_4/(F_4)^\gamma, F_4/(F_4)^{\gamma_{{}_{\scriptscriptstyle {H}}}}, F_4/(F_4)^{\gamma\gamma_{{}_{\scriptscriptstyle {H}}}})=(F_4/(F_4)^\gamma, F_4/(F_4)^\gamma, F_4/(F_4)^\gamma)$, that is, type {\rm (FI, FI, FI)}, abbreviated as {\rm FI-I-I}.
\end{thm}

Here, 
we prove lemma needed and make some preparations for the theorem below.

\begin{lem}
The mapping $\varphi_{\rm{F1}}: S\!p(1) \times S\!p(3) \to (F_4) ^\gamma$ of Theorem 3.2.1 satisfies  the equalities{\rm :} 
\vspace{-7mm}

\begin{eqnarray*}
&&{\rm (1)}\, \,
\gamma_{\scriptscriptstyle {H}}=\varphi_{{}_{\rm{F1}}}(e_1 , e_1 E), 
\quad 
\gamma_{\scriptscriptstyle {C}}=\varphi_{{}_{\rm{F1}}}(e_2 , e_2 E).
\\[0mm]
&&{\rm (2)}\,\,\gamma_{\scriptscriptstyle {H}}\varphi_{{}_{\rm{F1}}}(p, A)\gamma_{\scriptscriptstyle {H}}=\varphi_{{}_{\rm{F1}}}(\gamma_{\scriptscriptstyle {H}}p, \gamma_{\scriptscriptstyle {H}}A).
\end{eqnarray*}
\end{lem}
\begin{proof}
The proof of (1) is omitted (see \cite[Lemma 2.3.4]{Yokotaichiro1} in detail). The equality of (2) is the direct result of (1).
\end{proof}
Consider a group  $\mathcal{Z}_2 = \{1, \gamma_{\scriptscriptstyle {C}} \}$. Then the group $\mathcal{Z}_2 = \{1, \gamma_{\scriptscriptstyle {C}} \}$ acts on the group $U(1) \times U(3)$ by
$$
  \gamma_{\scriptscriptstyle {C}}(a, B)=(\gamma_{\scriptscriptstyle {C}} a, \gamma_{\scriptscriptstyle {C}} B)
$$ 
and let $(U(1) \times U(3)) \rtimes \mathcal{Z}_2$ be the semi-direct product of $U(1) \times U(3)$ and $\mathcal{Z}_2$ with this action. 
\vspace{1mm}

Now, we determine the structure of the group $(F_4)^\gamma \cap (F_4)^{\gamma_{{}_{\scriptscriptstyle {H}}}}$.

\begin{thm} We have that $(F_4)^\gamma \cap (F_4)^{\gamma_{{}_{\scriptscriptstyle {H}}}}\cong (U(1) \times U(3))/\Z_2 \rtimes \mathcal{Z}_2,
\Z_2= \{(1, E), $ $ (-1, -E) \}, \, \mathcal{Z}_2 = \{1, \gamma_{\scriptscriptstyle {C}} \}$.
\end{thm}
\begin{proof}
We define a mapping $\varphi_{424}: (U(1) \times U(3)) \rtimes \{1, \gamma_{\scriptscriptstyle {C}} \} \to (F_4)^\gamma \cap (F_4)^{\gamma_{{}_{\scriptscriptstyle {H}}}}$ by
\begin{eqnarray*}
   \varphi_{424}(a, B, 1)\!\!\!&=&\!\!\!\varphi_{{}_\text{F1}}(a,B),
\\
   \varphi_{424}(a, B, \gamma_{\scriptscriptstyle {C}})\!\!\!&=&\!\!\!\varphi_{{}_\text{F1}}(a,B)\, \gamma_{\scriptscriptstyle {C}},
\end{eqnarray*}
where $\varphi_{{}_\text{F1}}$ is defined in Theorem 3.2.1. As the proof of Theorem 4.1.5, it is easily to verify that $\varphi_{424}$ is well-defined and a homomorphism. 

We shall show that $\varphi_{424}$ is surjection. Let $\alpha \in  (F_4)^\gamma \cap (F_4)^{\gamma_{{}_{\scriptscriptstyle {H}}}}$. Since $(F_4)^\gamma \cap (F_4)^{\gamma_{{}_{\scriptscriptstyle {H}}}} \subset (F_4)^\gamma$, there exist $p\in S\!p(1)$ and $A \in S\!p(3)$ such that $\alpha=\varphi_{{}_\text{F1}}(p, A)$ (Theorem 3.2.1). Moreover, from $\alpha =\varphi_{{}_\text{F1}}(p, A) \in (F_4)^{\gamma_{{}_{\scriptscriptstyle {H}}}}$, 
that is, $\gamma_{\scriptscriptstyle {H}}\varphi_{{}_\text{F1}}(p, A)\gamma_{\scriptscriptstyle {H}}=\varphi_{{}_\text{F1}}(p, A)$, we have $\varphi_{{}_\text{F1}}(\gamma_{\scriptscriptstyle {H}} p , \gamma_{\scriptscriptstyle {H}} A )$ $=\varphi_{{}_\text{F1}}(p, A)$ (Lemma 4.2.3 (2)).
Hence it follows that  
$$
\left \{
         \begin{array}{l}
                          \gamma_{\scriptscriptstyle {H}} p = p
                         \vspace{3mm}\\
                           \gamma_{\scriptscriptstyle {H}} A = A
         \end{array}\right.\qquad \text{or}\qquad 
\left \{         
          \begin{array}{l}
                          \gamma_{\scriptscriptstyle {H}} p = -p
                         \vspace{3mm}\\
                          \gamma_{\scriptscriptstyle {H}} A = -A.
         \end{array}\right. 
$$

\noindent In the former case, we easily see that $p\in U(1),\, A \in U(3)$, then set $p=a \in U(1)$ and $A=B \in U(3)$. Hence we have that $\alpha=\varphi_{{}_\text{F1}}(a, B)=\varphi_{424}(a, B, 1)$.
In the latter case, in the way similar to the former case of Theorem 4.1.5 
we find that $p \in U(1) e_2=\{u e_2\,|\,u \in U(1) \},$ $ A \in U(3) (e_2 E)= \{B (e_2 E)\,|\, B \in U(3) \}$.
Then, set $p=ae_2, A=B(e_2 E), a \in U(1), B \in U(3)$, by using  $\gamma_{\scriptscriptstyle {C}}=\varphi_{{}_{\text{F1}}}(e_2, e_2 E)$ (Lemma 4.2.3 (1)), we have that 
$$
\alpha=\varphi_{{}_\text{F1}}(ae_2, B (e_2 E))=\varphi_{{}_\text{F1}}(a, B)\varphi_{{}_\text{F1}}(e_2, e_2 E)=\varphi_{{}_\text{F1}}(a, B)\gamma_{\scriptscriptstyle {C}}=\varphi_{424}(a, B, \gamma_{\scriptscriptstyle {C}}).
$$
Thus $\varphi_{423}$ is surjection. 

From $\Ker\,\varphi_{{}_\text{F1}}=\{(1,E), (-1, -E) \}$, we can easily obtain that $\Ker\,\varphi_{424}=\{(1, E, 1), (-1,$ $ -E, 1) \}\cong(\Z_2, 1)$. 

Therefore we have the required isomorphism
 $$
(F_4)^\gamma \cap (F_4)^{\gamma_{{}_{\scriptscriptstyle {H}}}}\cong (U(1) \times U(3))/\Z_2 \rtimes \mathcal{Z}_2.
$$
\end{proof}
\subsection{Type FI-I-II}
In this section, we give a pair of involutive inner automorphisms $\tilde{\gamma}$ and $\tilde{\gamma\sigma}$
\vspace{1mm}

\begin{lem}In $F_4$, $\gamma$ is conjugate to $\gamma\sigma${\rm :}
$\gamma \sim \gamma\sigma$.
\end{lem}
\begin{proof}
We define an $\R$- linear transformation $\delta_5$ of $\mathfrak{J}$ by
$$
  \delta_5 X=\begin{pmatrix} 
                          \xi_1 & x_3 e_4 & \ov{x}_2 e_4 \\
          -e_4 \ov{x}_3 & \xi_2 & -e_4 x_1 e_4\\                            
          -e_4 x_2 & -e_4 \ov{x}_1 e_4 & \xi_3    
  \end{pmatrix}, \,\, X \in \mathfrak{J}.
$$
Then we have that $\delta_5 \in F_4, {\delta_3}^2 =1, 
\delta_5 \gamma=(\gamma\sigma)\delta_5$, that is, $\gamma \sim \gamma\sigma$ in $F_4$. 
\end{proof}

We have the following proposition which is the direct result of Lemma 4.3.1.
\begin{prop}
The group $(F_4)^\gamma$ is isomorphic to the group $(F_4)^{\gamma\sigma}${\rm :}\,$(F_4)^\gamma \cong (F_4)^{\gamma\sigma}$.
\end{prop}

From the result of types FI, FII in Table 2 and Proposition 4.3.2, we have the following theorem.
\begin{thm} For $\mathbb{Z}_2 \times \mathbb{Z}_2=\{1,\gamma \} \times \{1, \gamma\sigma \}$, the $\mathbb{Z}_2 \times \mathbb{Z}_2$-symmetric space is of type $(F_4/(F_4)^\gamma, F_4/(F_4)^{\gamma\sigma}, F_4/(F_4)^{\gamma(\gamma\sigma)})\!=\!(F_4/(F_4)^\gamma, F_4/(F_4)^{\gamma}, F_4/(F_4)^\sigma)$, that is, type {\rm (FI, FI, FII)}, abbreviated as {\rm FI-I-II}.
\end{thm}
Here, 
we prove lemma needed in theorem below.
\begin{lem}
The mapping $\varphi_{{}_{\rm F1}}: S\!p(1) \times S\!p(3) \to (F_4) ^\gamma$ of Theorem 3.2.1 satisfies  the equalities{\rm :} 
\begin{eqnarray*}
&&{\rm (1)}\, \,
\gamma=\varphi_{{}_{\rm{F1}}}(-1, -E), 
\,\, 
\sigma=\varphi_{{}_{\rm{F1}}}(-1 , I_1).
\\[0mm]
&&{\rm (2)}\,\,\gamma\varphi_{{}_{\rm{F1}}}(p, A)\gamma=\varphi_{{}_{\rm{F1}}}(p,A),\,\, \sigma\varphi_{{}_{\rm{F1}}}(p, A)\sigma=\varphi_{{}_{\rm{F1}}}(p,I_1 A I_1),
\end{eqnarray*}
where $I_1=\diag(-1, 1, 1)$.
\end{lem}
\begin{proof}
The proof of (1) is omitted (see \cite[Lemma 2.3.4]{Yokotaichiro1} in detail). The equalities of (2) are the direct result of (1).
\end{proof}

Now, we determine the structure of the group $(F_4)^\gamma \cap (F_4)^{\gamma\sigma}$.

\begin{thm} We have that $(F_4)^\gamma \cap (F_4)^{\gamma\sigma} \cong (S\!p(1) \times S\!p(1) \times S\!p(2))/\Z_2, \,\Z_2 =\{(1, 1, E), $ $(-1, -1, E)  \}$.
\end{thm}
\begin{proof}
We define a mapping $\varphi_{435}; S\!p(1) \times S\!p(1) \times S\!p(2) \to (F_4)^\gamma \cap (F_4)^{\gamma\sigma}$ by 
$$
    \varphi_{435} (p, q, B)=\varphi_{{}_\text{F1}}(p, h(q, B)),
$$
where $h$ is defined by $h:S\!p(1) \times S\!p(2) \to S\!p(3),\,h(q, B)=\begin{pmatrix} 
                                      q & 0 \\
                                      0 & B
                    \end{pmatrix}.
$ 
Since the mapping $\varphi_{435}$ is the restriction of the mapping $\varphi_{{}_{\rm F1}}$, it is easily to verify that
 $\varphi_{435}$ is well-defined and a homomorphism. 

We shall show that $\varphi_{435}$ is surjection. Let  $\alpha  \in (F_4)^\gamma \cap (F_4)^{\gamma\sigma}$. Since $(F_4)^\gamma \cap (F_4)^{\gamma\sigma} \subset (F_4)^\gamma$,  there exist $p\in S\!p(1)$ and $A \in S\!p(3)$ such that $\alpha=\varphi_{{}_\text{F1}}(p, A)$ (Theorem 3.2.1). Moreover, from $\alpha=\varphi_{{}_\text{F1}}(p, A) \in (F_4)^{\gamma\sigma}$,
that is, $(\gamma\sigma)\varphi_{{}_\text{F1}}(p, A)(\sigma\gamma)=\varphi_{{}_\text{F1}}(p, A)$, using $ \gamma\varphi_{{}_{\rm{F1}}}(p, A)\gamma=\varphi_{{}_{\rm{F1}}}(p,A)$ and $\sigma\varphi_{{}_{\rm{F1}}}(p, A)\sigma=\varphi_{{}_{\rm{F1}}}(p,I_1 A I_1)$ (Lemma 4.3.4 (2)), we have $\varphi_{{}_\text{F1}}(p, I_1 A I_1)=\varphi_{{}_\text{F1}}(p, A)$. 
Hence it follows that  
$$
\left \{
         \begin{array}{l}
                         p = p
                         \vspace{3mm}\\
                         I_1 A I_1= A
         \end{array}\right.\qquad \text{or}\qquad 
\left \{         
          \begin{array}{l}
                          p = -p
                         \vspace{3mm}\\
                         I_1 A I_1= -A.
         \end{array}\right. 
$$

\noindent In the former case, it is trivial that $p \in S\!p(1)$, and we get the explicit form of $A \in S\!p(3)$ as follows: 
$$
A=\begin{pmatrix} 
                                      q & 0 \\
                                      0 & B
                    \end{pmatrix}
, q \in S\!p(1), B \in S\!p(2).
$$
Hence we have $\alpha
=\varphi_{{}_\text{F1}}(p,h(q, B))=\varphi_{435}(p, q, B)$.
In the latter case, this case is impossible because of $p=0$ for $p \in S\!p(1)$. Thus $\varphi_{435}$ is surjection. 

From $\Ker\,\varphi_{{}_\text{F1}}\!=\!\{(1, E), (-1, -E) \}$, we can easily obtain that $\Ker\,\varphi_{435}\!=\!\{(1, 1, E), (-1, -1,$ $ -E) \} \cong \Z_2$. 

Therefore we have the required isomorphism 
$$
(F_4)^\gamma \cap (F_4)^{\gamma\sigma} \cong (S\!p(1) \times S\!p(1) \times \vspace{-3mm}S\!p(2))/\Z_2.
$$
\end{proof}

\subsection{Type FII-II-II}
In this section, we give a pair of involutive inner automorphisms $\tilde{\sigma}$ and $\tilde{\sigma}'$, where an $\R$-linear transformation $\sigma'$ of $\mathfrak{J}$ is defined below.
\vspace{1mm}

We define an $\R$-linear transformation $\sigma'$ of $\mathfrak{J}$ by 
$$
\sigma' X = \begin{pmatrix} \xi_1 & x_3 & -\ov{x}_2 \\
                         \ov{x}_3 & \xi_2 & -x_1 \\
                          -x_2 & -\ov{x}_1 & \xi_3   \end{pmatrix}, \,\,X \in \mathfrak{J}.
$$
Then we have that $\sigma' \in F_4, {\sigma'}^2=1, \sigma\sigma'=\sigma'\sigma$. Hence $\sigma'$ induces involutive inner automorphism ${\tilde{\sigma}}'$ of $F_4$: ${\tilde{\sigma}}'(\alpha)=\sigma' \alpha \sigma', \alpha \in F_4$. 
\begin{lem} In $F_4$, $\sigma$ is conjugate to both of $\sigma'$ and $ \sigma\sigma'${\rm :} $\sigma \sim \sigma', \sigma \sim \sigma\sigma'$.
\end{lem}
\begin{proof}
We define $\R$-linear transformations $\delta_6, \delta_7$ of $\mathfrak{J}$ by
$$
\delta_6 X=\begin{pmatrix} \xi_3 & \ov{x}_1 & x_2 \\
                            x_1 & \xi_2 & \ov{x}_3 \\
                          \ov{x}_2 & x_3 & \xi_1   \end{pmatrix},\,\,\,
\delta_7 X=\begin{pmatrix} \xi_2 & \ov{x}_3 & x_1 \\
                         x_3 & \xi_1 & \ov{x}_2 \\
                          \ov{x}_1 & x_2 & \xi_3   \end{pmatrix},\,\,X \in \mathfrak{J}.
$$
Then we have that $\delta_6, \delta_7 \in F_4, 
{\delta_6}^2={\delta_7}^2=1$.
Hence, by straightforward computation, we have that $\delta_6 \sigma=\sigma' \delta_6,\, \delta_7 \sigma=(\sigma\sigma')\delta_7$, that is, $\sigma \sim \sigma', \sigma \sim \sigma\sigma'$ in $F_4$.
\end{proof}

We have the following proposition which is the direct result of Lemma 4.4.1.

\begin{prop}
The group $(F_4)^\sigma$ is isomorphic to both of the groups $(F_4)^{\sigma'}$ and $(F_4)^{\sigma\sigma'}${\rm:}\\$(F_4)^\sigma \cong (F_4)^{\sigma'} \cong (F_4)^{\sigma\sigma'}$.
\end{prop}

From the result of type FII in Table 2 and Proposition 4.4.2, we have the following theorem.

\begin{thm} For $\mathbb{Z}_2 \times \mathbb{Z}_2=\{1,\sigma \} \times \{1, \sigma' \}$, the $\mathbb{Z}_2 \times \mathbb{Z}_2$-symmetric space is of type $(F_4/(F_4)^\sigma, F_4/(F_4)^{\sigma'}, F_4/(F_4)^{\sigma \sigma'})=(F_4/(F_4)^\sigma, F_4/(F_4)^\sigma, F_4/(F_4)^\sigma)$, that is, type {\rm (FII, FII, FII)}, abbreviated as {\rm FII-II-II}.
\end{thm}
Here, we prove lemma needed in the theorem below.

\begin{lem} 
The Lie algebra 
$(\mathfrak{f}_4)^{\sigma} \cap  (\mathfrak{f}_4)^{\sigma'}$ of the group 
$(F_4)^{\sigma} \cap (F_4)^{\sigma'}$ is given by
$$
(\mathfrak{f}_4)^{\sigma} \cap  (\mathfrak{f}_4)^{\sigma'}=\{ D \in \mathfrak{so}(8) \} = \mathfrak{so}(8).
$$

In particular, we have
$$
     \dim((\mathfrak{f}_4)^{\sigma} \cap  (\mathfrak{f}_4)^{\sigma'}) = 28. 
$$  
\end{lem}
\begin{proof}
Since any element $\delta$ of the Lie algebra $\mathfrak{f}_4$ of the group $F_4$ is uniquely expressed as
$$
   \delta = D + \ti{A}_1(a_1) + \ti{A}_2(a_2) + \ti{A}_3(a_3), \, D \in \mathfrak{so}(8), a_i \in \mathfrak{C}, k=1,2,3, \vspace{-3mm}
$$
using 
\begin{eqnarray*}
\sigma \delta \sigma=D+\ti{A}_1(a_1) + \ti{A}_2(-a_2) + \ti{A}_3(-a_3),
\\
\sigma' \delta \sigma' =D+\ti{A}_1(-a_1) + \ti{A}_2(-a_2) + \ti{A}_3(a_3),
\end{eqnarray*}
we can easily prove this lemma.
\end{proof}

Now, we determine the structure of the group $(F_4)^\sigma \cap (F_4)^{\sigma'}$.

\begin{thm} We have that $(F_4)^\sigma \cap (F_4)^{\sigma'} \cong S\!pin(8)$.
\end{thm}
\begin{proof}
 Let $S\!pin(8) = \{(\alpha_1, \alpha_2, \alpha_3) \in S\!O(8) \times S\!O(8) \times S\!O(8)\,|\, (\alpha_1 x)(\alpha_2 y)=\ov{\alpha_3 (\ov{xy})}, x, y \in \mathfrak{C} \}$. We define a mapping $\varphi_{445} : S\!pin(8) \to 
(F_4)^{\sigma} \cap (F_4)^{\sigma'}$ by
$$
   \varphi_{445}(\alpha_1, \alpha_2, \alpha_3) X
= \begin{pmatrix}\xi_1 & \alpha_3x_3 & \ov{\alpha_2x_2} \\
                  \ov{\alpha_3x_3} & \xi_2 & \alpha_1x_1 \\
                  \alpha_2x_2 & \ov{\alpha_1x_1} & \xi_3
\end{pmatrix}, \,\,  X \in \mathfrak{J}.
 $$
 It is easily to verify that  $\varphi_{445}$ is well-defined, a homomorphism and injection. 

We shall show that  $\varphi_{445}$ is surjection. From $(F_4)^ \sigma \cong S\!pin(9)$ (\cite [Proposition 1.4]{M.Y.01}), we have that $(F_4)^\sigma \cap (F_4)^{\sigma'} =((F_4)^\sigma)^{\sigma'} \cong (S\!pin(9))^{\sigma'}$. 
Hence $(F_4)^\sigma \cap (F_4)^{\sigma'}$ is connected. Moreover, together with $\dim((\mathfrak{f}_4)^{\sigma} \cap  (\mathfrak{f}_4)^{\sigma'}) = 28 = \dim (\mathfrak{so}(8))$ (Lemma 4.4.4), we have that $\varphi_{445}$ is surjection.

Therefore we have the required isomorphism 
$$
(F_4)^\sigma \cap (F_4)^{\sigma'} \cong S\!pin(8).
$$
\end{proof}

$\bullet$\,\,{\boldmath $[E_6]$}\,\,We study eight types in here.
\vspace{1mm}

\subsection{Type EI-I-II}

 
In this section, we give a pair of involutive  automorphisms $\lambda\tilde{\gamma}$ and $\lambda\tilde{\gamma\gamma_{\scriptscriptstyle {C}}}$.
\vspace{1mm}

Let the $C$-linear transformations $\gamma_{{}_{\scriptscriptstyle {H}}}, \gamma_{{}_{\scriptscriptstyle {C}}}$ of $\mathfrak{J}^C$ be the complexification of $\gamma_{{}_{\scriptscriptstyle {H}}}, \gamma_{{}_{\scriptscriptstyle {C}}} \in G_2 \subset F_4$. Then we have that $\gamma_{{}_{\scriptscriptstyle {H}}}, \gamma_{{}_{\scriptscriptstyle {C}}} \in E_6,{\gamma_{{}_{\scriptscriptstyle {H}}}}^2={\gamma_{{}_{\scriptscriptstyle {C}}}}^2=1$, so $\gamma_{{}_{\scriptscriptstyle {C}}}$ involutive inner automorphism $\tilde{\gamma_{{}_{\scriptscriptstyle {C}}}}$ of $E_6$:\,$\tilde{\gamma_{{}_{\scriptscriptstyle {C}}}}(\alpha)=\gamma_{{}_{\scriptscriptstyle {C}}}\alpha \gamma_{{}_{\scriptscriptstyle {C}}}, \alpha \in E_6$.

\noindent Using the inclusion $G_2 \subset F_4 \subset E_6$, the $\R$-linear transformations $\delta_3, \delta_4$ defined in Lemma 4.1.1 are naturally extended to $C$-linear transformation of $\mathfrak{J}^C$. Hence, as in $G_2$, since we easily see that $\delta_3 \gamma=\gamma_{{}_{\scriptscriptstyle {C}}}\delta_3, \delta_4 \gamma=\gamma \gamma_{{}_{\scriptscriptstyle {C}}}\delta_4$ as $\delta_3, \delta_4 \in E_6$, that is, $\gamma \sim \gamma_{{}_{\scriptscriptstyle {C}}}, \gamma \sim \gamma \gamma_{{}_{\scriptscriptstyle {C}}}$ in $E_6$, we have the following proposition.
\begin{prop} {\rm (1)} The group $(E_6)^{\lambda\gamma}$ is isomorphic to the group $(E_6)^{\lambda\gamma\gamma_{{}_{\scriptscriptstyle {C}}}}${\rm :}  $(E_6)^{\lambda\gamma} \cong (E_6)^{\lambda\gamma\gamma_{{}_{\scriptscriptstyle {C}}}}$.

{\rm (2)} The group $(E_6)^{\gamma}$ is isomorphic to the group $(E_6)^{\gamma_{{}_{\scriptscriptstyle {C}}}}${\rm :} $(E_6)^{\gamma} \cong (E_6)^{\gamma_{{}_{\scriptscriptstyle {C}}}}$.
\end{prop}
\begin{proof}
(1) We define a mapping $f: (E_6)^{\lambda\gamma} \to (E_6)^{\lambda\gamma\gamma_{{}_{\scriptscriptstyle {C}}}}$ by 
$$
    f(\alpha)=\delta_4 \alpha \delta_4.
$$
 In order to prove this proposition, it is sufficient to show that the mapping $f$ is well-defined. Indeed, it follows from $\gamma \sim \gamma \gamma_{{}_{\scriptscriptstyle {C}}}$ that 
\begin{eqnarray*}
  \lambda(\gamma\gamma_{\scriptscriptstyle {C}}f(\alpha) \gamma_{\scriptscriptstyle {C}} \gamma)\!\!\!&=\!\!\!&{}^t (\gamma\gamma_{\scriptscriptstyle {C}}  (\delta_4 \alpha \delta_4) \gamma_{\scriptscriptstyle {C}} \gamma)^{-1}=\gamma\gamma_{\scriptscriptstyle {C}} (\delta_4 {}^t \alpha^{-1} \delta_4)\gamma_{\scriptscriptstyle {C}} \gamma
 \\[1mm]
&=\!\!\! & \delta_4 (\gamma \lambda (\alpha)\gamma)\delta_4=\delta_4 \alpha \delta_4 =f(\alpha),
\end{eqnarray*}
that is, $f(\alpha) \in (E_6)^{\lambda\gamma\gamma_{{}_{\scriptscriptstyle {C}}}}\vspace{2mm}$.

(2) This isomorphism is the direct result of $\gamma \sim\gamma_{{}_{\scriptscriptstyle {C}}}$.
\end{proof}

\vspace{1mm}

From the result of types EI, EII in Table 2 and Propositions 4.5.2 
, we have the following theorem.

\begin{thm} For $\mathbb{Z}_2 \times \mathbb{Z}_2=\{1,\lambda\gamma \} \times \{1, \lambda\gamma\gamma_{\scriptscriptstyle {C}} \}$, the $\mathbb{Z}_2 \times \mathbb{Z}_2$-symmetric space is of type $(E_6/(E_6)^{\lambda\gamma}, E_6/(E_6)^{\lambda\gamma\gamma_{{}_{\scriptscriptstyle {C}}}}, E_6/(E_6)^{(\lambda\gamma)(\lambda\gamma\gamma_{{}_{\scriptscriptstyle {C}}})})\!=\!(E_6/(E_6)^{\lambda\gamma}, E_6/(E_6)^{\lambda\gamma\gamma_{{}_{\scriptscriptstyle {C}}}}, E_6/(E_6)^{\gamma_{{}_{\scriptscriptstyle {C}}}})\!=\!(E_6/(E_6)^{\lambda\gamma}, E_6/(E_6)^{\lambda\gamma}, E_6/(E_6)^{\gamma})$, that is, type {\rm (EI, EI, EII)}, abbreviated as {\rm  EI-I-II}.
\end{thm}

Here, 
we prove lemma and proposition needed and make some preparations for the theorem below. 




\begin{lem} The mapping $\varphi_{{}_{\rm E2}}: S\!p(1) \times S\!U(6) \to (E_6)^\gamma$ of Theorem 3.3.2 satisfies the following equalities{\rm :}
\vspace{-5mm}

\begin{eqnarray*}
&&
{\rm (1)}\,\,\gamma =\varphi_{{}_{\rm E2}}(-1, E),\,  \gamma_{\scriptscriptstyle {H}}=\varphi_{{}_{\rm E2}}(e_1, iI),\, \gamma_{\scriptscriptstyle {C}}=\varphi_{{}_{\rm E2}}(e_2, J), \,\sigma=\varphi_{{}_{\rm E2}}(-1, I_2).
\\[0mm]
&& {\rm (2)}\,\,\gamma\varphi_{{}_{\rm E2}}(p, U)\gamma= \varphi_{{}_{\rm E2}}(p, U),\,
\gamma_{\scriptscriptstyle {H}}\varphi_{{}_{\rm E2}}(p, U)\gamma_{\scriptscriptstyle {H}}=\varphi_{{}_{\rm E2}}(\gamma_{\scriptscriptstyle {H}} p, IUI), 
\\[0mm]
&& 
\hspace*{5mm}\gamma_{\scriptscriptstyle {C}}\varphi_{{}_{\rm E2}}(p, U)\gamma_{\scriptscriptstyle {C}}=\varphi_{{}_{\rm E2}}(\gamma_{\scriptscriptstyle {C}} p, -JUJ),\, \sigma \varphi_{{}_{\rm E2}}(p, U) \sigma=\varphi_{{}_{\rm E2}}(p, I_2 U I_2).
\\[0mm]
&&{\rm (3)}\,\,\lambda(\varphi_{{}_{\rm E2}}(p, U) )=\varphi_{{}_{\rm E2}}(p, -J(\tau U)J),
\end{eqnarray*}
where $i \in C, I=\diag(1, -1, 1, -1, 1, -1), J=\diag(J_1, J_1, J_1), J_1=\begin{pmatrix} 0 & 1 \\
                                       -1 & 0
                       \end{pmatrix}, I_2 =\diag(-1,$ $ -1, 1,1,1,1)$.

\end{lem}                         
\begin{proof}
The proof of (1) is omitted (see \cite[Lemmas 3.5.7, 3.5.10]{Yokotaichiro1} in detail). The equalities of (2) are the direct results of (1). 
 \vspace{1mm}
 
(3)\,\,Using the equality $\tau k(M)=-Jk(\tau M)J$,  
we  have the required result. Indeed,
\begin{eqnarray*}
 \lambda (\varphi_{{}_{\text E2}}(p, U))(M+\a) &\!\!=\!\!&\tau (\varphi_{{}_{\text E2}}(p, U))\tau (M+\a)\,\,{\text{(see Section 3.3)}}\\
                                      &\!\!=\!\!& \tau (\varphi_{{}_{\text E2}}(p, U))(\tau M +\tau\a) \\                                      &\!\!=\!\!& 
\tau ({k_J}^{-1}(U k_J (\tau M) {}^t U)+p(\tau\a) k^{-1}(\tau \,{}^t U)) \\ 
                                       &\!\!=\!\!&
-\tau k^{-1}((U k_J (\tau M) {}^t U)J)+p\a \tau k^{-1}(\tau \,{}^t U)  \\
                                       &\!\!=\!\!&
k^{-1}(J(\tau U)J(kM)({}^t U)(-E))+p\a k^{-1}(\tau\,{}^t(-J(\tau U)J)) \\                                       &\!\!=\!\!&
k^{-1}((-J(\tau U)J)(kM)J(-J(\tau \,{}^t U)J)(-J))+p\a  k^{-1}(\tau\,{}^t(-J(\tau U)J)) \\                                       &\!\!=\!\!&
{k_J}^{-1}((-J(\tau U)J)(k_J M){}^t ((-J(\tau U)J))) + p\a  k^{-1}(\tau\,{}^,(-J(\tau U)J)) \\
                                       &\!\!=\!\!&
\varphi_{{}_{\text E2}}(p, -J(\tau U)J)(M+\a),
\end{eqnarray*}

\noindent that is, $\lambda( \varphi_{{}_{\text E2}}(p, U))=\varphi_{{}_{\text E2}}(p, -J(\tau U)J)$.
\end{proof} 

\begin{prop} The group $(E_6)^{\lambda\gamma} \cap (E_6)^{\lambda\gamma\gamma_{{}_{\scriptscriptstyle {C}}}}$ is isomorphic to the group $(E_6)^{\lambda\gamma} \cap (E_6)^{\gamma_{{}_{\scriptscriptstyle {C}}}}${\rm :} $(E_6)^{\lambda\gamma} \cap (E_6)^{\lambda\gamma\gamma_{{}_{\scriptscriptstyle {C}}}} \cong (E_6)^{\lambda\gamma} \cap (E_6)^{\gamma_{{}_{\scriptscriptstyle {C}}}}$.
\end{prop} 
\begin{proof}
We define a mapping $g : (E_6)^{\lambda\gamma} \cap (E_6)^{\lambda\gamma\gamma_{{}_{\scriptscriptstyle {C}}}} \to (E_6)^{\lambda\gamma} \cap (E_6)^{\gamma_{{}_{\scriptscriptstyle {C}}}}$ by 
$$
  g(\alpha)=\lambda(\alpha).
$$
In order to prove this proposition, it is sufficient to show that the mapping $g$ is well-defined. Indeed, it follows from $\lambda(\gamma)=\gamma, \lambda(\gamma_{{}_{\scriptscriptstyle {C}}})=\gamma_{{}_{\scriptscriptstyle {C}}}$ that
\vspace{-5mm}

\begin{eqnarray*}
  && \lambda(\gamma g(\alpha) \gamma)=\lambda (\gamma \lambda (\alpha) \gamma)=\lambda(\alpha)=g(\alpha)\,\,\, {\text {and}}
\\
&&  \gamma_{\scriptscriptstyle {C}} g(\alpha) \gamma_{\scriptscriptstyle {C}}=\gamma_{\scriptscriptstyle {C}} \lambda(\alpha) \gamma_{\scriptscriptstyle {C}}=\gamma(\gamma\gamma_{\scriptscriptstyle {C}} \lambda(\alpha)\gamma_{\scriptscriptstyle {C}}\gamma)\gamma=\gamma\alpha\gamma=\gamma(\gamma \lambda (\alpha) \gamma)\gamma =\lambda (\alpha)=g(\alpha),  
\end{eqnarray*}
that is, $g(\alpha) \in  (E_6)^{\lambda\gamma}$ and $g(\alpha) \in  (E_6)^{\gamma_{{}_{\scriptscriptstyle {C}}}}$.
\end{proof}
Let $\{a=x+y e_2 \,|\, \ov{a}a=1, x, y \in \R \} \subset S\!p(1)$ be a group which  is isomorphic to the ordinary unitary group $U(1)$, so this group is also denoted by $U(1)$. In this section, we use this \vspace{1mm}as $U(1)$.

Consider a group $\mathcal{Z}_2 = \{1, \gamma_{\scriptscriptstyle {H}} \}$. Then the group $\mathcal{Z}_2 = \{1, \gamma_{\scriptscriptstyle {H}} \}$ acts on the group $U(1) \times S\!O(6)$ by
$$
  \gamma_{\scriptscriptstyle {H}}(a, A)=(\ov{a}, (iI)A(iI)^{-1}),
$$ 
where $I=\diag(1, -1, 1, -1, 1, -1)$, and let $(U(1) \times S\! O(6)) \rtimes \mathcal{Z}_2$ be the semi-direct product of $U(1) \times S\!O(6)$ and $\mathcal{Z}_2$ with this action. 
\vspace{2mm}

Now, we determine the structure of the group $(E_6)^{\lambda\gamma} \cap (E_6)^{\lambda\gamma\gamma_{{}_{\scriptscriptstyle {C}}}}$.

\begin{thm} We have that $(E_6)^{\lambda\gamma} \cap (E_6)^{\lambda\gamma\gamma_{{}_{\scriptscriptstyle {C}}}} \cong (U(1) \times S\! O(6))/\Z_2 \rtimes \mathcal{Z}_2, \Z_2=\{(1, E),$ $ (-1, -E) \}, \mathcal{Z}_2 =\{ 1, \gamma_{\scriptscriptstyle {H}} \}$.
\end{thm}
\begin{proof}
We define a mapping $(U(1) \times S\! O(6)) \rtimes \{ 1, \gamma_{\scriptscriptstyle {H}} \} \to (E_6)^{\lambda\gamma} \cap (E_6)^{\gamma_{{}_{\scriptscriptstyle {C}}}}$ by
\begin{eqnarray*}
\hspace*{-7mm} \varphi_{456}((a, A), 1)\!\!\!&=&\!\!\!\delta_7 \varphi_{{}_{\text E2}}(a, A) \vspace{-1mm}\delta_7,
\\
\varphi_{456}((a, A), \gamma_{\scriptscriptstyle {H}})\!\!\!&=&\!\!\!\delta_7 (\varphi_{{}_{\text E2}}(a, A)\gamma_{\scriptscriptstyle {H}})\delta_7,
\end{eqnarray*}
where $\varphi_{{}_{\text E2}}, \delta_7$ are defined in Theorem 3.3.2, Lemma 4.5.1, respectively. From Lemmas 4.5.1, 4.5.4,
 we have $\varphi_{456}((a, A), 1),$ $ \varphi_{456}((a, A), \gamma_{\scriptscriptstyle {H}})  \in (E_6)^{\lambda\gamma} \cap (E_6)^{\gamma_{{}_{\scriptscriptstyle {C}}}}$. Hence  $\varphi_{456}$ is well-defined. Using $\gamma_{\scriptscriptstyle {H}}=\varphi_{{}_{\text E2}}(e_1, iI)$ (Lemma 4.5.4 (1)), we can confirm that $\varphi_{456}$ is a homomorphism. Indeed, we show that the case of $\varphi_{456}((a, A), \gamma_{\scriptscriptstyle {H}})\varphi_{456}((b, B), 1)\!=\!\varphi_{456}(((a, A), \gamma_{\scriptscriptstyle {H}})((b,$ $B), 1))$ as example. For the left hand side of this equality , we have that 
\begin{eqnarray*}
\varphi_{456}((a, A), \gamma_{\scriptscriptstyle {H}})\varphi_{456}((b, B), 1)
&\!\!=\!\!&
(\delta_7 (\varphi_{{}_{\text E2}}(a, A)\gamma_{\scriptscriptstyle {H}})\delta_7)(\delta_7 \varphi_{{}_{\text E2}}(b, B) \delta_7)\\
&\!\!=\!\!&
(\delta_7 (\varphi_{{}_{\text E2}}(a, A)\varphi_{{}_{\text E2}}(e_1, iI))\delta_7)(\delta_7 \varphi_{{}_{\text E2}}(b, B) \delta_7)\\
&\!\!=\!\!&
\delta_7 (\varphi_{{}_{\text E2}}(a\ov{b}, A(iI)B(iI)^{-1})\gamma_{\scriptscriptstyle {H}})\delta_7.\\
\end{eqnarray*}

\noindent On the other hand,  for the right hand side of same one, we have that
\begin{eqnarray*}
\varphi_{456}(((a, A), \gamma_{\scriptscriptstyle {H}})((b, B), 1))
&\!\!=\!\!&
\varphi_{456}((a, A)\gamma_{\scriptscriptstyle {H}}(b, B), \gamma_{\scriptscriptstyle {H}})\\
&\!\!=\!\!&
\varphi_{456}((a\ov{b}, A(iI)B(iI)^{-1}),\gamma_{\scriptscriptstyle {H}})\\
&\!\!=\!\!&
\delta_7 (\varphi_{{}_{\text E2}}(a\ov{b}, A(iI)B(iI)^{-1})\gamma_{\scriptscriptstyle {H}})\delta_7,\\
\end{eqnarray*}
\vspace{-9mm}

\noindent that is, $\varphi_{456}((a, A), \gamma_{\scriptscriptstyle {H}})\varphi_{456}((b, B), 1)=\varphi_{456}(((a, A), \gamma_{\scriptscriptstyle {H}})((b, B), 1))$. Similarly, the other cases are shown.

We shall show that $\varphi_{456}$ is surjection. Let $\alpha \in (E_6)^{\lambda\gamma} \cap (E_6)^{\gamma_{{}_{\scriptscriptstyle {C}}}}$. 
Hence, since $\alpha \in (E_6)^{\lambda\gamma} \cap (E_6)^{\gamma_{{}_{\scriptscriptstyle {C}}}} \subset (E_6)^{\gamma_{{}_{\scriptscriptstyle {C}}}} \cong (E_6)^\gamma$ (Proposition 4.5.2 (2)), there exist $p \in S\!p(1)$ and $U \in S\!U(6)$ such that $\alpha =\delta_7 \varphi_{{}_{\text E2}}(p, U)\delta_7$ (Theorem 3.3.2). Moreover, from $\alpha =\delta_7 \varphi_{{}_{\text E2}}(p, U)\delta_7 \in  (E_6)^{\lambda\gamma}$,
that is, $\lambda(\gamma(\delta_7 \varphi_{{}_{\text E2}}(p, U)\delta_7)\gamma)=\delta_7 \varphi_{{}_{\text E2}}(p, U)\delta_7$,  using 
$\gamma_{\scriptscriptstyle {C}}\varphi_{{}_{\rm E2}}(p, U)\gamma_{\scriptscriptstyle {C}}=\varphi_{{}_{\rm E2}}(\gamma_{\scriptscriptstyle {C}} p, -JUJ) $ and $\lambda(\varphi_{{}_{\rm E2}}(p, U) )=\varphi_{{}_{\rm E2}}(p, -J(\tau U)J) $ (Lemma 4.5.4 (2), (3)), we have that $\varphi_{{}_{\text E2}}(\gamma_{\scriptscriptstyle {C}}p, \tau A)=\varphi_{{}_{\text E2}}(p, A)$. Hence it follows that
$$
\left \{
         \begin{array}{l}
                \gamma_{\scriptscriptstyle {C}}p = p
                         \vspace{3mm}\\
                \tau U = U
         \end{array}\right.\qquad \text{or}\qquad 
\left \{         
          \begin{array}{l}
        \gamma_{\scriptscriptstyle {C}}p =- p
                         \vspace{3mm}\\
                 \tau U = -U.
         \end{array}\right. 
$$
In the former case, we see that $ p \in U(1)= \{ a=p_0 +p_2 e_2\,|\,a \ov{a}=1 \}$ and $U \in S\!O(6)$. Hence we have that $\alpha = \delta_7 \varphi_{{}_{\text E2}}(a, A)\delta_7=\varphi_{456}((a, A), 1)$.
In the latter case, we can find the explicit forms of  $p \in S\!p(1), U \in S\!U(6)$ as follows:
$$
  p\! =\!p_1 e_1 +p_2 e_3 \!=\!a e_1\,(a= p_1 - p_2 e_2 \in U(1)),\,\, \, U\!=\!A(iI), A \in S\!O(6).
$$
Hence we have that
\begin{eqnarray*}
  \alpha\!\!\!&=&\!\!\!\delta_7 (\varphi_{{}_{\text E2}}(a e_1, A(iI)))\delta_7=\delta_7 (\varphi_{{}_{\text E2}}(a, A)\varphi_{{}_{\text E2}}(e_1, iI))\delta_7\\
 \!\!\!&=&\!\!\! \delta_7 (\varphi_{{}_{\text E2}}(a, A)\gamma_{\scriptscriptstyle {H}}))\delta_7=\varphi_{456}((a, A), \gamma_{\scriptscriptstyle {H}}).
\end{eqnarray*}
Thus $\varphi_{456}$ is surjection. 

From $\Ker\,\varphi_{{}_{\text E2}}\!=\!\{(1, E), (-1, -E) \}$, we can easily obtain that $\Ker \,\varphi_{456}\!=\! \{((1, E), 1), ((-1,$ $ -E),1) \} \cong (\Z_2, 1)$. 
Therefore we have the following isomorphism 
$
(E_6)^{\lambda\gamma} \cap (E_6)^{\gamma_{{}_{\scriptscriptstyle {C}}}} \cong (U(1) \times S\! O(6))/\Z_2 \rtimes \{1, \gamma_{\scriptscriptstyle {H}} \}.
$ 

Namely, from Proposition 4.5.5, we have the required isomorphism 
$$
(E_6)^{\lambda\gamma} \cap (E_6)^{\lambda\gamma\gamma_{{}_{\scriptscriptstyle {C}}}} \!\!\cong (U(1) \times S\! O(6))/\Z_2 \rtimes \mathcal{Z}_2.
$$
\end{proof}

\subsection{Type EI-I-III}
In this section, we give a pair of involutive  automorphisms $\lambda\tilde{\gamma}$ and $\lambda\tilde{\gamma\sigma}$.


\vspace{1mm}

Using the inclusion $F_4 \subset E_6$, the $\R$-linear transformation $\delta_5$ defined in Lemma 4.3.1 is naturally extended to $C$-linear transformation of $\mathfrak{J}^C$. Hence,  
as in $F_4$, we easily see that $\delta_5 \gamma=(\gamma\sigma)\delta_5$ as $\delta_5 \in F_4 \subset E_6$, that is, $\gamma \sim \gamma\sigma$ in $E_6$. 
\begin{prop} The group $(E_6)^{\lambda\gamma}$ is isomorphic to the group $(E_6)^{\lambda\gamma\sigma}$ {\rm : }$(E_6)^{\lambda\gamma} \cong (E_6)^{\lambda\gamma\sigma}$.
\end{prop}
\begin{proof}
We define a mapping $f : (E_6)^{\lambda\gamma} \to (E_6)^{\lambda\gamma\sigma}$ by
$$
f(\alpha)=\delta_5 \alpha {\delta_5}^{-1},
$$
where $\delta_3$ is same one above. (Remark. since $\delta_5 \in F_4$, it follows that $\lambda(\delta_5)=\delta_5$.)
In order to prove this proposition, it is sufficient to show that the mapping $f$ is well-defined. However, it is almost evident from $\lambda(\delta_5)=\delta_5$ and $\delta_5 \gamma=(\gamma\sigma)\delta_5$. 
\end{proof}

From the result of types EI, EIII in Table 2 and Proposition 4.6.1, we have the following theorem.

\begin{thm} For $\mathbb{Z}_2 \times \mathbb{Z}_2=\{1,\lambda\gamma \} \times \{1, \lambda\gamma\sigma \}$, the $\mathbb{Z}_2 \times \mathbb{Z}_2$ -symmetric space is of type $(E_6/(E_6)^{\lambda\gamma}, E_6/(E_6)^{\lambda\gamma\sigma}, E_6/(E_6)^{(\lambda\gamma)(\lambda\gamma\sigma)})\!=\!(E_6/(E_6)^{\lambda\gamma}\!, $  $E_6/(E_6)^{\lambda\gamma}, E_6/(E_6)^\sigma)$, that is, type {\rm (EI, EI, EIII)},  abbreviated as {\rm  EI-I-III}.
\end{thm}
Here, 
we prove lemma needed and make some preparations for the theorem below.
\begin{lem} The mapping $\varphi_{{}_{\rm{E1}}}:S\!p(4) \to (E_6)^{\lambda\gamma}$ of Theorem 3.3.1 satisfies the following equalities{\rm :} 
\vspace{-7mm}

\begin{eqnarray*}
&&{\rm(1)}\, \,\gamma=\varphi_{{}_{\rm{E1}}}(I_1), \,\,
\sigma=\varphi_{{}_{\rm{E1}}}(I_2).
 \\[0mm]
&&{\rm(2)}\,\, \gamma\varphi_{{}_{\rm{E1}}}(P)\gamma=\varphi_{{}_{\rm{E1}}}(I_1 P I_1),\,\, 
\sigma\varphi_{{}_{\rm{E1}}}(P)\sigma=\varphi_{{}_{\rm{E1}}}(I_2 P I_2).
\\[0mm]
&&{\rm(3)}\,\,\lambda(\varphi_{{}_{\rm{E1}}}(P))=\varphi_{{}_{\rm{E1}}}(I_1 P I_1),
\end{eqnarray*}
where $I_1=\diag(-1, 1, 1, 1), I_2=\diag(-1, -1, 1,1)$.
\end{lem} 
\begin{proof}
The proof of (1) is omitted (see \cite[Lemma 3.4.4]{Yokotaichiro1} in detail). The equalities of (2) are the direct results of (1). As for (3), from $\lambda(\varphi_{{}_{\rm{E1}}}(P))=\tau\varphi_{{}_{\rm{E1}}}(P)\tau$, it is easily obtained.
\end{proof} 

We define some element $\rho \in (E_6)^{\lambda\gamma}$ by
$$
   \rho=\varphi_{{}_\text{E1}}(J_E),
$$
where $J_E=\begin{pmatrix}
                                           0 & E \\
                                           E & 0
                    \end{pmatrix} \in S\!p(4), E=\begin{pmatrix}
                                           1 & 0 \\
                                           0 & 1
                    \end{pmatrix} \vspace{1mm}$.
Then we easily see \vspace{1mm}$\rho^2=1$.

Consider a group  $\mathcal{Z}_2=\{1, \rho \}$. Then the group $\mathcal{Z}_2=\{1, \rho \}$ acts on the group $S\!p(2) \times S\!p(2)$ by 
$$
    \rho(A, B)=(B, A),
$$
and let $(S\!p(2) \times S\!p(2)) \rtimes \mathcal{Z}_2$ be  the semi-product of $S\!p(2) \times S\!p(2)$ and $\mathcal{Z}_2$ with \vspace{2mm}this action. 

Now, we determine the structure of the group $(E_6)^{\lambda\gamma} \cap (E_6)^{\lambda\gamma\sigma}$.
\begin{thm} We have that $(E_6)^{\lambda\gamma} \cap (E_6)^{\lambda\gamma\sigma} \cong (S\!p(2) \times S\!p(2))/\Z_2 \rtimes \mathcal{Z}_2$, where $\Z_2=\{(E, E), (-E, -E)\}, \mathcal{Z}_2=\{1, \rho \}$.
\end{thm}
\begin{proof}
We define a mapping $\varphi_{464}: (S\!p(2) \times S\!p(2)) \rtimes \{1, \rho \} \to (E_6)^{\lambda\gamma} \cap (E_6)^{\lambda\gamma\sigma}$ by
\begin{eqnarray*}
\varphi_{464}((A, B), 1) \!\!\! &=& \!\!\! \varphi_{{}_\text{E1}}(h_1(A, B)),
\\
\varphi_{464}((A, B), \rho) \!\!\! &=& \!\!\! \varphi_{{}_\text{E1}}(h_1(A, B))\,\rho,
\end{eqnarray*}
where $\varphi_{{}_\text{E1}}$ is defined in\vspace{0.5mm} Theorem 3.3.1, and $h_1$ is defined as follows:
$h_1 \!:\!S\!p(2) \times S\!p(2) \to S\!p(4), h_1(A, B)=\begin{pmatrix}A & 0 \\
                                              0 & B
\end{pmatrix}$.
From Lemma 4.6.3 (2), (3), 
we see \vspace{0.5mm}that $\varphi_{464}((A, B), 1), \,\varphi_{464}((A,$ $ B), \rho) \in (E_6)^{\lambda\gamma} $ $\cap \,(E_6)^{\lambda\gamma\sigma}$. Hence $\varphi_{464}$ is well-defined, and using  $\rho=\varphi_{{}_\text{E1}}(J_E)$, it is easily to verify that $\varphi_{464}$ ia a homomorphism. 

We shall show that $\varphi_{464}$ is surjection.
Let $\alpha \in (E_6)^{\lambda\gamma} \cap (E_6)^{\lambda\gamma\sigma}$.
Since $\alpha \in (E_6)^{\lambda\gamma} \cap (E_6)^{\lambda\gamma\sigma} \subset (E_6)^{\lambda\gamma}$, there exists $P \in S\!p(4)$  such that $\alpha =\varphi_{{}_{\text E1}}(P)$ (Theorem 3.3.1). Moreover, from $\alpha =\varphi_{{}_{\text E1}}(P) \in (E_6)^{\lambda\gamma\sigma}$, 
that is, $(\lambda\gamma\sigma) \varphi_{{}_{\text E1}}(P)(\lambda\gamma\sigma)^{-1}=\varphi_{{}_{\text E1}}(P)$, using $\gamma\varphi_{{}_{\rm{E1}}}(P)\gamma=\varphi_{{}_{\rm{E1}}}(I_1 P I_1) , \sigma\varphi_{{}_{\rm{E1}}}(P)\sigma=\varphi_{{}_{\rm{E1}}}(I_2 P I_2)$ and $ \lambda(\varphi_{{}_{\rm{E1}}}(P))=\varphi_{{}_{\rm{E1}}}(I_1 P I_1)$
(Lemma 4.6.3 (2), (3)), we have that $\varphi_{{}_{\text E1}}(I_2 PI_2)=\varphi_{{}_{\text E1}}(P)$. Hence, it follows that
$$
    I_2 PI_2=P \hspace{7mm}{\text{or}}\hspace{7mm}I_2 PI_2=-P.
$$
In the former case, we easily get the explicit form of $P \in S\!p(4)$ as follows: 
$$
P=\begin{pmatrix}A & 0 \\
                  0 & B 
   \end{pmatrix},\,\, A, B \in S\!p(2).
$$
Hence, for $\alpha \in (E_6)^{\lambda\gamma} \cap (E_6)^{\lambda\gamma\sigma}$, we have that 
$$
\alpha = \varphi_{{}_{\text E1}}(\begin{pmatrix}A & 0 \\
                  0 & B 
   \end{pmatrix})= \varphi_{{}_{\text E1}}(h_1(A, B))=\varphi_{464}((A, B), 1).
$$
In the latter case, as the former case, we can also find the explicit form of $P \in S\!p(4)$ as follows:
$$
P=\begin{pmatrix} 0 & C \\
                  D & 0
   \end{pmatrix},\,\, C, D \in S\!p(2).
$$
Hence, for $\alpha \in (E_6)^{\lambda\gamma} \cap (E_6)^{\lambda\gamma\sigma}$, we have that 
\begin{eqnarray*}
\alpha \!\!\!&=&\!\!\! \varphi_{{}_{\text E1}}(\begin{pmatrix} 0 & C \\
                  D & 0
   \end{pmatrix})=\varphi_{{}_{\text E1}}(\begin{pmatrix}C & 0 \\
                  0 & D 
   \end{pmatrix}
    \begin{pmatrix} 0 & E \\
                  E & 0
   \end{pmatrix})=\varphi_{{}_{\text E1}}(\begin{pmatrix}C & 0 \\
                  0 & D 
   \end{pmatrix})\,
   \varphi_{{}_{\text E1}}(
    \begin{pmatrix} 0 & E \\
                  E & 0
   \end{pmatrix})
\\[2mm]   
   \!\!\!&=&\!\!\!\varphi_{{}_{\text E1}}(h_1(C, D))\rho
   =\varphi_{464}((C, D), \rho).
\end{eqnarray*}
Thus $\varphi_{464}$ is surjection. 

From $\Ker\,\varphi_{{}_{\text E1}}=\{E,-E \}$, we can easily obtain that $\Ker\, \varphi_{464}=\{((E,E),1), ((-E, -E), 1)  \}$ $ \cong (\Z_2, 1)$. 

Therefore we have the required isomorphism 
$$
  (E_6)^{\lambda\gamma} \cap (E_6)^{\lambda\gamma\sigma} \cong (S\!p(2) \times S\!p(2))/\Z_2 \rtimes \mathcal{Z}_2.
$$
\end{proof}

\subsection{Type EI-II-IV}
In this section, we give a pair of involutive  automorphisms $\lambda\tilde{\gamma}$ and $\tilde{\gamma}$.
\vspace{1mm}

From the result of type EI, EII, EIV in Table 2, we have the following theorem.
\begin{thm} For $\mathbb{Z}_2 \times \mathbb{Z}_2=\{1,\lambda\gamma \} \times \{1, \gamma \}$, the $\mathbb{Z}_2 \times \mathbb{Z}_2$ -symmetric space is of type $(E_6/(E_6)^{\lambda\gamma}, E_6/(E_6)^{\gamma}, E_6/(E_6)^{(\lambda\gamma)\gamma})\!=\!(E_6/(E_6)^{\lambda\gamma}, E_6/(E_6)^{\gamma}, E_6/(E_6)^\lambda)$, that is, type {\rm (EI, EII, EIV)},  abbreviated as {\rm  EI-II-IV}.
\end{thm}
Now, we determine the structure of the group $(E_6)^{\lambda\gamma} \cap (E_6)^{\gamma}$.
\begin{thm} We have that $(E_6)^{\lambda\gamma} \cap (E_6)^{\gamma} \cong (S\!p(1) \times S\!p(3))/\Z_2,\,\Z_2=\{(1,E), (-1, -E)\}$.
\end{thm}
\begin{proof}
We define a mapping $\varphi_{472}:S\!p(1) \times S\!p(3) \to (E_6)^{\lambda\gamma} \cap (E_6)^{\gamma}$ by
$$
\varphi_{472}(p, A)=\varphi_{{}_\text{E1}}(h_2(p, A)),
$$
where $h_2$ is defined by $h_2:S\!p(1) \times S\!p(3) \to S\!p(4),\,h_2(p, A)=\begin{pmatrix} 
                                      p & 0 \\
                                      0 & A
                    \end{pmatrix}$. 
Since the mapping $\varphi_{472}$ is the restriction of the mapping $\varphi_{{}_\text{E1}}$, it is easily to verify that $\varphi_{472}$ is well-defined and a homomorphism. 

We shall show that $\varphi_{472}$ is surjection. Let $\alpha \in (E_6)^{\lambda\gamma} \cap (E_6)^{\gamma}$. Since $\alpha \in (E_6)^{\lambda\gamma} \cap (E_6)^{\gamma} \subset (E_6)^{\lambda\gamma}$, there exists $P \in S\!p(4)$  such that $\alpha =\varphi_{{}_{\text E1}}(P)$ (Theorem 3.3.1). 
Moreover, from $\alpha =\varphi_{{}_{\text E1}}(P) \in (E_6)^{\gamma}$, that is, $\gamma\varphi_{{}_{\text E1}}(P)\gamma=\varphi_{{}_{\text E1}}(P)$, using $\gamma\varphi_{{}_{\rm{E1}}}(P)\gamma=\varphi_{{}_{\rm{E1}}}(I_1 P I_1)$ (Lemma 4.6.3 (2)) we have that $\varphi_{{}_{\text E1}}(I_1 PI_1)=\varphi_{{}_{\text E1}}(P)$. Hence it follows that 
$$
    I_1 PI_1=P \hspace{7mm}{\text{or}}\hspace{7mm}I_1 PI_1=-P.
$$
In the former case,  we easily get the explicit form of $P \in S\!p(4)$ as follows:
$$
P=\begin{pmatrix} p & 0 \\
                  0 & A 
   \end{pmatrix},\,\, p \in S\!p(1),A \in S\!p(3).
$$
Hence for $\alpha \in (E_6)^{\lambda\gamma} \cap (E_6)^{\lambda\gamma\sigma}$, we have that 
$$
\alpha = \varphi_{{}_{\text E1}}(\begin{pmatrix}p & 0 \\
                  0 & A 
   \end{pmatrix})= \varphi_{{}_{\text E1}}(h_2(p, A))=\varphi_{472}((p, A)).
$$
In the latter case,  as the former case, we can also find the explicit form of $P \in S\!p(4) $ as follows:
$$
P=\begin{pmatrix} 0 & b & c & d \\
                  l & 0 & 0 & 0 \\
                  m & 0 & 0 & 0 \\
                  n & 0 & 0 & 0 
   \end{pmatrix}, \,b, c, d, l, m, n \in \H.
$$
This is contrary to the condition $P \in S\!p(4)$ because of  $b\!=\!c\!=\!d\!=0$. Hence this case is impossible.
Thus $\varphi_{472}$ is surjection.

From $\Ker\, \varphi_{{}_{\text E1}}\!=\!\{ (1, E), (-1, -E)\}$, we can easily obtain that $\Ker\, \varphi_{472}\!=\!\{ (1, E), (-1, -E)\} $ $\cong \Z_2$.

Therefore we have the required isomorphism 
$$
  (E_6)^{\lambda\gamma} \cap (E_6)^\gamma \cong (S\!p(1) \times S\!p(3))/\Z_2.
$$                   
\end{proof} 

\subsection{Type EII-II-II}

In this section, we give a pair of involutive inner automorphisms $\tilde{\gamma}$ and $\tilde{ \gamma_{\scriptscriptstyle {H}}}$.
\vspace{1mm}

Again,let the $C$-linear transformation $\gamma_{{}_{\scriptscriptstyle {H}}}$ of $\mathfrak{J}^C$. As $\gamma_{{}_{\scriptscriptstyle {C}}}$ in section 4.5, 
$\gamma_{{}_{\scriptscriptstyle {H}}}$ induces involutive inner automorphism $\tilde{ \gamma_{{}_{\scriptscriptstyle {H}}}}$ of $E_6$: $\tilde{ \gamma_{{}_{\scriptscriptstyle {H}}}}(\alpha)= \gamma_{{}_{\scriptscriptstyle {H}}}\alpha\gamma_{{}_{\scriptscriptstyle {H}}}, \alpha \in E_6$.

\noindent Using the inclusion $F_4 \subset E_6$, the $\R$-linear transformations $\delta_1, \delta_2$ defined in the proof of Lemma 4.1.1 are naturally extended to the $C$-linear transformations of $\mathfrak{J}^C$. Obviously, we have $\delta_1,$ $ \delta_2 \in E_6, {\delta_1}^2={\delta_2}^2=1$. Hence,  
as in $G_2$ and $F_4$, since we easily see that $\delta_1 \gamma=\gamma_{{}_{\scriptscriptstyle {H}}}\delta_1, \delta_2 \gamma=(\gamma\gamma_{{}_{\scriptscriptstyle {H}}})\delta_2$
, that is, 
$\gamma \sim \gamma_{{}_{\scriptscriptstyle {H}}}, \gamma \sim \gamma \gamma_{{}_{\scriptscriptstyle {H}}}$ in $E_6$, we have the following proposition. 
\vspace{1mm}

\begin{prop}
The group $(E_6)^\gamma$ is isomorphic to both of the groups $(E_6)^{\gamma_{{}_{\scriptscriptstyle {H}}}}$ and $(E_6)^{\gamma \gamma_{{}_{\scriptscriptstyle {H}}}}${\rm:} \\$(E_6)^\gamma \cong (E_6)^{\gamma_{{}_{\scriptscriptstyle {H}}}} \cong (E_6)^{\gamma \gamma_{{}_{\scriptscriptstyle {H}}}}$.
\end{prop}
From the result of type EII in Table 2 and Proposition 4.8.1, we have the following theorem.

\begin{thm} For $\mathbb{Z}_2 \times \mathbb{Z}_2=\{1,\gamma \} \times \{1, \gamma_{\scriptscriptstyle {H}} \}$, the $\mathbb{Z}_2 \times \mathbb{Z}_2$-symmetric space is of type $(E_6/(E_6)^\gamma, E_6/(E_6)^{\gamma_{{}_{\scriptscriptstyle {H}}}}, E_6/(E_6)^{\gamma\gamma_{{}_{\scriptscriptstyle {H}}}})$, that is, type {\rm (EII, EII, EII)}, abbreviated as {\rm EII-II-II}.
\end{thm}
Here, we prove proposition needed and make some preparations for the theorem\vspace{2mm} below.

We define spaces $G_{3,3}$ and $G^{-}_{3,3}$ by 
$\{ U\!\in \! S\!U(6)\,|\, IUI\!=\!U \}$ and $\{ U \! \in \! S\!U(6)\,|\, IUI\!=\!-U \}$: 
$$
G_{3,3}=\{ U \in S\!U(6)\,|\, IUI=U \},\quad
G^{-}_{3,3}= \{ U \in S\!U(6)\,|\, IUI=-U \},
$$ 
respectively, 
where $I=\diag(1, -1, 1, -1, 1, -1)$.
\vspace{2mm}

Moreover, we consider some \vspace{1mm}element  $M_1=
{\scriptstyle 
\begin{pmatrix}           1&0&0&0&0&0     \\
                                    0&0&0&0&-1&0     \\
                                    0&0&1&0&0&0    \\
                                    0&0&0&1&0&0     \\
                                    0&1&0&0&0&0     \\
                                    0&0&0&0&0&1
\end{pmatrix}}
\in \!S\!O(6) \subset S\!U(6)$ such that $IM_1=M_1 I_3$, where $I_3=\diag(1,1,1,-1,-1,-1)$. 
\begin{prop} We have the following isomorphisms{\rm:}
\vspace{2mm}

{\rm (1)}\,$G_{3,3} \cong S(U(3) \times U(3))${\rm(}as a group{\rm)},
where $S(U(3) \times U(3))\!=\!\{ U \in S\!U(6)\,|\, I_3 U I_3 \!=\!U \}$.
\vspace{-1mm}

{\rm (2)}\,$G^{-}_{3,3} \cong S(U(3) \times U(3))^-${\rm(}as a set{\rm)},
where $S(U(3) \times U(3))^-\!=\!\{ U \in S\!U(6)\,|\,I_3 U I_3 \!=\!-U \}$.
\vspace{-1mm}

{\rm (3)}\,$S\!(U(3) \times U(3)) \cong (U(1) \times  S\!U(3) \times S\!U(3))/\Z_3, \,\Z_3=\{ (1,E,E), (\omega_1, {\omega_1}^2, \omega_1 E), $  $({\omega_1}^2, \omega_1, {\omega_1}^2 E)\}$, where $\omega \in C, \omega^3  =1, \omega \ne 1$.
\end{prop}
\begin{proof}
(1)\,We define a homomorphism $k_1:S(U(3) \times U(3)) \to G_{3,3}$ by 
$$
  k_1(U)=M_1 U {M_1}^{-1}.
$$
Then it is easily to verify that $k_1$ is isomorphism as a Lie group.
\vspace{2mm}

(2)\,We define a mapping $k^{-}_1:S(U(3) \times U(3))^{-} \to G^{-}_{3,3}$ by 
$$
  k^{-}_1(U^{-})=M_1 U^{-} {M_1}^{-1}.
$$
Then it is easily to verify that $k^{-}_1$ is isomorphism as a set.
\vspace{2mm}

(3)\,We define a homomorphism $h_3:U(1) \times  S\!U(3) \times S\!U(3) \to S\!(U(3) \times U(3))$ by
$$
  h_3(\theta, A, B)=
          \begin{pmatrix} \theta A & 0 \\
                              0    & \theta^{-1} B
          \end{pmatrix}. 
$$
Then we can easily show that the mapping $h_3$ induces the required isomorphism.
\end{proof} 
We define some element $\rho_1 \in (E_6)^{\gamma}$ by
$$
  \rho_1=\varphi_{{}_\text{E2}}(e_2,\,\begin{pmatrix}
                              0 & E  \\
                              -E & 0  
              \end{pmatrix}\,),\,\,
E=\diag(1, 1, 1) \in M(3, \R),
$$
where $\varphi_{{}_\text{E2}}$ is defined in Theorem 3.3.2. Hereafter, \vspace{0.5mm}we denote the matrix $\begin{pmatrix}
                              0 & E  \\
                              -E & 0  
              \end{pmatrix}$ by $E_J$.
Then we remark that $E_J$ commutes with $M_1$:\,$E_J M_1=M_1 E_J$. 
\vspace{2mm}

Consider a group $\mathcal{Z}_2=\{1,\rho_1 \}$. Then the group $\mathcal{Z}_2=\{1,\rho_1 \}$ acts on the group $U(1)$ $ \times (U(1) \times S\!U(3) \times S\!U(3))$ by
$$
   \rho_1(a,(\theta, A, B))=(\ov{a},(\theta^{-1},B, A))
$$
and let $U(1) \times (U(1) \times S\!U(3) \times S\!U(3)) \rtimes \mathcal{Z}_2$ be the semi-direct product of $U(1) \times (U(1) \times S\!U(3) \times S\!U(3))$ and $\mathcal{Z}_2$ with this action. 
\vspace{2mm}

Now, we determine the structure of the group $(E_6)^\gamma \cap (E_6)^{\gamma_{{}_{\scriptscriptstyle {H}}}}$.
\begin{thm} We have that 
$
(E_6)^\gamma \cap (E_6)^{\gamma_{{}_{\scriptscriptstyle {H}}}} \cong (U(1) \times (U(1) \times  S\!U(3) \times S\!U(3)))/(\Z_2 \times \Z_3) \rtimes \mathcal{Z}_2,
$
$\Z_2= \{(1,1,E,E), (-1, -1, -E, -E)\}, \Z_3=\{ (1,1,E,E), (1, \omega, {\omega}^2, \omega E), (1, {\omega}^2, \omega,  $ ${\omega}^2 E)\}, \mathcal{Z}_2=\{1,\rho_1 \}$, where $\omega \in C, {\omega}^3=1, \omega \ne 1$. 
\end{thm}
\begin{proof}
We define a mapping $\varphi_{484}:(U(1) \times (U(1) \times  S\!U(3) \times S\!U(3)) \rtimes \{1,\rho_1 \} \to (E_6)^\gamma \cap (E_6)^{\gamma_{{}_{\scriptscriptstyle {H}}}}$ by
\begin{eqnarray*}
\varphi_{484}((a, (\theta, A, B)),1)\!\!\!&=&\!\!\!\varphi_{{}_\text{E2}}(a, k_1 h_3(\theta, A, B)),
\\[0mm]
\varphi_{484}((a, (\theta, A, B)), \rho_1)\!\!\!&=&\!\!\!\varphi_{{}_\text{E2}}(a, k_1 h_3(\theta, A, B))\,\rho_1,
\end{eqnarray*}
where $k_1, h_3$  are defined in Proposition 4.8.3 (1), (3), respectively. 
From 
Lemma 4.5.4 (2), we \vspace{0.5mm}have $\varphi_{484}((a, (\theta, A, B)),1),$ $ \varphi_{484}((a, (\theta, A, B)), \rho_1) \in (E_6)^\gamma \cap (E_6)^{\gamma_{{}_{\scriptscriptstyle {H}}}}$. Hence $\varphi_{482}$ is well-defined. 
Using $
  \rho_1\!=\!\varphi_{{}_\text{E2}}(e_2,\begin{pmatrix}
                              0 & E  \\
                              -E & 0  
              \end{pmatrix})
$, we can \vspace{0.5mm}confirm that $\varphi_{484}$ is a homomorphism. Indeed, we show that the case of $\varphi_{484}((a, (\theta, A, B)), \rho_1)\varphi_{484}((b, (\nu,  C, D)),\!1)$ $=\varphi_{484}(((a, (\theta, A,$ $ B)), \rho_1 )((b, (\nu, C, D)),1))$ as example. For the left hand side of this equality, we have that
\begin{eqnarray*}
\!\!\!& &\!\!\!\varphi_{484}((a, (\theta, A, B)), \rho_1)\varphi_{484}((b, (\nu, C, D)),1)
\\
\!\!\!&=&\!\!\!(\varphi_{{}_\text{E2}}(a, k_1 h_3(\theta, A, B))\,\rho_1)(\varphi_{{}_\text{E2}}(b, k_1 h_3(\nu, C, D)))
\\
\!\!\!&=&\!\!\!\varphi_{{}_\text{E2}}(a, M_1 h_3(\theta, A, B)M_1)\, \varphi_{{}_\text{E2}}(e_2,\,E_J)\,(\varphi_{{}_\text{E2}}(b, M_1 h_3(\nu, C, D)M_1))
\\
\!\!\!&=&\!\!\!\varphi_{{}_\text{E2}}((ae_2)b,M_1 h_3(\theta, A, B)M_1\,E_J\,M_1 h_3(\nu, C, D)M_1)
\\
\!\!\!&=&\!\!\!\varphi_{{}_\text{E2}}((a\ov{b})e_2,M_1 h_3(\theta, A, B)h_3(\nu^{-1}, D, C)M_1 E_J)
\\
\!\!&=&\!\!\!\varphi_{{}_\text{E2}}(a\ov{b},k_1 h_3(\theta\nu^{-1}, AD, BC))\rho_1.
\end{eqnarray*} 
On the other hand,  for the right hand side of same one, we have that
\begin{eqnarray*}
\!\!&&\!\!\!\varphi_{484}(((a, (\theta, A, B)), \rho_1 )((b, (\nu, C, D)),1))
\\
\!\!&=&\!\!\!\varphi_{484} ((a,(\theta, A, B))\,\rho_1(b,(\nu, C, D)), \rho_1)
\\
\!\!&=&\!\!\!\varphi_{484}(a\ov{b},(\theta, A, B)(\nu^{-1}, D, C), \rho_1)
\\
\!\!&=&\!\!\!\varphi_{484}(a\ov{b},(\theta\nu^{-1}, AD, BC), \rho_1)
\\
\!\!&=&\!\!\!\varphi_{{}_\text{E2}}(a\ov{b},k_1 h_3(\theta\nu^{-1}, AD, BC))\rho_1,
\end{eqnarray*} 
that is,  
$
\varphi_{484}((a, (\theta, A, B)), \rho_1)\varphi_{484}((b, (\nu, C, D)),\!1)\!\!=\!\!\varphi_{484}(((a, (\theta, A, B)), \rho_1 )((b, (\nu, C, D)),\!1))$. Similarly, the other cases are shown.

We shall show that $\varphi_{484}$ is surjection. Let $\alpha \in (E_6)^\gamma \cap (E_6)^{\gamma_{{}_{\scriptscriptstyle {H}}}}$. Since $\alpha \in (E_6)^\gamma \cap (E_6)^{\gamma_{{}_{\scriptscriptstyle {H}}}} \subset (E_6)^\gamma$, ther exist $p \in S\!p(1)$ and $U \in S\!U(6)$ such that $\alpha = \varphi_{{}_\text{E2}}(p, U)$ (Theorem 3.3.2). 
Moreover, from $\alpha = \varphi_{{}_\text{E2}}(p, U) \in (E_6)^{\gamma_{{}_{\scriptscriptstyle {H}}}}$, that is, $\gamma_{{}_{\scriptscriptstyle {H}}} \varphi_{{}_\text{E2}}(p, U) \gamma_{{}_{\scriptscriptstyle {H}}}=\varphi_{{}_\text{E2}}(p, U)$, 
using $\gamma_{\scriptscriptstyle {H}} \varphi_{{}_\text{E2}}(p, U) \gamma_{\scriptscriptstyle {H}}=\varphi_{{}_\text{E2}}(\gamma_{\scriptscriptstyle {H}} p, IUI)$ (Lemma 4.5.4 (2)), we have  $\varphi_{{}_\text{E2}}(\gamma_{\scriptscriptstyle {H}} p, IUI)=\varphi_{{}_\text{E2}}(p, U)$. Hence it follows that
$$
\left \{
         \begin{array}{l}
                \gamma_{\scriptscriptstyle {H}}p = p
                         \vspace{3mm}\\
                IUI = U
         \end{array}\right.\qquad \text{or}\qquad 
\left \{         
          \begin{array}{l}
        \gamma_{\scriptscriptstyle {H}}p =- p
                         \vspace{3mm}\\
                 IUI = -U.
         \end{array}\right. 
$$
 In the former case, we see that $p \in U(1)$ and $U \in G_{3,3}$. Since there exist $\theta \in U(1)$ and $A, B \in S\!U(3)$ such that $U=k_1 h_3(\theta, A, B)$ for $U \in G_{3,3}$ (Proposition 4.8.3 (1), (3)), 
we have that $\alpha = \varphi_{{}_\text{E2}}(a,k_1 h_3 (\theta, A, B))=\varphi_{484}(a,(\theta, A, B), 1)$.
 In the latter case, first we get the explicit form of $p \in S\!p(1)$ as follows:
$$
  p=p_2 e_2+p_3 e_3=b e_2\,(b=p_2+p_3e_1 \in U(1)),
$$
moreover since $U \in G^{-}_{3,3}$,  there exists $U^{-} \in S(U(3) \times U(3))^-$ such that $U=k^{-}_1(U^-)$ (Proposition 4.8.3 (2)), that is, there exist $C, D \in U(3)$ such that $ U=k^{-}_1(\begin{pmatrix} 0 & C \\
                 D & 0
 \end{pmatrix}), (\det C)(\det D)$ $=-1$. 
Hence, from $\begin{pmatrix} 0 & C \\
                 D & 0
 \end{pmatrix}=\begin{pmatrix} C & 0 \\
                 0 & -D
 \end{pmatrix}\,E_J$, we have that 
\begin{eqnarray*}
U\!\!\!&=&\!\!\!k^{-}_1(\begin{pmatrix} C & 0 \\
                 0 & -D
 \end{pmatrix}\,\begin{pmatrix} 0 & E \\
                 -E & 0
 \end{pmatrix})=M_1(\begin{pmatrix} C & 0 \\
                 0 & -D
 \end{pmatrix}\,E_J) {M_1}^{-1}
 \\[1mm]
\!\!\!&=&\!\!\! 
(M_1\begin{pmatrix} C & 0 \\
                 0 & -D
 \end{pmatrix}{M_1}^{-1})(M_1 E_J {M_1}^{-1}) \cdots {\rm (*)}.
\end{eqnarray*}
Here, since $\begin{pmatrix} C & 0 \\
                 0 & -D
 \end{pmatrix} \in S(U(3) \times U(3))$ and $M_J E_J=E_J M_1$, we modify ${\rm (*)}$ above \vspace{1mm}as follows:
 $
 {\rm (*)}=k_1(\begin{pmatrix} C & 0 \\
                 0 & -D
 \end{pmatrix})\,E_J.
 $
Hence, since there exist $\nu \in U(1)$ and $L, N \in S\!U(3)$ \vspace{1mm}such that $U=k_1 h_3(\nu, L, N)\,E_J$ (Proposition 4.8.3 (3)), we have that 
\begin{eqnarray*}
\alpha \!\!\!&=&\!\!\! \varphi_{{}_\text{E2}}(be_2,k_1 h_3 (\nu, L, N)E_J)=\varphi_{{}_\text{E2}}(b,k_1 h_3 (\nu, L, N))\,\varphi_{{}_\text{E2}}(e_2,E_J)
\\[1mm]
\!\!\!&=&\!\!\!\varphi_{{}_\text{E2}}(b,k_1 h_3 (\nu, L, N))\,\rho_1=\varphi_{484}((b,(\nu, L, N)), \rho_1).
\end{eqnarray*}
Thus $\varphi_{484}$ is surjection.

From $\Ker\, \varphi_{{}_\text{E2}}=\{(1, E), (-1, -E)\}$ and $\Ker \,h_3=\{ (1,E,E), (\omega, {\omega}^2,$ $ \omega E), ({\omega}^2, \omega, {\omega}^2 E)\}$, we can easily obtain that 
\begin{eqnarray*}
\Ker \,\varphi_{484}
\!\!\!&=&\!\!\!\{(1,(1,E,E),1), (-1, (-1, -E, -E),1)\}
\\[1mm] 
 \!\!\!&&\!\!\! \hspace*{10mm} \times \{ (1,(1,E,E),1), (1, (\omega, {\omega}^2, \omega E),1), (1, ({\omega}^2, \omega, {\omega}^2 E),1)\}
\\[1mm]
 \!\!\!& \cong &\!\!\!(\Z_2 \times \Z_3,1),
\end{eqnarray*}
where $ \omega \in C, {\omega}^3=1,\omega \ne 1$.

Therefore we have the required isomorphism
$$
(E_6)^\gamma \cap (E_6)^{\gamma_{{}_{\scriptscriptstyle {H}}}} \cong (U(1) \times (U(1) \times  S\!U(3) \times S\!U(3)))/(\Z_2 \times \Z_3) \rtimes \mathcal{Z}_2.
$$
\end{proof} 

\subsection{Type EII-II-III}

In this section, we use a pair of involutive inner automorphisms $\tilde{\gamma}$ and $\tilde{\gamma\sigma}$.
\vspace{1mm}

Since $\gamma \sim \gamma\sigma$ in $E_6$ as mentioned in Section 4.6, we have the following proposition which is the direct result of this.

\begin{prop} 
The group $(E_6)^\gamma$ is isomorphic to the group $(E_6)^{\gamma\sigma}${\rm:} $(E_6)^\gamma \cong (E_6)^{\gamma\sigma}$.
\end{prop}
From the result of types EII, EIII in Table 2 and Proposition 4.9.1, we have the following theorem.

\begin{thm} For $\mathbb{Z}_2 \times \mathbb{Z}_2=\{1,\gamma \} \times \{1, \gamma\sigma \}$, the $\mathbb{Z}_2 \times \mathbb{Z}_2$ -symmetric space is of type $(E_6/(E_6)^{\gamma}, E_6/(E_6)^{\sigma\gamma}, E_6/(E_6)^{(\gamma)(\gamma\sigma)})=(E_6/(E_6)^{\gamma}, E_6/(E_6)^{\gamma}, E_6/(E_6)^\sigma)$, that is, type {\rm (EII, EII, EIII)},  abbreviated as {\rm  EII-II-III}.
\end{thm} 
Here, we prove Proposition needed in theorem below.
\begin{prop} We have the following isomorphism{\rm :} 
$
 (U(1) \times S\!p(1) \times S\!U(4))/\Z_4 \cong  S(U(2) \times U(4)),
$
$
\Z_4=\{(1,1,E),(-1, -1, -E), (i, -1, iE), (-i, -1, -iE) \}.
$
\end{prop}
\begin{proof}
We define a mapping
\begin{eqnarray*}
f:U(1) \times S\!p(1) \times S\!U(4)
\stackrel{k} {\to}& \!\! U(1) \times S\!U(2) \times S\!U(4)
\stackrel{h_4}{\to} &\!\! S(U(2) \times U(4)) 
\end{eqnarray*}
by 
$$
f(\theta, q, A)= h_4(\theta, k(q),A),$$
where\vspace{1mm} a isomorphism $k$ is defined by $k:S\!p(1) \to S\!U(2), \,k(a+be_2)=\begin{pmatrix}      
                                                                                                   a' & b' \\
                                                                                        -\tau{b'} & \tau{a'}
                                                                           \end{pmatrix}$ and a \vspace{1mm}homomorphism $h_4: U(1) \times S\!U(2) \times S\!U(4) \to S(U(2) \times U(4))$ by $h_4(\theta, C, A)=\begin{pmatrix}
                                                       \theta^2 C & 0 \\
                                                            0        & \theta^{-1} A
                                   \end{pmatrix}$. (Remark. For $a=a_1 +a_2 e_1 \in \C$, we express $a'$ as the components \vspace{1mm}replacing $e_1$ by $i$, that is, $a'=a_1+a_2 i$.  It is similar to $a'$ as for $b'$, so is the components of $S\!U(4)$.)
We can show easily that the homomorphism $f$ induces the required isomorphism. 
\end{proof}
Now, we determine the structure of the group $(E_6)^\gamma \cap (E_6)^{\gamma\sigma}$. 
\begin{thm}\,We have that
$
(E_6)^\gamma \cap (E_6)^{\gamma\sigma} \cong (S\!p(1) \times (U(1) \times S\!p(1) \times S\!U(4)))/(\Z_2 \times  \Z_4),  \Z_2\!=\!\{(1, (1, 1, E)),
$
$ 
 (-1, (-1, -1, -E)) \}, \,\Z_4\!=\!\{(1,(1,1,E)),(1,(-1, -1, -E)), (1,(i, -1,$ $ iE)),  (1,(-i, -1, -iE)) \}$.
\end{thm}
\begin{proof}
We define a mapping $\varphi_{494}:S\!p(1) \times (U(1) \times S\!p(1) \times S\!U(4)) \to (E_6)^\gamma \cap (E_6)^{\gamma\sigma}$ by
$$
\varphi_{494} (p, (\theta, q, A))=\varphi_{{}_\text{E2}} (p, f (\theta, q ,A) ), 
$$
where $f$ is defined in Proposition 4.9.3. Since the mapping $\varphi_{494}$ is the restriction of  the mapping  $\varphi_{{}_\text{E2}}$,  it is easily to verify that $\varphi_{494}$ is well-defined and a homomorphism. 

We shall show that $\varphi_{494}$ is surjection. Let $\alpha \in (E_6)^\gamma \cap (E_6)^{\gamma\sigma}$. Since $(E_6)^\gamma \cap (E_6)^{\gamma\sigma} \subset (E_6)^\gamma$, there exist $p \in S\!p(1)$ and $U \in S\!U(6)$ such that $\alpha = \varphi_{{}_\text{E2}}(p, U)$ (Theorem 3.3.2). Moreover, from $\alpha = \varphi_{{}_\text{E2}}(p, U) \in (E_6)^{\gamma\sigma}$, that is, $(\gamma\sigma) \varphi_{{}_\text{E2}}(p, U) (\sigma\gamma)=\varphi_{{}_\text{E2}}(p, U)$, 
using $\gamma\varphi_{{}_{\rm E2}}(p, U)\gamma= \varphi_{{}_{\rm E2}}(p, U)$ and $\sigma \varphi_{{}_{\rm E2}}(p, U) \sigma=\varphi_{{}_{\rm E2}}(p, I_2 U I_2)$ (Lemma 4.5.4 (2)), 
we easily see $\varphi_{{}_\text{E2}}(p, I_2 U I_2)=\varphi_{{}_\text{E2}}(p, U)$. 
Hence it follows that
$$
\left \{
         \begin{array}{l}
                p = p
                         \vspace{3mm}\\
                I_2 U I_2 = U
         \end{array}\right.\qquad \text{or}\qquad 
\left \{         
          \begin{array}{l}
                  p =- p
                         \vspace{3mm}\\
                 I_2 U I_2 = -U.
         \end{array}\right. 
$$
In the former case, we see that $p \in S\!p(1)$ and $U \in S(U(2) \times U(4))$. Moreover for $U \in S(U(2) \times U(4))$, there exist $\theta \in U(1), q \in S\!p(1)$ and $ A \in S\!U(4)$ such that $U=f(\theta, q, A)$ (Proposition 4.9.3). Hence we have that $\alpha =\varphi_{{}_\text{E2}}(p,  f(\theta, q, A)) =\varphi_{494}(p,  (\theta, q, A))$. In the latter case, 
this is contrary to the condition $p \in S\!p(1)$ because of $p=0$. Hence this case is impossible. Thus  $\varphi_{494}$ is surjection.

From $\Ker\, \varphi_{{}_\text{E2}}\!=\!\{(1, E), (-1, -E)\}$ and $\Ker \, f\!=\!\{(1,1,E),(-1, -1, -E), (i, -1, iE),(-i, -1, $ $ -iE) \}$, we have easily obtain that 
\begin{eqnarray*}
\Ker\, \varphi_{494}\!\!\!&=&\!\!\!\{(1, (1, 1, E)), (-1, (-1, -1, -E)) \}\times \{(1,(1,1,E)), (1,(-1, -1, -E)),  
\\
&& (1,(i, -1, iE)),  (1,(-i, -1, -iE)) \} \cong \Z_2 \times \Z_4.
\end{eqnarray*}

Therefore we have the required isomorphism 
$$
(E_6)^\gamma \cap (E_6)^{\gamma\sigma} \cong (S\!p(1) \times (U(1) \times S\!p(1) \times S\!U(4)))/(\Z_2 \times  \Z_4)\vspace{-3mm}.
$$
\end{proof}

\subsection{Type EII-III-III}

In this section, we give a pair of involutive inner automorphisms $\tilde{\gamma}$ and $\tilde{\gamma_{\scriptscriptstyle {H}} \rho_2}$, where
$\tilde{\gamma_{\scriptscriptstyle {H}}\rho_2}$ is induced by a $C$-linear transformation $\gamma_{\scriptscriptstyle {H}}\rho_2$ of $\mathfrak{J}^C$: $\tilde{\gamma_{\scriptscriptstyle {H}} \rho_2}(\alpha)=(\gamma_{\scriptscriptstyle {H}}\rho_2) \alpha (\rho_2\gamma_{\scriptscriptstyle {H}}) , \alpha \in E_6$, and 
a $C$-linear transformation $\rho_2$ of $\mathfrak{J}^C$is defined below.
\vspace{1mm}

We define some element $\rho_2 \in (E_6)^\gamma$ by
$$
     \rho_2=\varphi_{{}_\text{E2}}(1, L_2),
$$
where ${L}_2=\diag(1, 1, -1, 1, -1, 1) \in S\!O(6) \subset S\!U(6)$, and the explicit form of $\rho_2$ as action to ${\mathfrak{J}}^C$ is given by
$$
\rho_2 X
=\begin{pmatrix}
                     \xi_1  &   -i x_3 e_1  &    i \,\ov{x}_2e_1         \\
                     i e_1\ov{x}_3   &   -\xi_2      &  e_1 x_1 e_1  \\
                          i e_1 x_2     &   e_1 \ov{x}_1 e_1  &  -\xi_3
       \end{pmatrix}, \, X \in \mathfrak{J}^C.
$$
Then we have that ${\rho_2}^2=1, \delta_1 \rho_2 =\rho_2 \delta_1, \delta_2 \rho_2 =\rho_2 \delta_2$,  
where $\delta_1, \delta_2$ are the same ones used in Section 4.8.
\vspace{1mm}

Now, for
 $D_8=
{\scriptstyle 
\begin{pmatrix}           0&0&0&0&1&0     \\
                                    0&0&1&0&0&0     \\
                                    0&-1&0&0&0&0    \\
                                    0&0&0&1&0&0     \\
                                   -1&0&0&0&0&0     \\
                                    0&0&0&0&0&1
     \end{pmatrix}}
\in \!S\!O(6) \subset S\!U(6)$
, we consider \vspace{1mm}some element $\varphi_{{}_\text{E2}}(1, D_8) \in (E_6)^\gamma$, and we denote $\varphi_{{}_\text{E2}}(1, D_8)$ by $\delta_8$: $\delta_8 =\varphi_{{}_\text{E2}}(1, D_8)$. 
Then since $\gamma\sigma=\varphi_{{}_\text{E2}}(1, {I_2})$, we have $\delta_8 (\gamma\sigma)=\rho_2 \delta_8 $. Hence, together with $\gamma\rho_2 = \rho_2 \gamma$, we have ${\delta_8}^{-1} (\gamma\rho_2 )=\sigma {\delta_8}^{-1}$.
\begin{lem} In $E_6$, $\sigma$ is conjugate to both of $\gamma_{\scriptscriptstyle {H}} \rho_2$ and $\gamma\gamma_{\scriptscriptstyle {H}} \rho_2${\rm:}  $\sigma \sim \gamma_{\scriptscriptstyle {H}} \rho_2, \sigma \sim \gamma\gamma_{\scriptscriptstyle {H}} \rho_2$.
\end{lem}
\begin{proof}
Using ${\delta_8}^{-1} (\gamma\rho_2 )=\sigma {\delta_8}^{-1}, \delta_k \rho_2=\rho_2 \delta_k, k=1,2$, we have that
\begin{eqnarray*}
\!\!\!&&\!\!\!({\delta_8}^{-1} \delta_1)(\gamma_{\scriptscriptstyle {H}} \rho_2)
={\delta_8}^{-1} (\delta_1\gamma_{\scriptscriptstyle {H}}) \rho_2
={\delta_8}^{-1}(\gamma\delta_1) \rho_2
\\[0mm]
\!\!\!&=&\!\!\!{\delta_8}^{-1}\gamma(\delta_1 \rho_2)
={\delta_8}^{-1}\gamma(\rho_2 \delta_1)
=({\delta_8}^{-1}(\gamma\rho_2)) \delta_1
\\[0mm]
\!\!\!&=&\!\!\!(\sigma {\delta_8}^{-1})\delta_1=\sigma ({\delta_8}^{-1}\delta_1),
\end{eqnarray*}
that is, $\sigma \sim \gamma_{\scriptscriptstyle {H}} \rho_2$. Moreover, we have $\sigma \sim \gamma\gamma_{\scriptscriptstyle {H}} \rho_2$ in the same way as the former case. Indeed,   
\begin{eqnarray*}
\!\!\!&&\!\!\!({\delta_8}^{-1} \delta_2)(\gamma\gamma_{{}_{\scriptscriptstyle {H}}} \rho_2)
={\delta_8}^{-1} (\delta_2\gamma\gamma_{{}_{\scriptscriptstyle {H}}}) \rho_2
={\delta_8}^{-1}(\gamma\delta_2) \rho_2
\\[0mm]
\!\!\!&=&\!\!\!{\delta_8}^{-1}\gamma(\delta_2 \rho_2)
={\delta_8}^{-1}\gamma(\rho_2 \delta_2)
=({\delta_8}^{-1}(\gamma\rho_2)) \delta_2
\\[0mm]
\!\!\!&=&\!\!\!(\sigma {\delta_8}^{-1})\delta_2=\sigma ({\delta_8}^{-1}\delta_2),
\end{eqnarray*}
that is, $\sigma \sim \gamma\gamma_{{}_{\scriptscriptstyle {H}}} \rho_2$.
\end{proof}
\vspace{1mm}

We have the following proposition which is the direct result of Lemma 4.10.1.

\begin{prop} The group $(E_6)^\sigma$ is isomorphic to both of the groups $(E_6)^{\gamma_{{}_{\scriptscriptstyle {H}}} \rho_2}$ and $(E_6
)^{\gamma\gamma_{{}_{\scriptscriptstyle {H}}} \rho_2}${\rm:} $(E_6)^\sigma \cong (E_6)^{\gamma_{{}_{\scriptscriptstyle {H}}} \rho_2} \cong (E_6
)^{\gamma\gamma_{{}_{\scriptscriptstyle {H}}} \rho_2}$.

\end{prop}
\vspace{1mm}

From the result of types EII, EIII in Table 2 and Proposition 4.10.2, we have the following theorem.

\begin{thm} For $\mathbb{Z}_2 \times \mathbb{Z}_2=\{1,\gamma \} \times \{1, \gamma_{{}_{\scriptscriptstyle {H}}}\rho_2 \}$, the $\mathbb{Z}_2 \times \mathbb{Z}_2$-symmetric space is of type $(E_6/(E_6)^{\gamma}, E_6/(E_6)^{\gamma_{{}_{\scriptscriptstyle {H}}}\rho_2}, E_6/(E_6)^{\gamma (\gamma_{{}_{\scriptscriptstyle {H}}} \rho_2)})\!=\!(E_6/(E_6)^{\gamma}, E_6/(E_6)^{\sigma}, E_6/(E_6)^{\sigma})$, that is, type {\rm (EII, EIII, EIII)}, abbreviated as {\rm  EII-III-III}.
\end{thm}
Here, we prove proposition needed in theorem below.

\begin{prop}\,We have the following isomorphism{\rm :} 
$
 (U(1) \times S\!U(5))/\Z_5 \cong  S(U(1) \times U(5)),
$
$
\Z_5=\{ (1, \nu^k E)\,|\,\nu \in C, \nu^5=1, k=0,1, 2, 3, 4 \}$.
\end{prop}
\begin{proof}
We define a mapping $h_5 : U(1) \times S\!U(5) \to S(U(1) \times U(5))$ by
$$
h_5 (\theta, A)=\begin{pmatrix}
                                                  \theta^5 & 0  \\   
                                                         0      & \theta^{-1} A
                         \end{pmatrix}.                        
$$
Then we can easily show that $h_5$ induces the required isomorphism.
\end{proof}

Now, we determine the structure of the group $(E_6)^\gamma \cap (E_6)^{\gamma_{{}_{\scriptscriptstyle {H}}}\rho_2}$.

\begin{thm} We have that $(E_6)^\gamma \cap (E_6)^{\gamma_{{}_{\scriptscriptstyle {H}}}\rho_2} \cong (U(1) \times U(1) \times S\!U(5))/(\Z_2 \times \Z_5), \,\Z_2=\{(1, 1, E), (-1, -1, -E)  \}, \,\Z_5=\{ (1, \nu^k, \nu^k E)\}, k=0, 1, 2, 3, 4$.
\end{thm}
\begin{proof}
We define a mapping $\varphi_{4105}:U(1) \times U(1) \times S\!U(5) \to (E_6)^\gamma \cap (E_6)^{\gamma_{{}_{\scriptscriptstyle {H}}}\rho_2}$ by
$$
\varphi_{4105}(a, \theta, A)=\varphi_{{}_\text{E2}}(a, h_5 (\theta, A)),
$$
where
 $h_5$ is defined in Proposition 4.10.4. Since the mapping $\varphi_{4105}$ is the restriction of the mapping $\varphi_{{}_\text{E2}}$, it is easily to verify that $\varphi_{4105}$ is well-defined and a homomorphism. 

We shall show that $\varphi_{4105}$ is surjection. Let $\alpha \in (E_6)^\gamma \cap  (E_6)^{\gamma_{{}_{\scriptscriptstyle {H}}}\rho_2}$. Since $(E_6)^\gamma \cap  (E_6)^{\gamma_{{}_{\scriptscriptstyle {H}}}\rho_2} \subset (E_6)^\gamma$, there exist $p \in S\!p(1)$ and $U \in S\!U(6)$ such that $\alpha = \varphi_{{}_\text{E2}}(p, U)$ (Theorem 3.3.2). Moreover, from $\alpha = \varphi_{{}_\text{E2}}(p, U) \in  (E_6)^{\gamma_{{}_{\scriptscriptstyle {H}}}\rho_2}$, that is,  
$(\gamma_{{}_{\scriptscriptstyle {H}}}\rho_2)\varphi_{{}_\text{E2}}(p, U) (\rho_2 \gamma_{{}_{\scriptscriptstyle {H}}})=\varphi_{{}_\text{E2}}(p, U)$, 
using $\gamma_{\scriptscriptstyle {H}}\varphi_{{}_{\rm E2}}(p, U)\gamma_{\scriptscriptstyle {H}}=\varphi_{{}_{\rm E2}}(\gamma_{\scriptscriptstyle {H}} p, IUI)$ (Lemma 4.5.4 (2)) and $  \rho_2=\varphi_{{}_\text{E2}}(1, L_2)$, we have that  $\varphi_{{}_\text{E2}}( \gamma_{{\scriptscriptstyle {H}}} p , (I {L}_2) U ({L}_2 I))=\varphi_{{}_\text{E2}}(p, U)$. 
Hence it follows that
$$
\left \{
         \begin{array}{l}
                \gamma_{{\scriptscriptstyle {H}}}p  = p
                         \vspace{3mm}\\
                (I {L}_2) U ({L}_2 I)) = U
         \end{array}\right.\quad \text{or}\quad 
\left \{         
          \begin{array}{l}
                 \gamma_{{\scriptscriptstyle {H}}}  p  =- p
                         \vspace{3mm}\\
                (I {L}_2) U ({L}_2 I)) = -U.
         \end{array}\right. 
$$
In the former case, we see that $p \in U(1)$, moreover since $I {L}_2={L}_2 I=\diag(1, -1, -1, -1,$ $ -1, -1)$, we get the explicit form of $U \in S\!U(6)$ as follows:
$$
U=\begin{pmatrix} \zeta & 0 \\
                                0      & B
   \end{pmatrix} ,\, \zeta \in U(1),\, B \in U(5),\, \det \, U=1,                            
$$
that is , $U \in S(U(1) \times U(5))$. Hence, since there exist $\theta \in U(1)$ and $A \in S\!U(5)$ such that $U=h_5(\theta, A)$ (Proposition 4.10.4), we have  $\alpha=\varphi_{{}_\text{E2}}(a, h_5 (\theta, A))=\varphi_{4105}(a, \theta, A)$. 
 In the latter case, as the former case, we can also find the explicit form of $U \in S\!U(6)$ as follows:
$$
U=\begin{pmatrix}  0  &   \x  \\
                              {}^t \y &   0
     \end{pmatrix}, \,\,\x , \y \in C^5,
$$
where $C=\{x_1 +x_2 i \,|\, x_k \in \R, k=1,2 \}$. However, this case is impossible because of $\det\, U=0$ for $U \in S\!U(6)$. Thus $\varphi_{4105}$ is surjection.

From $\Ker \,\varphi_{{}_\text{E2}}=\{ (1, E), (-1, -E)\}$ and $\Ker \, h_5=\{ (\nu^k, \nu^k E)\,|\, k=0,1,2,3,4 \}$, we can easily obtain that $\Ker\, \varphi_{4105}=\{ (1, 1, E), (-1, -1, -E) \} \times \{ (1, \nu^k, \nu^k E)\,|\,k=0, 1, 2, 3, 4\} \cong \Z_2 \times \Z_5$. 

Therefore we have the required isomorphism
$$
(E_6)^\gamma \cap (E_6)^{\gamma_{{}_{\scriptscriptstyle {H}}}\rho_2} \cong (U(1) \times U(1) \times S\!U(5))/(\Z_2 \times \Z_5).
$$
\end{proof}

\subsection{Type EIII-III-III}

In this section, we give a pair of involutive inner automorphisms $\tilde{\sigma}$ and $\tilde{\sigma}'$.
\vspace{1mm}
 
Let the $C$-linear transformation $\sigma'$ of $\mathfrak{J}^C$ be the complexification of $\sigma' \in F_4$, so
$\sigma'$ induces involutive inner automorphism $\tilde{\sigma}'$ of $E_6$: $\tilde{\sigma}'(\alpha)= \sigma' \alpha\sigma', \alpha \in E_6$.

\noindent Using the inclusion $F_4 \subset E_6$, the $\R$-linear transformations $\delta_6, \delta_7$ defined in the proof of Lemma 4.4.1 are naturally extended to the $C$-linear transformations of $\mathfrak{J}^C$. Then we have $\delta_6, \delta_7 \in E_6, {\delta_6}^2={\delta_7}^2=1$.
Hence, as in $F_4$, since we easily see that $\delta_6 \sigma=\sigma' \delta_6, \delta_7 \sigma=(\sigma\sigma')\delta_7$ 
, that is,  
$\sigma \sim \sigma', \sigma \sim \sigma\sigma'$ in $E_6$, we have the following proposition. 

\begin{prop}
The group $(E_6)^\sigma$ is isomorphic to both of the groups $(E_6)^{\sigma'}$ and $(E_6)^{\sigma \sigma'}${\rm:}  
 $(E_6)^\sigma \cong (E_6)^{\sigma'} \cong (E_6)^{\sigma \sigma'}$.
\end{prop}
From the result of Type EIII in Table 2
 and  Proposition 4.11.1, we have the following theorem.
\begin{thm} For $\mathbb{Z}_2 \times \mathbb{Z}_2=\{1,\sigma \} \times \{1, \sigma' \}$, the $\mathbb{Z}_2 \times \mathbb{Z}_2$-symmetric space is of type $(E_6/(E_6)^{\sigma}, E_6/(E_6)^{\sigma'}, E_6/(E_6)^{\sigma\sigma'})=(E_6/(E_6)^{\sigma}, E_6/(E_6)^{\sigma}, E_6/(E_6)^{\sigma})$, that is, type {\rm (EIII, EIII, EIII)}, abbreviated as {\rm  EIII-III-III}.
\end{thm}

Here, we prove lemmas needed and make some preparations for\vspace{1mm}
 Proposition 4.11.5 below.

\noindent First, we investigate the subgroup $((E_6)_{E_1})^{\sigma'}$ of $E_6$ defined by
$$
((E_6)_{E_1})^{\sigma'}=\{ \alpha \in E_6 \,| \,  \alpha E_1=E_1,\, \sigma'\alpha\sigma' =\alpha \},
$$
where the group $(E_6)_{E_1}$ is isomorphic to $S\!pin(10)$ as the double  covering group of $S\!O(10)$ (As for the group $(E_6)_{E_1} \cong S\!pin(10)$, see \cite[Theorem 3.10.4]{Yokotaichiro0}).

\begin{lem} 
The Lie algebra $((\mathfrak{e}_6 )_{E_1})^{\sigma'}$ of the group $((E_6)_{E_1})^{\sigma'}$ are  given by
$$
((\mathfrak{e}_6 )_{E_1})^{\sigma'} = \{\delta + i\wti{T} \in \mathfrak{e}_6 \, | \, \delta \in ((\mathfrak{f}_4)_{E_1})^{\sigma'},\, T \in \mathfrak{J},\, \tr(T) = 0, \,{\sigma'} T = T ,\, T \circ E_1 = 0 \},
$$
where $((\mathfrak{f}_4)_{E_1})^{\sigma'}=(\mathfrak{f}_4)^\sigma \cap  (\mathfrak{f}_4)^{\sigma'} \cong \mathfrak{so}(8)$ {\rm {(Section 4.4)}}.

 In particular, we have
$$
             \dim(((\mathfrak{e}_6 )_{E_1})^{\sigma'}) = 28 + 1 = 29.
$$
\end{lem}
\begin{proof}
Since any element $\phi$ of the Lie algebra $
\mathfrak{e}_6$ of the group $E_6$ is uniquely as $\phi=\delta+i\tilde{T}, \delta \in \mathfrak{f}_4, T \in \mathfrak{J}_0 =\{ T \in \mathfrak{J}\,|\, \tr(T)=0 \}$, using $\sigma' \phi \sigma' = \sigma' \delta \sigma' +i \tilde{(\sigma' T)}$,  
we can easily prove this lemma (As for $((\mathfrak{f}_4)_{E_1})^{\sigma'} \cong \mathfrak{so}(8)$, see Theorem 4.4.5 and \cite[Theorem 2.9.1]{Yokotaichiro0}).
\end{proof}

\begin{lem} For $\theta$, $\nu \in U(1) = \{ \theta \in C \, | \, (\tau\theta)\theta = 1 \}$, we define $C$-linear transformations $\phi_1(\theta)$ and $\phi_2(\nu)$ of $\mathfrak{J}^C$ by
$$
   \phi_1(\theta)X = \begin{pmatrix}
                             \theta^4\xi_1 & \theta x_3 & \theta\,\ov{x}_2 
\\
                             \theta\,\ov{x}_3 & \theta^{-2}\xi_2 & \theta^{-2}x_1 
\\
                             \theta x_2 & \theta^{-2}\,\ov{x}_1 & \theta^{-2}\xi_3
                               \end{pmatrix}, 
\,
   \phi_2(\nu)X = \begin{pmatrix} 
                              \xi_1 & \nu x_3 & \nu^{-1}\,\ov{x}_2 
\\
                             \nu\,\ov{x}_3 & \nu^2\xi_2 & x_1 
\\
                             \nu^{-1}x_2 & \ov{x}_1 & \nu^{-2}\xi_3
                             \end{pmatrix}, \,X  \in \mathfrak{J}^C ,
$$
respectively. Then we have that $\phi_1(\theta), \phi_2(\nu) \in (E_6)^{\sigma} \cap (E_6)^{\sigma'}$, moreover we have that $\phi_2(\nu) \in ((E_6)_{E_1})^{\sigma'}$ and that $\phi_1(\theta), \phi_2(\nu)$ are commutative. 
\end{lem}
\begin{proof}
By straightforward computation, we can also easily prove this \vspace{1mm}lemma.
\end{proof}

\begin{prop} We have the following isomorphism{\rm :}
$((E_6)_{E_1})^{\sigma'} \cong (U(1) \times S\!pin(8))/\Z_2, $ $\Z_2 = \{ (1, 1), (-1,  \sigma) \}$.
\end{prop}
\begin{proof}
Let $S\!pin(8) \cong (F_4)^\sigma \cap  (F_4)^{\sigma'} \!=\!((F_4)^\sigma)^{\sigma'}\!=\! ((F_4)_{E_1})^{\sigma'}\subset  ((E_6)_{E_1})^{\sigma'}$ (Theorem 4.4.5).  Now, we define a mapping $\varphi_2 : U(1) \times S\!pin(8) \to ((E_6)_{E_1})^{\sigma'}$ by
$$
      \varphi_2(\nu, \beta) = \phi_2(\nu)\beta. 
$$
It is clear that $\varphi_2(\nu, \beta) \in ((E_6)_{E_1})^{\sigma'}$ (Lemma 4.11.4). Hence $\varphi_2$ is well-defined. Since $\phi_2(\nu)$ commutes with $\beta$, $\varphi_2$ is a homomorphism. Let $(\nu, \beta) \in \Ker\, \varphi_2$. Then since we see that $\beta =1$, we can easily obtain that 
$\Ker\,\varphi_2 = \{(1, 1), (-1, \sigma) \}$. Furthermore since $((E_6)_{E_1})^{\sigma'}\! = (S\!pin(10))^{\sigma'}$ is connected and $\dim(((\mathfrak{e}_6 )_{E_1})^{\sigma'})
 = 29  = 1 + 28 = \dim(\mathfrak{u}(1) \oplus \mathfrak{so}(8))$ (Lemma 4.11.3), we have that $\varphi_2$ is surjection. Therefore we have the isomorphism 
$
((E_6)_{E_1})^{\sigma'} \cong (U(1) \times S\!pin(8))/\Z_2.
$  
\end{proof}
Now, we determine the structure of the group $(E_6)^{\sigma} \cap (E_6)^{\sigma'}$ from Proposition 4.11.5.
\begin{thm} We have that 
$(E_6)^{\sigma} \cap (E_6)^{\sigma'} \cong (U(1) \times U(1) \times S\!pin(8))/(\Z_2 \times \Z_4), $ $ \Z_2= \{ (1, 1, 1), (1, -1, \sigma) \}, \Z_4 = \{(1, 1, 1), (-i, i, \sigma'), (-1, -1, 1), (i, -i, \sigma') \}$.
\end{thm}
\begin{proof}
Let $S\!pin(8) \cong (F_4)^{\sigma} \cap (F_4)^{\sigma'} \subset (E_6)^{\sigma} \cap (E_6)^{\sigma'}$(Theorem 4.4.5). We define a mapping $\varphi_{4116} : U(1) \times U(1) \times S\!pin(8) \to (E_6)^{\sigma} \cap (E_6)^{\sigma'}$ by
$$
      \varphi_{4116}(\theta, \nu, \beta) = \phi_1(\theta)\phi_2(\nu)\beta. $$
It is clear that $\varphi_{4116}(\theta,\nu, \beta) \in (E_6)^{\sigma} \cap (E_6)^{\sigma'}$ (Lemma 4.11.4). Hence $\varphi_{4116}$ is well-defined. Since $\phi_1(\theta), \phi_2(\nu)$ and $\beta$ commute with each other,  it is clear that $\varphi_{4116}$ is a homomorphism. 

We shall show that $\varphi_{4116}$ is surjection. Let $\alpha \in (E_6)^{\sigma} \cap (E_6)^{\sigma'}$. Since $\alpha \in (E_6)^{\sigma} \cap (E_6)^{\sigma'} \subset (E_6)^{\sigma}$, there exist $\theta \in U(1)$ and $\delta \in S\!pin(10)$ such that $\alpha = \varphi_{{}_{\rm E3}}(\theta,\delta)$ (Theorem 3.3.3). Moreover, from $\alpha\!=\! \varphi_{{}_{\rm E3}}(\theta,\delta) \!\in\! (E_6)^{\sigma'}$, that is, $\sigma' \varphi_{{}_{\rm E3}}(\theta,\delta) \sigma' \!=\! \varphi_{{}_{\rm E3}}(\theta,\delta) $, 
using $\sigma' \phi_1(\theta)=\phi_1(\theta) \sigma'$ (Lemma 4.11.4), we have $\varphi_{{}_{\rm E3}}(\theta,\sigma' \delta\sigma')=\varphi_{{}_{\rm E3}}(\theta,\delta)$. Hence it follows that
\begin{eqnarray*}
 &&  \mbox{(i)} \;\; 
\left\{\begin{array}{l}
          \theta = \theta 
\vspace{1mm}\\
          \sigma'\delta\sigma' = \delta,  
 \end{array} \right.
 \qquad\quad\,\, \mbox{(ii)} \;
 \left\{\begin{array}{l}
          \theta=- \theta 
\vspace{1mm}\\
         \sigma'\delta\sigma'= \phi_1(-1)\delta, 
 \end{array}\right. 
 \\[2mm]
&& \mbox{(iii)} \;
\left\{\begin{array}{l}
         \theta= i\theta 
\vspace{1mm}\\
          \sigma'\delta\sigma' =\phi_1(-i)\delta,  
\end{array}\right. 
\; \mbox{ (iv)} \;
\left\{\begin{array}{l}
         \theta= - i\theta 
\vspace{1mm}\\
          \sigma'\delta\sigma' =\phi_1(i)\delta.  
\end{array}\right.
\end{eqnarray*}
The cases (ii), (iii) and (iv) are impossible because of $\theta =0$ for $\theta \in U(1)$. In the case (i), from $\sigma'\delta\sigma' = \delta$, we have  that $\delta \in (S\!pin(10))^{\sigma'} \cong ((E_6)_{E_1})^{\sigma'}$. Since there exist $\nu \in U(1)$ and $\beta \in S\!pin(8)$ such that $\delta = \phi_2(\nu)\beta$ (Proposition 4.11.5), we have that
$$
     \alpha = \phi_1(\theta)\delta = \phi_1(\theta)\phi_2(\nu)\beta = \varphi(\theta, \nu, \beta). 
$$
Thus $\varphi$ is surjection. 

From $\Ker \,\varphi_{{}_{\rm E3}}\!=\{ (1, \phi(1)), (-1, \phi(-1)), (i, \phi(-i)), (-i,\phi(i)) \}$ and $\Ker \,\varphi_2\!=\{ (1, 1), (-1,  \sigma) \}$, we can easily obtain that
\begin{eqnarray*}
   \Ker\,\varphi_{4116} \!\!\! &=& \!\!\! \{(1, 1, 1), (1, -1, \sigma), (-1, -1, \sigma), (-1, -1, 1), (i, i, \sigma\sigma'), (i, -i, \sigma'), 
\vspace{1mm}\\
    \!\!\! && \!\!\!\qquad\qquad  (-i, i, \sigma'), (-i, -i, \sigma\sigma') \}
\vspace{1mm}\\
   \!\!\! &=& \!\!\!  \{(1, 1, 1), (1, -1, \sigma)\} \times \{(1, 1, 1), (-i, i, \sigma'), (-1, -1, 1), (i, -i, \sigma') \}
\vspace{1mm}\\
   \!\!\! &\cong & \!\!\! \Z_2 \times \Z_4.
\end{eqnarray*} 

Therefore we have the required isomorphism 
$$
(E_6)^{\sigma} \cap (E_6)^{\sigma'} \cong (U(1) \times U(1) \times S\!pin(8))/(\Z_2 \times \Z_4).
$$
\end{proof}

\subsection{Type EIII-IV-IV}

In this section, we give a pair of involutive  automorphisms $\lambda$ and $\tilde{\sigma}$.
\vspace{1mm}

We define a $C$-linear transformation $\delta_9$ of $\mathfrak{J}^C$ by
$$
\delta_9 X =\begin{pmatrix} \xi_1         &  i x_3           & i\,\ov{x}_2 \\
                                          i\,\ov{x}_3 & -\xi_2           & - x_1           \\
                                          i x_2            &- \,\ov{x}_1 & -\xi_3
              \end{pmatrix}=D_9 X D_9, D_9=\begin{pmatrix}1 & 0 & 0 \\
                                0 & i & 0 \\
                                0 & 0 & i 
        \end{pmatrix}, \, X \in \mathfrak{J}^C.   
$$
Then we have that $\delta_9 \in E_6, {\delta_9}^2 =\sigma, \delta_9 \sigma=\sigma \delta_9$ and $ {}^t \delta_9=\delta_9$. 
\begin{prop} The group $(E_6)^\lambda$ is isomorphic to the group $(E_6)^{\lambda\sigma}${\rm :} $(E_6)^\lambda \cong (E_6)^{\lambda\sigma}$.
\end{prop}
\begin{proof}
We define a mapping $f : (E_6)^\lambda \to (E_6)^{\lambda\sigma}$ by
$$
   f(\alpha)=\delta_9 \alpha {\delta_9}^{-1}.
$$
In order to prove this proposition, it is sufficient to show that the mapping $f$ is well-defined.
Indeed, it follows from the properties of ${\delta_9}^2 =\sigma, \delta_9 \sigma=\sigma \delta_9, {}^t \delta_9=\delta_9$ and $\lambda(\sigma)=\sigma$ that 
\begin{eqnarray*}
(\lambda\sigma)f(\alpha)\!\!\!&=&\!\!\!(\lambda\sigma)(\delta_9 \alpha {\delta_9}^{-1})=\lambda(\delta_9 \sigma \alpha {\delta_9}^{-1})
=\lambda(\delta_9)\lambda(\sigma)\lambda(\alpha)\lambda({\delta_9}^{-1})
\\[1mm]
\!\!\!&=&\!\!\!{\delta_9}^{-1}\sigma\alpha\delta_9={\delta_9}^{-1}(\delta_9)^2\alpha({\delta_9}^{-1}\sigma)
=(\delta_9 \alpha {\delta_9}^{-1})\sigma
\\[1mm]
\!\!\!&=&\!\!\!f(\alpha)\sigma,
\end{eqnarray*}
that is, $f(\alpha) \in (E_6)^{\lambda\sigma}$.
\end{proof}
From the results of Types EIII, EIV in Table 2 and  Proposition 4.12.1, we have the following theorem.
\begin{thm} For $\mathbb{Z}_2 \times \mathbb{Z}_2=\{1,\sigma \} \times \{1, \lambda \}$, the $\mathbb{Z}_2 \times \mathbb{Z}_2$-symmetric space is of type $(E_6/(E_6)^{\sigma}, E_6/(E_6)^{\lambda}, E_6/(E_6)^{\lambda\sigma})\!=\!(E_6/(E_6)^{\sigma},E_6/(E_6)^{\lambda},E_6/(E_6)^{\lambda})$, that is, type {\rm (EIII, EIV, EIV)}, abbreviated as {\rm  EIII-IV-IV}.
\end{thm}
Now, we determine the structure the group $(E_6)^{\sigma} \cap (E_6)^{\lambda}$.
\begin{thm} We have that $(E_6)^{\sigma} \cap (E_6)^{\lambda} \cong S\!pin(9)$.
\end{thm}
\begin{proof}
We define a mapping $\varphi_{4123}: S\!pin(9) \to (E_6)^{\sigma} \cap (E_6)^{\lambda}$ by 
$$
\varphi_{4123}(\alpha)=\alpha.
$$
Since $S\!pin(9) \cong (F_4)^\sigma \subset (E_6)^\sigma$ (Theorem 3.2.2) and $ S\!pin(9) \subset F_4 = (E_6)^\lambda$ (Theorem 3.3.4), it is clear that $\varphi_{4123}$ is well-defined , a homomorphism and injection. 

We shall show that $\varphi_{4123}$ is surjection. Let $\alpha \in (E_6)^\sigma \cap  (E_6)^\lambda$. Since $(E_6)^\sigma \cap  (E_6)^\lambda \subset (E_6)^\lambda =F_4$, it is clear that $\alpha \in F_4$.
Moreover, from $\alpha \in  (E_6)^\sigma$, that is, $\sigma\alpha\sigma=\alpha$, 
we easily see $\alpha \in (F_4)^\sigma \cong S\!pin(9)$. Thus $\varphi_{4123}$ is surjection. 

Therefore we have the required isomorphism 
$$
(E_6)^{\sigma} \cap (E_6)^{\lambda} \cong \vspace{-2mm}S\!pin(9).
$$ 
\end{proof}
\vspace{2mm}

$\bullet$\,\,{\boldmath $[E_7]$}\,\,We study seven types in here.

\subsection{Type EV-V-V}

In this section, we give a pair of involutive inner automorphisms $\tilde{\lambda\gamma}$ and $\tilde{\iota\gamma_{\scriptscriptstyle {C}}}$.
\vspace{1mm}

\vspace{1mm}

We define  $C$-linear transformations 
$
\gamma_{{}_{\scriptscriptstyle{C}}}$  of $\mathfrak{P}^C$ by
\begin{eqnarray*}
\gamma_{{}_{\scriptscriptstyle{C}}}(X, Y, \xi, \eta)
\!\!\!&=&\!\!\!(\gamma_{{}_{\scriptscriptstyle{C}}} X, \gamma_{{}_{\scriptscriptstyle{C}}} Y, \xi, \eta),\,(X, Y, \xi, \eta) \in \mathfrak{P}^C,
\end{eqnarray*}
where $
\gamma_{{}_{\scriptscriptstyle{C}}}$ 
of the right hand side are the same ones as $
\gamma_{{}_{\scriptscriptstyle{C}}} \in G_2 \subset F_4 \subset  E_6$.
 Then we have that $
\gamma_{{}_{\scriptscriptstyle{C}}} \in  E_7, 
{\gamma_{{}_{\scriptscriptstyle{C}}}}^2=1$, so $
{\tilde{\gamma}}_{{}_{\scriptscriptstyle{C}}}$ of $E_7$: 
$
{\tilde{\gamma}}_{{}_{\scriptscriptstyle{C}}}(\alpha)= \gamma_{\scriptscriptstyle{C}}\alpha\gamma_{{}_{\scriptscriptstyle{C}}},
\alpha \in E_7$. 

\noindent Similarly, for $
 \delta_3, \delta_4 \in  G_2 \subset F_4 \subset E_6$, we have $
\delta_3, \delta_4 \in E_7$. Hence, as in $E_6$, we easily see that $
\delta_3 \gamma=\gamma_{\scriptscriptstyle{C}}\delta_3, \delta_4 \gamma=(\gamma\gamma_{\scriptscriptstyle{C}})\delta_4$, that is, $
\gamma \sim \gamma_{\scriptscriptstyle{C}}, \gamma \sim  \gamma\gamma_{\scriptscriptstyle{C}} $ in $E_7$.
\begin{lem}
In $E_7$, $\iota$ is conjugate to $\lambda${\rm :} $\iota \sim \lambda$.
\end{lem}
\begin{proof}
We define a $C$- linear transformation $\delta_{\lambda}$ of $\mathfrak{P}^C$ by  
$$
\delta_{\lambda}
\begin{pmatrix}
                                                 X 
\\
                                                 Y
\\
                                               \xi
\\
                                               \eta
\end{pmatrix}
                         =\frac{1}{\sqrt 8}
\begin{pmatrix}
                         -(\tr(X)E-2 X)+i(\tr(Y)E-2 Y)-\xi E+i \eta E
\\
                         i(\tr(X)E-2 X)-(\tr(Y)E-2 Y)+i \xi E-\eta E
\\
                          -\tr(X)+i\tr(Y)+\xi-i \eta
\\
                            i \tr(X)-\tr(Y)-i \xi+\eta   
\end{pmatrix}, \begin{pmatrix}
                                                 X 
\\
                                                 Y
\\
                                               \xi
\\
                                               \eta
\end{pmatrix} \in \mathfrak{P}^C.
$$
Then, by straightforward computation, we have $\delta_{\lambda}\iota=\lambda\delta_{\lambda}$, that is $\iota \sim \lambda$ in $E_7$, moreover $\delta_{\lambda} \gamma=\gamma\delta_{\lambda}, \delta_{\lambda} \gamma_{\scriptscriptstyle {C}}=\vspace{1mm}\gamma_{\scriptscriptstyle {C}} \delta_{\lambda}$. 
\end{proof}
\noindent {\bf Remark.} In fact, using $\varPhi(0,(i\pi/4)E,  (i\pi/4)E, 0) \in \mathfrak{e}_7$, 
$\delta_\lambda$ is expressed as $\exp \varPhi(0, (i\pi/4) E,$ $ (i\pi/4)E, 0)$: $\delta_\lambda=\exp \varPhi(0, (i\pi/4) E, (i\pi/4)E, 0)$. 
\vspace{1mm}

\begin{prop}
The group $(E_7)^{\lambda\gamma}$ is isomorphic to the group $(E_7)^{\iota\gamma_{{}_{\scriptscriptstyle {C}}}}${\rm :} $(E_7)^{\lambda\gamma} \cong (E_7)^{\iota\gamma_{{}_{\scriptscriptstyle {C}}}}$.
\end{prop}
\begin{proof}
We define a mapping $f: (E_7)^{\lambda\gamma} \to (E_7)^{\iota\gamma_{{}_{\scriptscriptstyle {C}}}}$ by 
$$
    f(\alpha)=\delta_3 {\delta_{\lambda}}^{-1}\alpha \delta_{\lambda} \delta_3.
$$
In order to prove this proposition, it is sufficient to show that the mapping $f$ is well-defined. However, it is almost evident from $\delta_{\lambda}\iota=\lambda\delta_{\lambda}, \delta_3 \gamma= \gamma_{\scriptscriptstyle {C}}\delta_3$. 
\end{proof}

For $\theta \in U(1)=\{\theta \in C \,|\,(\tau {}^t \theta)\theta=1 \}$, we define a $C$-linear transformation $\phi(\theta)$ of $\mathfrak{P}^C$ by 
$$
   \phi(\theta)(X, Y, \xi, \eta)=(\theta X, \theta^{-1}Y,\theta^{-3} \xi,  \theta^3\eta),\,\, (X, Y, \xi, \eta) \in \mathfrak{P}^C.
$$

\noindent Then we have $\phi(\theta) \in E_7$, and  set $\delta_\iota=\phi( e^{i \frac{\pi}{4}})$. Needless to say, we see $\delta_\iota \in E_7$. 
Besides, the mapping $\phi$ gives an embedding from $U(1)$ into $E_7$.
\begin{lem}
In $E_7$, $\lambda\iota$ is conjugate to $-\lambda${\rm:} $\lambda\iota \sim -\lambda$.
\end{lem}
\begin{proof}
By using the definition of $\delta_\iota$ above, we can easily obtain $(\lambda\iota)\delta_{\iota}=\delta_{\iota} (-\lambda)$, that is, $\lambda\iota \sim -\lambda$. 
\end{proof}
\begin{prop}
The group $(E_7)^{\lambda\gamma}$ is isomorphic to the group $(E_7)^{\lambda\iota\gamma\gamma_{{}_{\scriptscriptstyle {C}}}}$\!{\rm :}$(E_7)^{\lambda\gamma} \!\!\cong \!(E_7)^{\lambda\iota\gamma\gamma_{{}_{\scriptscriptstyle {C}}}}$\!.
\end{prop}
\begin{proof}
We define a mapping $g: (E_7)^{\lambda\gamma} \to (E_7)^{\lambda\iota\gamma\gamma_{{}_{\scriptscriptstyle {C}}}}$ by
$$
   g(\alpha)=\delta_4 \delta_{\iota} \alpha {\delta_{\iota}}^{-1} \delta_4,
$$
where we remark that $\delta_4$ has the property of $\delta_4 (\lambda\iota)=(\lambda\iota)\delta_4$. 
In order to prove this proposition, it is sufficient to show that the mapping $g$ is well-defined. Indeed, it follows from $(\lambda\iota)\delta_{\iota}=\delta_{\iota} (-\lambda)$ (Lemma 4.13.3) and the property of $\delta_4$ that 
\begin{eqnarray*}(\lambda\iota\gamma\gamma_{{}_{\scriptscriptstyle {C}}})g(\alpha)\!\!\!&=&\!\!\!(\lambda\iota\gamma\gamma_{{}_{\scriptscriptstyle {C}}})(\delta_4 \delta_{\iota} \alpha {\delta_{\iota}}^{-1} \delta_4)=(\lambda\iota)\delta_4 \gamma \delta_{\iota} \alpha {\delta_{\iota}}^{-1} \delta_4 
=\delta_4 (\lambda\iota)\delta_{\iota} \gamma \alpha {\delta_{\iota}}^{-1} \delta_4 
\\
\!\!\!&=&\!\!\!\delta_4 \delta_{\iota}(-\lambda)\gamma \alpha {\delta_{\iota}}^{-1} \delta_4=\delta_4 \delta_{\iota}\alpha(-\lambda)\gamma  {\delta_{\iota}}^{-1} \delta_4=\delta_4 \delta_{\iota}\alpha(-\lambda)  {\delta_{\iota}}^{-1} \gamma \delta_4 
\\
\!\!\!&=&\!\!\!(\delta_4 \delta_{\iota}\alpha {\delta_{\iota}}^{-1}) (\lambda\iota\gamma \delta_4)=(\delta_4 \delta_{\iota}\alpha {\delta_{\iota}}^{-1}\delta_4) (\lambda\iota\gamma\gamma_{{}_{\scriptscriptstyle {C}}})=g(\alpha)(\lambda\iota\gamma\gamma_{{}_{\scriptscriptstyle {C}}}),
\end{eqnarray*}
\noindent that is, $g(\alpha) \in (E_7)^{\lambda\iota\gamma\gamma_{{}_{\scriptscriptstyle {C}}}}$.
\end{proof}
From the result of type EV in Table 2 and Propositions 4.13.2, 4.13.4, we have the following theorem.
\begin{thm}
 For $\mathbb{Z}_2 \times \mathbb{Z}_2=\{1,\lambda\gamma \} \times \{1, \iota\gamma_{{}_{\scriptscriptstyle {C}}} \}$, the $\mathbb{Z}_2 \times \mathbb{Z}_2$-symmetric space is of type $(E_7/(E_7)^{\lambda\gamma}, E_7/(E_7)^{\iota\gamma_{{}_{\scriptscriptstyle {C}}}}, E_7/(E_7)^{(\lambda\gamma)(\iota\gamma_{{}_{\scriptscriptstyle {C}}})})=(E_7/(E_7)^{\lambda\gamma}, E_7/(E_7)^{\lambda\gamma}, E_7/(E_7)^{\lambda\gamma})$, that is, type {\rm (EV, EV, EV)}, abbreviated as {\rm  EV-V-V}.
\end{thm}

Here, 
we prove lemma needed in theorem below.

\begin{lem}
The mapping $\varphi_{{}_{\rm E5}}: S\!U(8) \to (E_7)^{\lambda\gamma}$ of Theorem 3.4.1 satisfies the following equalities{\rm :}
\begin{eqnarray*}
&&{\rm (1)}\,\, \gamma=\varphi_{{}_{\rm E5}}(I_2), \,\,\gamma_{\scriptscriptstyle {C}}=\varphi_{{}_{\rm E5}}(J), \sigma=\varphi_{{}_{\rm E5}}(I_4).
\\[0mm]
&&{\rm (2)}\,\, \gamma \varphi_{{}_{\rm E5}}(A) \gamma=\varphi_{{}_{\rm E5}}(I_2 A I_2), \,\,\gamma_{{}_{\scriptscriptstyle {C}}}\varphi_{{}_{\rm E5}}(A)\gamma_{{}_{\scriptscriptstyle {C}}}=\varphi_{{}_{\rm E5}}(JAJ),\,\,\sigma \varphi_{{}_{\rm E5}}(A) \sigma=\varphi_{{}_{\rm E5}}(I_4 A I_4),
\\[0mm]
&&\hspace*{5mm}\iota\varphi_{{}_{\rm E5}}(A)\iota^{-1
}=\varphi_{{}_{\rm E5}}(J\ov{A}J),
\end{eqnarray*}
where \vspace{1mm}$I_2=\diag(-1, -1, 1,1,1,1,1,1),\,J=\diag(J, J, J, J),J=\begin{pmatrix} 0&1 \\
                    -1&0
     \end{pmatrix},\,I_4=\diag(-1, $ $-1, -1,-1,1,1,1,1)$.
\end{lem}
\begin{proof}
Since the equalities above are the direct results of \cite[Lemma 4.5.4]{Yokotaichiro0}, this proof is omitted.
\end{proof}

Now, we determine the structure of the group $(E_7)^{\lambda\gamma} \cap (E_7)^{\iota\gamma_{{}_{\scriptscriptstyle {C}}}}$.
\begin{thm}
We have that $(E_7)^{\lambda\gamma} \cap (E_7)^{\iota\gamma_{{}_{\scriptscriptstyle {C}}}} \cong S\!O(8)/\Z_2 \times \mathcal{Z}_2,\, \Z_2=\{E, -E \}, \,\mathcal{Z}_2= \{1, -1 \}$.
\end{thm}
\begin{proof}
We define a mapping $S\!O(8) \times \{1, -1 \} \to (E_7)^{\lambda\gamma} \cap (E_7)^{\iota\gamma_{\scriptscriptstyle {C}}}$ by
\begin{eqnarray*}
\varphi_{4137}(B, 1)\!\!\!&=&\!\!\!\varphi_{{}_\text{E5}}(B), 
\\
\varphi_{4137}(B, -1)\!\!\!&=&\!\!\!\varphi_{{}_\text{E5}}(B) (-1),
\end{eqnarray*}
where $\varphi_{{}_\text{E5}}$ is defined in Theorem 3.4.1. Since 
the element $-1 \in z(E_7)$ (the center of $E_7$), 
it is clear that $\varphi_{4137}(B, 1), \varphi_{4137}$ $(B, -1) \in (E_7)^{\lambda\gamma}$, moreover using $\iota \varphi_{{}_\text{E5}}(A) {\iota}^{-1}=\varphi_{{}_\text{E5}}(J \ov{A}J)$ and $\gamma_{{}_{\scriptscriptstyle {C}}} \varphi_{{}_\text{E5}}(A) \gamma_{{}_{\scriptscriptstyle {C}}}=\varphi_{{}_\text{E5}}(JAJ)$ (Lemma 4.13.6 (2)), we see that $\varphi_{4137}(B, 1),  \varphi_{4137}(B, -1) $ $\in (E_7)^{\iota\gamma_{{}_{\scriptscriptstyle {C}}}}$. Hence $\varphi_{4137}$ is well-defined. Since the mapping $\varphi_{4137}$ is the restriction of the mapping $\varphi_{{}_\text{E5}}$, it is easy to verify that $\varphi_{4137}$ is a homomorphism. 

We shall show that $\varphi_{4137}$ is surjection. Let $\alpha \in  (E_7)^{\lambda\gamma} \cap (E_7)^{\iota\gamma_{{}_{\scriptscriptstyle {C}}}}$. Since $\alpha \in (E_7)^{\lambda\gamma} \cap (E_7)^{\iota\gamma_{{}_{\scriptscriptstyle {C}}}} \subset  (E_7)^{\lambda\gamma}$, there exists $A \in S\!U(8)$ such that $\alpha = \varphi_{{}_\text{E5}}(A)$ (Theorem 3.4.1). Moreover, from $\alpha=\varphi_{{}_\text{E5}}(A) \in (E_7)^{\iota\gamma_{{}_{\scriptscriptstyle {C}}}}$, that is, $(\iota\gamma_{{}_{\scriptscriptstyle {C}}})  \varphi_{{}_\text{E5}}(A) (\gamma_{{}_{\scriptscriptstyle {C}}}\iota^{-1}) =\varphi_{{}_\text{E5}}(A)$, 
again using $\iota \varphi_{{}_\text{E5}}(A) {\iota}^{-1}=\varphi_{{}_\text{E5}}(J \ov{A}J)$ and $\gamma_{{}_{\scriptscriptstyle {C}}} \varphi_{{}_\text{E5}}(A) \gamma_{{}_{\scriptscriptstyle {C}}}=\varphi_{{}_\text{E5}}(JAJ)$, we have  $\varphi_{{}_\text{E5}}(\ov{A})=\varphi_{{}_\text{E5}}(A)$. Hence it follows that
$$
    \ov{A}=A   \quad   {\rm or}  \quad   \ov{A}=-A.
$$ 
In the former case, we see $ A \in S\!O(8)$, then set $A=B \in S\!O(8)$. Hence we have that $\alpha=\varphi_{{}_\text{E5}}(B)=\varphi_{4137}(B)$.
In the latter case, we have that $A=e_1 B, \, B \in S\!O(8)$. Hence we have that $\alpha=\varphi_{{}_\text{E5}}(e_1 B)=\varphi_{{}_\text{E5}}(e_1 E )\varphi_{{}_\text{E5}}(B)=(-1)\varphi_{{}_\text{E5}}(B)$, that is,  $\alpha=\varphi_{4137}(B)(-1)$. Thus $\varphi_{4137}$ is surjection. 

Finally, we shall determine $\Ker \varphi_{4137}$. From the definition of kernel, we have that 
$$
\Ker\,\varphi_{4137}=\{(B,1)\,|\, \varphi_{4137}(B, 1)=1 \} \cup \{(B,-1)\,|\, \varphi_{4137}(B, -1)=1 \}.
$$
In the former case, from $\Ker\, \varphi_{{}_\text{E5}}=\{E, -E  \}$, we can easily obtain that $
\{(B,1)\,|\, \varphi_{4137}(B, 1)=1 \}
=1 \}=\{ (E, 1), (-E, 1) \}$. 
In the latter case, from $-1=\varphi_{{}_\text{E5}}(e_1 E)$, it is not difficult to see that $ \{(B,-1)\,|\, \varphi_{4136}(B, -1)=1 \}
=\{(e_1, -1), (-e_1, -1)\}$. However, since this is contrary to $B \in S\!O(8)$, this case is impossible. 
Hence we have $\Ker \varphi_{4137}=\{ (E, 1), (-E, 1) \} \cong (\Z_2, 1)$.

Therefore we have the required isomorphism 
$$
(E_7)^{\lambda\gamma} \cap (E_7)^{\iota\gamma_{{}_{\scriptscriptstyle {C}}}} \cong S\!O(8)/\Z_2 \times \vspace{-3mm}\mathcal{Z}_2.
$$
\end{proof}

\subsection{Type EV-V-VI}

In this section, we give a pair of involutive inner automorphisms $\tilde{\lambda\gamma}$ and $\tilde{\lambda\gamma\sigma}$.
\vspace{1mm}

Using the inclusion $E_6 \subset E_7$, the $C$-linear transformation $\delta_5$ used in Section 4.6 is naturally extended to the $C$-linear transformation of $\mathfrak{P}^C$. Hence,  
as in $E_6$, since we easily see that $\delta_5 \gamma=(\gamma\sigma)\delta_5$ as $\delta_5 \in  E_6 \subset E_7$, that is, $\gamma \sim \gamma\sigma$ in $E_7$, we have the following proposition.

\begin{prop}
The group $(E_7)^{\lambda\gamma}$ is isomorphic to the group $(E_7)^{\lambda\gamma\sigma}${\rm :}$(E_7)^{\lambda\gamma} \cong (E_7)^{\lambda\gamma\sigma} $.
\end{prop}


From the result of types EV, EVI in Table 2 and Proposition 4.14.1, we have the following Theorem.
\begin{thm}
For $\mathbb{Z}_2 \times \mathbb{Z}_2=\{1,\lambda\gamma \} \times \{1, \lambda\gamma\sigma \}$, the $\mathbb{Z}_2 \times \mathbb{Z}_2$-symmetric space is of type $(E_7/(E_7)^{\lambda\gamma}, E_7/(E_7)^{\lambda\gamma\sigma}, E_7/(E_7)^{(\lambda\gamma)(\lambda\gamma\sigma)})=(E_7/(E_7)^{\lambda\gamma}, E_7/(E_7)^{\lambda\gamma}, E_7/(E_7)^{-\sigma})=(E_7/(E_7)^{\lambda\gamma}, E_7/(E_7)^{\lambda\gamma}, E_7/(E_7)^{\gamma})$, that is, type {\rm (EV, EV, EVI)}, abbreviated as {\rm  EV-V-VI}.
\end{thm}
Here, we prove proposition needed and make some preparations for the theorem below.
\begin{prop}
We have the following isomorphism{\rm :} $ S\!(U(4) \times U(4)) \cong (U(1) \times S\!U(4) \times S\!U(4))/\Z_4 , \Z_4=\{(1,E,E),(-1,-E,-E),(e_1,-e_1E, e_1 E),(-e_1, e_1 E, -e_1 E)  \}$.
\end{prop}
\begin{proof}
We define a mapping $g: U(1) \times S\!U(4) \times S\!U(4) \to S\!(U(4) \times U(4))$ by 
$$
g_{4}(a, A, B)=\begin{pmatrix} aA  & 0 \\
                                         0   & a^{-1}B
                \end{pmatrix}.                         
$$
Then we easily see that $g_{4}$ is well-defined and a epimorphism. By straightforward computation, $\Ker\, g_4$ is obtained as follows:
\begin{eqnarray*}
       \Ker\, g_{4} \!\!\!&=&\!\!\! \{(a, A, B) \in  U(1) \times S\!U(4) \times S\!U(4)\,|\,g_4 (a, A,B)=E \}
\\
                \!\!\!&=&\!\!\! \{(a, A, B) \in  U(1) \times S\!U(4) \times S\!U(4)\,|\,a A=a^{-1}B=E \}
\\
                 \!\!\!&=&\!\!\! \{(a, a^{-1}E, a E) \in  U(1) \times S\!U(4) \times S\!U(4)\,|\,a=\pm 1, \pm e_1 \}
 \\
                   \!\!\!&=&\!\!\! \{(1,E,E),(-1,-E,-E),(e_1,-e_1E, e_1 E),(-e_1, e_1 E, -e_1 E)  \} \cong \Z_4.
\end{eqnarray*}
Therefore we have the required isomorphism 
$$ S\!(U(4) \times U(4))  \cong (U(1) \times S\!U(4) \times S\!U(4))/\Z_4.
$$
\end{proof}
\vspace{-5mm}

We define some element $\varepsilon \in (E_7)^{\lambda\gamma}$ by 
$$
  \varepsilon =\varphi_{{}_\text{E5}}(J'),
$$
where $J'\!=\begin{pmatrix} 0 & E \\
                                         E  &   0
       \end{pmatrix} \in S\!U(8), \vspace{1mm}E=\diag(1, 1,1,1)$. Then we easily see $\varepsilon^2=1$.

Consider a group $\mathcal{Z}_2=\{1, \varepsilon \}$. Then the group $\mathcal{Z}_2=\{1, \varepsilon \}$ acts on the group $S\!(U(4) \times U(4))$ \vspace{-1mm}by 
$$
      \varepsilon(A)=J' A J', \, (J')^2=E(=\diag(1,1,1,1,1,1,1,1)),
$$
and let $S\!(U(4) \times U(4)) \rtimes \mathcal{Z}_2$ be the semi-direct product of $S\!(U(4) \times U(4))$ and $\mathcal{Z}_2$ with this action.
\vspace{1mm}

Now, we determine the structure of the group $(E_7)^{\lambda\gamma} \cap (E_7)^{\lambda\gamma\sigma}$.
\begin{thm}
We have that $(E_7)^{\lambda\gamma} \cap (E_7)^{\lambda\gamma\sigma} \cong (U(1) \times S\!U(4) \times S\!U(4))/(\Z_2 \times \Z_4) \rtimes \mathcal{Z}_2, \Z_2=\{(1, E, E), (1, -E, -E)  \}, \Z_4=\{(1,E,E),(-1,-E,-E),(e_1,-e_1E, e_1 E),(-e_1, e_1 E, $ $ -e_1 E)  \}, \mathcal{Z}_2=\{1,\varepsilon \}$. 
\end{thm}
\begin{proof}
We define a mapping $\varphi_{4144}: (U(1) \times S\!U(4) \times S\!U(4)) \rtimes \{1, \varepsilon \} \to (E_7)^{\lambda\gamma} \cap (E_7)^{\lambda\gamma\sigma}$ by 
\begin{eqnarray*}
             && \varphi_{4144}((a, A, B), 1)=\varphi_{{}_\text{E5}}(g_4 (a, A, B)),
\\
             && \varphi_{4144}((a, A, B), \varepsilon)=\varphi_{{}_\text{E5}}(g_4 (a, A, B)) \varepsilon,
\end{eqnarray*}
where $g_4$ is defined in Proposition 4.14.3 above.
Since the mapping $\varphi_{4144}$ is the restriction of the mapping  $\varphi_{{}_\text{E5}}$ and $\varepsilon \in (E_7)^{\lambda\gamma}$, it is clear that $\varphi_{4144}((a, A, B), 1), \varphi_{4144}((a, A,$ $ B), \varepsilon) \in (E_7)^{\lambda\gamma}$, moreover using $\sigma \varphi_{{}_\text{E5}}(L) \sigma=\varphi_{{}_\text{E5}}(I_4 L I_4), L \in S\!U(8)$ (Lemma 4.13.6 (2)), it is easily to verify that  $\varphi_{4144}((a, A, B), \!1), $ $ \varphi_{4144}((a, A, B), \varepsilon) \in (E_7)^{\lambda\gamma\sigma}$\!\!. Hence $\varphi_{4144}$ is well-defined. Using $\varepsilon=\varphi_{{}_\text{E5}}(J')$, we can confirm that $\varphi_{4144}$ is a homomorphism. Indeed, we show the case of $\varphi_{4144}((a_1, A_1, B_1), \varepsilon) \varphi_{4144}((a_2, A_2, B_2), 1)=\varphi_{4144}(((a_1, A_1, B_1), \varepsilon)((a_2, A_2,$ $ B_2), 1 ))$ as example. 
For the left hand side of this equality, we have that
\begin{eqnarray*}
\varphi_{4144}((a_1, A_1, B_1), \varepsilon) \varphi_{4144}((a_2, A_2, B_2), 1)\!\!\!&=&\!\!\! \varphi_{{}_\text{E5}}(g_4 (a_1, A_1, B_1)) \varepsilon \varphi_{{}_\text{E5}}(g_4 (a_2, A_2, B_2))
\\
\!\!\!&=&\!\!\! \varphi_{{}_\text{E5}}(g_4 (a_1, A_1, B_1)) \varphi_{{}_\text{E5}}(J') \varphi_{{}_\text{E5}}(g_4 (a_2, A_2, B_2))
\\
\!\!\!&=&\!\!\! \varphi_{{}_\text{E5}}(g_4 (a_1, A_1, B_1) J' (g_4 (a_2, A_2, B_2) )
\\
\!\!\!&=&\!\!\! \varphi_{{}_\text{E5}}(g_4 (a_1, A_1, B_1) J' (g_4 (a_2, A_2, B_2)J' J' )
\\
\!\!\!&=&\!\!\! \varphi_{{}_\text{E5}}(g_4 (a_1, A_1, B_1) J' (g_4 (a_2, A_2, B_2)J')\varphi_{{}_\text{E5}}(J')
\\
\!\!\!&=&\!\!\! \varphi_{{}_\text{E5}}(g_4 (a_1, A_1, B_1) J' (g_4 (a_2, A_2, B_2)J') \varepsilon
\\
\!\!\!&=&\!\!\! \varphi_{{}_\text{E5}}(g_4 (a_1, A_1, B_1)  (g_4 ({a_2}^{-1}, B_2, A_2)) \varepsilon
\\
\!\!\!&=&\!\!\! \varphi_{{}_\text{E5}}(g_4 (a_1{a_2}^{-1}, A_1 B_2 , B_1 A_2 )  ) \varepsilon
\\
\!\!\!&=&\!\!\! \varphi_{4144}( (a_1{a_2}^{-1}, A_1 B_2 , B_1A_2 ), \varepsilon).
\end{eqnarray*}
On the other hand, since the action  of $\varepsilon$ to the group $S\!(U(4) \times U(4))$ is $\varepsilon (\begin{pmatrix} S & 0 \\
                                                         0 & T 
                       \end{pmatrix})=\begin{pmatrix} T & 0 \\
                                                                       0 & S 
                       \end{pmatrix}$, $\varepsilon$ acts on the group $U(1) \times S\!U(4) \times S\!U(4)$ as follows:
$$
     \varepsilon(a,A, B )=(a^{-1}, B, A).
$$
Hence, for the right hand side of same one, we have that  
\begin{eqnarray*}
\,\varphi_{4144}(((a_1, A_1, B_1), \varepsilon)((a_2, A_2, B_2), 1 ))
\!\!\!&=&\!\!\! \varphi_{4144}( (a_1, A_1, B_1) \varepsilon (a_2, A_2, B_2), \varepsilon)
\\
\!\!\!&=&\!\!\! \varphi_{4144}( (a_1, A_1, B_1)  ({a_2}^{-1}, B_2, A_2), \varepsilon)
\\
\!\!\!&=&\!\!\! \varphi_{4144}((a_1{a_2}^{-1}, A_1 B_2 , B_1A_2 ), \varepsilon).
\end{eqnarray*}
Similarly, the other cases are shown.

We shall show that $\varphi_{4144}$ is surjection. Let $\alpha \in (E_7)^{\lambda\gamma} \cap (E_7)^{\lambda\gamma\sigma}$. Since $(E_7)^{\lambda\gamma} \cap (E_7)^{\lambda\gamma\sigma}$ $ \subset (E_7)^{\lambda\gamma}$, there exists $L \in S\!U(8)$ such that $\alpha =\varphi_{{}_\text{E5}}(L)$ (Theorem 3.4.1). Moreover, from $\alpha =\varphi_{{}_\text{E5}}(L) \in (E_7)^{\lambda\gamma\sigma}$, that is, $ (\lambda\gamma\sigma)\varphi_{{}_\text{E5}}(L)(\sigma\gamma\lambda^{-1})=\varphi_{{}_\text{E5}}(L)$, using $\sigma\varphi_{{}_\text{E5}}(L) \sigma= \varphi_{{}_\text{E5}}(I_4 L I_4)$ (Lemma 4.13.6 (2)),  we have  $\varphi_{{}_\text{E5}}(I_4 L I_4)=\varphi_{{}_\text{E5}}(L)$.
Hence it follows that
$$
    I_4 L\,  I_4 =L \qquad   {\rm or} \qquad  I_4 L\, I_4 =-L.
$$
In the former case, we see that $L \in S\!(U(4) \times U(4))$. Hence, there exist $a \in U(1), A, B \in S\!U(4)$ such that $L=g_4(a, A, B)$ (Proposition 4.14.3). Thus we have  $\alpha=\varphi_{{}_\text{E5}}(g_4(a, A, B))=\varphi_{4144}((a, A, B),1)$. 
In the latter case, we take $L$ as form $L=M J', M \in S\!(U(4) \times U(4)), J'\!=\begin{pmatrix} 0 & E \\
                                        E  &   0
                 \end{pmatrix}, \vspace{1mm}E$ $=\diag(1, 1,1,1)$. 
Hence, in a similar way as above, we have that $\alpha\!=\!\varphi_{{}_\text{E5}}(g_4(a, A, $ $B)J')\!=\!\varphi_{{}_\text{E5}}(g_4(a, A, B))$ $\varphi_{{}_\text{E5}}(J')\!=\!\varphi_{{}_\text{E5}}(g_4(a, A,B))\varepsilon=\varphi_{4144}((a, A, B),\varepsilon)$. Hence  $\varphi_{4144}$ is surjection.

Finally, we shall determine $\Ker \,\varphi_{4144}$. From $\Ker\, \varphi_{{}_\text{E5}}=\{ E, -E \}$, we can easily obtain that 
\begin{eqnarray*}
\Ker\,\varphi_{4144} \!\!\!&=&\!\!\! \{((a, A, B), 1)\,|\,\varphi_{4144}((a, A, B),1)=1 \} \cup \{((a, A, B), \varepsilon)\,|\,\varphi_{4144}((a, A, B),\varepsilon)=1 \}
\\
\!\!\!&=&\!\!\!\{((a, A, B), 1)\,|\,\varphi_{{}_\text{E5}}(g_4(a, A, B))=1 \} \cup \{((a, A, B), \varepsilon)\,|\,\varphi_{{}_\text{E5}}(g_4(a, A, B))\varepsilon=1 \}
\\
\!\!\!&=&\!\!\!\ \{((a, A, B), 1)\,|\,a A=a^{-1}B=\pm E \} \cup \{((a, A, B), \varepsilon)\,|\,a A=a^{-1}B=0 \}
\\
\!\!\!&=&\!\!\! \{(a, a^{-1}E, a E), 1), (a, -a^{-1}E, -a E), 1)\,|\,a ^4=1 \} \cup \phi
\\
\!\!\!&=&\!\!\! \{((1, E, E),1), ((1, -E, -E),1)  \} 
\\
& & \times \{((1,E,E),1),((-1,-E,-E),1),((e_1,-e_1E, e_1 E),1),((-e_1, e_1 E,  -e_1 E),1)  \}
\\
\!\!\!&\cong&\!\!\! (\Z_2 \times \Z_4,1). 
\end{eqnarray*}

Therefore we have the required isomorphism 
$$
(E_7)^{\lambda\gamma} \cap (E_7)^{\lambda\gamma\sigma} \cong (U(1) \times S\!U(4) \times S\!U(4))/(\Z_2 \times \Z_4) \rtimes \mathcal{Z}_2.
$$
\end{proof}

\subsection{Type EV-V-VII}

In this section, we give a pair of involutive inner automorphisms $\tilde{\lambda\gamma}$ and $\tilde{\iota\lambda\gamma}$.
\vspace{1mm}

We have the following proposition which is the direct result of Lemmas 4.13.1, 4.13.3.

\begin{prop}
The group $(E_7)^{\lambda\gamma}$ is isomorphic to the group $(E_7)^{\iota\lambda\gamma}${\rm :} $(E_7)^{\lambda\gamma} \cong (E_7)^{\iota\lambda\gamma}$.
\end{prop}

From the result of types EV, EVII in Table 2 and Propositions 4.15.1, we have the following theorem.

\begin{thm}
For $\mathbb{Z}_2 \times \mathbb{Z}_2=\{1,\lambda\gamma \} \times \{1, \iota\lambda\gamma \}$, the $\mathbb{Z}_2 \times \mathbb{Z}_2$-symmetric space is of type $(E_7/(E_7)^{\lambda\gamma}, E_7/(E_7)^{\iota\lambda\gamma}, E_7/(E_7)^{(\lambda\gamma)(\iota\lambda\gamma)})=(E_7/(E_7)^{\lambda\gamma}, E_7/(E_7)^{\lambda\gamma}, E_7/(E_7)^{-\iota})=(E_7/(E_7)^{\lambda\gamma}, $ $E_7/(E_7)^{\lambda\gamma}, E_7/(E_7)^{\iota})$, that is, type {\rm (EV, EV, EVII)}, abbreviated as {\rm  EV-V-VII}.
\end{thm}

Here, we prove lemma needed in theorem below.

\begin{lem}
The $C$-linear  transformation $\phi(\theta)$ defined in Section 4.13 satisfies the following equalities{\rm :}
$$
  \lambda \phi(\theta)\lambda=\phi(\theta^{-1}), \quad \gamma\phi(\theta)\gamma=\phi(\theta).
$$
\end{lem}
\begin{proof}
By using  the definition of $\lambda, \gamma$ and $\phi(\theta)$ (Sections 3.4, 4.13) , it is easily to verify those. 
\end{proof}

Now, we determine the structure of the group $(E_7)^{\lambda\gamma} \cap (E_7)^{\iota\lambda\gamma}$.

\begin{thm}
We have that  $(E_7)^{\lambda\gamma} \cap (E_7)^{\iota\lambda\gamma} \cong (S\!p(4)/\Z_2) \times \mathcal{Z}_2,\Z_2\!=\!\{ E, -E \}, \mathcal{Z}_2=\{ 1, -1 \}$.
\end{thm}
\begin{proof}
We define a mapping $\varphi_{4154}: S\!p(4) \times \{ 1, -1 \} \to (E_7)^{\lambda\gamma} \cap (E_7)^{\iota\lambda\gamma}$ by
\begin{eqnarray*}
             \varphi_{4154}(A, 1)\!\!\!&=&\!\!\!\varphi_{{}_\text{EI}}(A),
\\
      \varphi_{4154}(A, -1)\!\!\!&=&\!\!\!\varphi_{{}_\text{EI}}(A)(-1),
\end{eqnarray*}
where $\varphi_{{}_\text{EI}}$ is defined in Theorem 3.3.1. (Remark. The element $\varphi_{{}_\text{EI}}(A) \in (E_6)^{\lambda\gamma}$ is identified as elements of the group $E_7$.) Since the mapping $ \varphi_{4154}$ is the restriction of the mapping $\varphi_{{}_\text{EI}}$, 
 it is clear that $\varphi_{4154}$ is well-defined and a homomorphism.  

We shall show that $\varphi_{4154}$ is surjection. Let $\alpha \in (E_7)^{\lambda\gamma} \cap (E_7)^{\iota\lambda\gamma}$. Since $\lambda\gamma\alpha\gamma {\lambda}^{-1}=\alpha$ and $\iota(\lambda\gamma\alpha\gamma {\lambda}^{-1}){\iota}^{-1}=\alpha$, we have that $\alpha \in (E_7)^{\iota}$. Hence, there exist $\theta \in U(1)$ and $\beta \in E_6$ such that $\alpha=\varphi_{{}_\text{E7}}(\theta, \beta)$ (Theorem 3.4.3). Moreover, from $\alpha \in (E_7)^{\lambda\gamma}$, that is, $(\lambda\gamma) \varphi_{{}_\text{E7}}({\theta}, \beta)(\gamma{\lambda}^{-1})=\varphi_{{}_\text{E7}}(\theta, \beta) $, using $\lambda \phi(\theta)\lambda^{-1}=\phi(\theta^{-1})$ and
 $\gamma\phi(\theta)\gamma=\phi(\theta)$ (Lemma 4.15.3), we have $\varphi_{{}_\text{E7}}({\theta}^{-1}, \lambda\gamma\beta\gamma{\lambda}^{-1})\!=\!\varphi_{{}_\text{E7}}(\theta, \beta)$. Then, as the argument above, we also see $\alpha \!\in \!(E_7)^{\iota\lambda\gamma}$. Hence, it follows that
 $$
 \text{(i)}\,\left \{
         \begin{array}{l}
                {\theta}^{-1}= \theta
                         \vspace{3mm}\\
                \lambda\gamma\beta\gamma{\lambda}^{-1} = \beta,
         \end{array}\right.\quad 
\text{(ii)}\,\left \{         
          \begin{array}{l}
        {\theta}^{-1}= \omega\theta
                         \vspace{3mm}\\
                \lambda\gamma\beta\gamma{\lambda}^{-1} = \phi({\omega}^2)\beta,
         \end{array}\right.\quad 
 \text{(iii)}\,\left \{         
          \begin{array}{l}
        {\theta}^{-1}= {\omega}^2\theta
                         \vspace{3mm}\\
                \lambda\gamma\beta\gamma{\lambda}^{-1} = \phi(\omega)\beta, 
         \end{array}\right. 
 $$
 where $\omega \in C, \omega^3=1, \omega \ne 1 $. For these cases above, we have the following results.
 \vspace{1mm}
 
Case (i). We have that $\theta =-1 \,\,\text{or}\,\, \theta =1$ and $ \beta \in (E_6)^{\lambda\gamma}$. Hence, in the case of $\theta=1$, there exists $A \in S\!p(4)$ such that  $\alpha=\varphi_{{}_\text{E7}}(1, \beta)=\beta=\varphi_{{}_\text{E1}}(A)=\varphi_{4154}(A, 1)$ (Theorem 3.3.1), and in the case of $\theta =-1$, similarly there exists $A \in S\!p(4)$ such that  $\alpha=\varphi_{{}_\text{E7}}(-1, \beta)=\phi(-1)\beta=(-1)\beta=\varphi_{{}_\text{E1}}(A)(-1)=\varphi_{4154}(A, -1)$.
\vspace{1mm}

Case (ii). We have that $\theta =-\omega \,\,\text{or}\,\, \theta =\omega$ and $ \beta =\phi(\omega^2)\beta', \beta' \in (E_6)^{\lambda\gamma}$.
Hence, in the case of $\theta=\omega$,  there exists $A' \in S\!p(4)$ such that  $\alpha=\varphi_{{}_\text{E7}}(\omega, \phi(\omega^2)\beta')=\phi(\omega)(\phi(\omega^2)\beta')=\beta'=\varphi_{{}_\text{E1}}(A')=\varphi_{4154}(A', 1)$ (Theorem 3.3.1), and in the case of $\theta =-\omega$, similarly there exists $A' \in S\!p(4)$ such that  $\alpha=\varphi_{{}_\text{E7}}(-\omega, \phi(\omega^2)\beta')=\phi(-\omega)(\phi(\omega^2)\beta')\!=\!(-1)\beta'=\varphi_{{}_\text{E1}}(A')(-1)=\varphi_{4154}(A', -1)$.  As a result, this case  is reduced to  Case (i).
\vspace{1mm}

Case (iii). We have that $\theta =-\omega^2 \,\,\text{or}\,\, \theta =\omega^2$ and $ \beta \in \phi(\omega)\beta', \beta' \in (E_6)^{\lambda\gamma}$.  Hence we have the same result as Case (ii), that is, this case is also reduced to  Case (i). 
  
\noindent Thus $\varphi_{4154}$ is surjection. 
Finally, we shall determine $\Ker \, \varphi_{4154}$. From $\Ker \,\varphi_{{}_\text{E1}}=\{ E, -E \}$, we can easily obtain that
\begin{eqnarray*}
         \Ker \,\varphi_{4154}\!\!\!&=&\!\!\! \{(A, 1)\,|\, \varphi_{4154}(A, 1)=1 \} \cup \{(A, 1)\,|\, \varphi_{4154}(A, -1)=1 \}
\\
\!\!\!&=&\!\!\! \{(A, 1)\,|\, \varphi_{{}_\text{E1}}(A)=1 \} \cup \{(A, -1)\,|\, \varphi_{{}_\text{E1}}(A)(-1)=1 \}
\\
\!\!\!&=&\!\!\! \{ (E, 1), (-E, 1)\} \cup \phi
\\
\!\!\!&=&\!\!\! \{ (E, 1), (-E, 1)\} \cong (\Z_2, 1).
\end{eqnarray*}

Therefore we have the required isomorphism 
$$
(E_7)^{\lambda\gamma} \cap (E_7)^{\iota\lambda\gamma} \cong (S\!p(4)/\Z_2) \times \mathcal{Z}_2.
$$
\end{proof}

\subsection{Type EV-VI-VII}

In this section, we give a pair of involutive inner automorphisms $\tilde{\lambda\gamma}$ and $\tilde{\gamma}$.
\vspace{1mm}

We have the following proposition which is the direct result of Lemma 4.13.1.

\begin{prop}
The group $(E_7)^\lambda$ is isomorphic to the group $(E_7)^\iota$ {\rm :} $(E_7)^\lambda \cong (E_7)^\iota$.
\end{prop}

From the result of types EV, EVI, EVII in Table 2 and Proposition 4.16.1, we have the following theorem.
\vspace{2mm}
\begin{thm}
For $\mathbb{Z}_2 \times \mathbb{Z}_2=\{1,\lambda\gamma \} \times \{1, \gamma \}$, the $\mathbb{Z}_2 \times \mathbb{Z}_2$-symmetric space is of type $(E_7/(E_7)^{\lambda\gamma}, E_7/(E_7)^{\gamma}, E_7/(E_7)^{(\lambda\gamma)(\gamma)})\!=\!(E_7/(E_7)^{\lambda\gamma}, E_7/(E_7)^{\gamma}, E_7/(E_7)^{\lambda})\!=\!(E_7/(E_7)^{\lambda\gamma}, $ $   E_7/(E_7)^{\gamma}, E_7/(E_7)^{\iota})$, that is, type {\rm (EV, EVI, EVII)}, abbreviated as {\rm  EV-VI-VII}.
\end{thm}
\vspace{1mm}

Here, we prove proposition needed in theorem below.
\begin{prop}
We have the following isomorphism{\rm :}\vspace{1mm} $S\!(U(2) \times U(6)) \cong (U(1) \times S\!U(2) \times S\!U(6))/\Z_{12}, \Z_{12}=\{(e^{e_1 \frac{2\pi}{12}k}, e^{-e_1 \frac{12\pi}{12}k}E,e^{e_1 \frac{4\pi}{12}k}E) \,|\,k=0,1,2, \ldots, 11   \}$.
\end{prop}
\begin{proof}
We define a mapping $f_6: U(1) \times S\!U(2) \times S\!U(6) \to S\!(U(2) \times U(6))$ by 
$$
f_{6}(a, A, B)=\begin{pmatrix} a^6 A  & 0 \\
                                         0   & a^{-2}B
                \end{pmatrix}.                         
$$
Then we easily see that $f_{6}$ is well-defined and a epimorphism.  By straightforward computation, $\Ker\, f_6$ is obtained as follows:
\begin{eqnarray*}
       \Ker\, f_{6} \!\!\!&=&\!\!\! \{(a, A, B) \in  U(1) \times S\!U(2) \times S\!U(6)\,|\,f_ 6 (a, A,B)=E \}
\\
                \!\!\!&=&\!\!\! \{(a, A, B) \in  U(1) \times S\!U(2) \times S\!U(6)\,|\,a^6 A=a^{-2}B=E \}
\\
                 \!\!\!&=&\!\!\! \{(a, a^{-6}E, a^2 E) \in  U(1) \times S\!U(2) \times S\!U(6)\,|\,a^{12}=1 \}
 \\
                   \!\!\!&=&\!\!\! \{ (e^{e_1 \frac{2\pi}{12}k}, e^{-e_1 \frac{12\pi}{12}k}E,e^{e_1 \frac{4\pi}{12}k}E) \,|\,k=0,1,2, \ldots, 11  \} \cong \Z_{12}.
\end{eqnarray*}
Therefore we have the required isomorphism 
$$
S\!(U(2) \times U(6)) \cong (U(1) \times S\!U(2) \times S\!U(6))/\Z_{12}.
$$
\end{proof}

Now, we determine the structure of the group $(E_7)^{\lambda\gamma} \cap (E_7)^{\gamma}$.

\begin{thm}
We have that $(E_7)^{\lambda\gamma} \cap (E_7)^{\gamma} \cong  (U(1) \times S\!U(2) \times S\!U(6))/\Z_{24}, \,\Z_{24}=\{ (a, a^6 E, a^{-2} E)\,|\, a=e^{e_1 \frac{\pi}{12}k}, k=0, 1, 2,\ldots, 23 \}$.
\end{thm}
\begin{proof}
We define a mapping $\varphi_{4164}:  U(1) \times S\!U(2) \times S\!U(6) \to (E_7)^{\lambda\gamma} \cap (E_7)^{\gamma}$ by 
$$
 \varphi_{4164}(f_6 (a, A, B)) =\varphi_{{}_\text{E5}}(f_6 (a, A, B)) .  
$$
 Since the mapping $\varphi_{4164}$ is the restriction of the mapping $\varphi_{{}_\text{E5}}$, it is clear that $\varphi_{4164}(\!f_6 (a, A, B)) $ $\in (E_7)^{\lambda\gamma}$, and using $\gamma\varphi_{{}_\text{E5}}(L)\gamma=\varphi_{{}_\text{E5}}(I_2 L I_2) , L \in S\!U(8)$  (Lemma 4.13.6 (2)), it is easily to verify that $\varphi_{4164}(f_6 (a, A, B)) \!\in (E_7)^{\gamma}$. Hence, $\varphi_{4164}$ is well-defined. Again, since the mapping $\varphi_{4164}$ is the restriction of the mapping $\varphi_{{}_\text{E5}}$, it is clear that $\varphi_{4164}$ is a homomorphism. 

We shall show that $\varphi_{4164}$ is  surjection. Let $\alpha \in (E_7)^{\lambda\gamma} \cap (E_7)^{\gamma}$. 
From $(E_7)^{\lambda\gamma} \cap (E_7)^{\gamma} \subset  (E_7)^{\lambda\gamma}$, there exists $L \in S\!U(8)$ such that $\alpha=\varphi_{{}_\text{E5}}(L)$ (Theorem 3.4.1). Moreover, from $\alpha \in (E_7)^\gamma$, that is, $ \gamma\varphi_{{}_\text{E5}}(L)\gamma=\varphi_{{}_\text{E5}}(L)$, again using $\gamma\varphi_{{}_\text{E5}}(L)\gamma=\varphi_{{}_\text{E5}}(I_2 L I_2) $  (Lemma 4.13.6 (2)),  we have that $\varphi_{{}_\text{E5}}(I_2 L\, I_2)=\varphi_{{}_\text{E5}}(L)$. Hence, it follows that
$$
     I_2 L\, I_2=L \qquad {\rm or}  \qquad I_2 L\, I_2=-L.
$$
In the former case, we see that $L \in S\!(U(2) \times U(6))$. Hence, there exist $a \in U(1), A \in S\!U(2)$ and $S\!U(6)$ such that $L=f_6(a, A, B)$ (Proposition 4.16.3).  Thus we have that $\alpha=\varphi_{{}_\text{E5}}(L)=\varphi_{{}_\text{E5}}(f_6(a, A, B))=\varphi_{4164}(a, A, B)$. 
In the latter case, as the former case, we can also find the explicit form of $L \in S\!U(8)$ as follows: 
$$
   L=\begin{pmatrix}0 & C \\
                              D & 0
       \end{pmatrix} ,\,C \in M(2,6,\C), D \in M(6,2,\C).
$$
This case is impossible because of $\det\, L=0$. 
Thus $\varphi_{4164}$ is surjection. 

Finally, we shall determine $\Ker \, \varphi_{4164}$. 
From $\Ker \,\varphi_{{}_\text{E5}}=\{ E, -E\}$, we can easily obtain that
\begin{eqnarray*}
\Ker \,\varphi_{4164}\!\!\!&=&\!\!\! \{(a, A, B) \in U(1) \times S\!U(2) \times S\!U(6)\,|\, \varphi_{4163}(a, A, B) =1\}
\\
\!\!\!&=&\!\!\! \{(a, A, B) \in U(1) \times S\!U(2) \times S\!U(6)\,|\, \varphi_{{{}_\text{E5}}}(f_6 (a, A, B)) =1\}
\\
\!\!\!&=&\!\!\!\{(a, A, B) \in U(1) \times S\!U(2) \times S\!U(6)\,|\, f_6 (a, A, B)=E, f_6 (a, A, B))=-E\}
\\
\!\!\!&=&\!\!\!\{(a, a^6 E, a^{-2}E) \in U(1) \times S\!U(2) \times S\!U(6)\,|\, a^{12}=1, a^{12}=-1 \}
\\
\!\!\!&=&\!\!\! \{ (a, a^6 E, a^{-2} E)\,|\, a=e^{e_1 \frac{\pi}{12}k}, k=0, 1, 2,\ldots, 23 \} \cong \Z_{24}.
\end{eqnarray*}

Therefore we have the required isomorphism 
$$
(E_7)^{\lambda\gamma} \cap (E_7)^{\gamma} \cong  (U(1) \times S\!U(2) \times S\!U(6))/\Z_{24}.
$$
\end{proof}

\subsection{Type EVI-VI-VI}
In this section, there exist two cases with this type. 
\subsubsection{ }

We begin from the first case: $\mathfrak{k}=\mathfrak{so}(8) \oplus \mathfrak{so}(4) \oplus i\R$. 
 \vspace{1mm}
 
 In the first case, we give a pair of involutive inner automorphisms $\tilde{\gamma}$ and $\tilde{-\sigma}$.
We remark that $\tilde{-\sigma}$ is same as $\tilde{\sigma}$ because of $-1 \in z(E_7)$ (the center of $E_7$) .
Again, we state $\gamma \sim \gamma\sigma, \gamma \sim -\sigma$ in $E_7$ as mentioned in Sections 4.14, 3.4, respectively.
\vspace{1mm}

From the result of type EVI in Table 2 and $\gamma \sim \gamma\sigma, \gamma \sim -\sigma$
, we have the following theorem.
\vspace{2mm}

\noindent {\bf Remark.}\,\, From $-1 \in z(E_7)$, it is clear that 
$(E_7)^{-\gamma\sigma}=(E_7)^{\gamma\sigma}$.
\vspace{2mm}

\noindent {\bf Theorem 4.17.1-1}
{\it \, For $\mathbb{Z}_2 \times \mathbb{Z}_2=\{1,\gamma \} \times \{1, -\sigma \}$, the $\mathbb{Z}_2 \times \mathbb{Z}_2$-symmetric space is of type $(E_7/(E_7)^{\gamma}, E_7/(E_7)^{-\sigma}, E_7/(E_7)^{(\gamma)(-\sigma)})\!=\!(E_7/(E_7)^{\gamma}, E_7/(E_7)^{\gamma}, E_7/(E_7)^{\gamma\sigma})\!=\!(E_7/(E_7)^{\gamma},$ $ E_7/(E_7)^{\gamma}, E_7/(E_7)^{\gamma})$, that is, type {\rm (EVI, EVI, EVI)}, abbreviated as {\rm  EVI-VI-VI}.}
\vspace{2mm}

Now, we determine the structure of the group $(E_7)^\gamma \cap (E_7)^{-\sigma} =(E_7)^\gamma \cap (E_7)^\sigma$.
\vspace{2mm}

\noindent {\bf Theorem 4.17.1-2}
{\it \, We have that $(E_7)^\gamma \cap (E_7)^{-\sigma} = (E_7)^\gamma \cap (E_7)^\sigma \cong (S\!U(2) \times S\!pin(4) \times S\!pin(8))/(\Z_2 \times \Z_2),  \Z_2 \times \Z_2= \{(E,1,1), (E, \sigma, \sigma)  \} \times \{(E, 1, 1),(-E, \gamma, -\sigma\gamma)  \}$.} 
\begin{proof}
Let $S\!pin(4) \cong (((E_7)^{\kappa,\, \mu})^\gamma)_{{\dot{F_1}(h),\ti{E}_1,\ti{E}_{-1},\dot{E}_{23}}}$ and $S\!pin(8) \cong  ({((E_7)^{\kappa,\, \mu})}^\gamma)_{\dot{F_1}(h{e_4})}$. Since \vspace{0.5mm}both of the groups $S\!pin(4)$ and $S\!pin(8)$ are the subgroups of $S\!pin(12) \cong (E_7)^{\kappa,\,\mu}$, we can define a mapping $\varphi_{4171-2}: S\!U(2)\times S\!pin(4)\times S\!pin(8) \to (E_7)^{\gamma} \cap (E_7)^{\sigma}$ as the restriction of the mapping $\varphi_{{}_{\rm E6}}$ as follows:
$$
  \varphi_{4171-2}(A, \beta_4, \beta_8) = \phi_2(A)\beta_4 \beta_8. 
$$
Then this mapping induces the required isomorphism (see \cite[Theorem 3.23]{M.01}in detail).\vspace{1mm}
\end{proof}

\subsubsection{ }
Next, we study the second case: $\mathfrak{k}=\mathfrak{u}(6) \oplus i\R$.
\vspace{1mm}

 In the second case, we give a pair of involutive inner automorphisms $\tilde{\gamma}$ and $\tilde{\gamma_{{}_{\scriptscriptstyle{H}}}}$.
\vspace{1mm}

We define  $C$-linear transformations $\gamma_{{}_{\scriptscriptstyle{H}}}$  of $\mathfrak{P}^C$ by
\begin{eqnarray*}
\gamma_{{}_{\scriptscriptstyle{H}}}(X, Y, \xi, \eta)\!\!\!&=&\!\!\!(\gamma_{{}_{\scriptscriptstyle{H}}} X, \gamma_{{}_{\scriptscriptstyle{H}}} Y, \xi, \eta), 
\end{eqnarray*}
where $\gamma_{{}_{\scriptscriptstyle{H}}}$ of the right hand side are the same ones as $\gamma_{{}_{\scriptscriptstyle{H}}} \in G_2 \subset F_4 \subset  E_6$.
 Then we have that $\gamma_{{}_{\scriptscriptstyle{H}}} \in  E_7, {\gamma_{{}_{\scriptscriptstyle{H}}}}^2=1$, so $\gamma_{{}_{\scriptscriptstyle{H}}}$ induce involutive inner automorphism ${\tilde{\gamma}}_{{}_{\scriptscriptstyle{H}}} $ of $E_7$: 
${\tilde{\gamma}}_{{}_{\scriptscriptstyle{H}}}(\alpha)= \gamma_{{}_{\scriptscriptstyle{H}}} \alpha \gamma_{{}_{\scriptscriptstyle{H}}}
\alpha \in E_7$. 

\noindent Similarly, for $\delta_1, \delta_2 \in  G_2 \subset F_4 \subset E_6$, we have $\delta_1, \delta_2 \in E_7$. Hence, as in $E_6$, since we easily see that $\delta_1 \gamma=\gamma_{\scriptscriptstyle{H}} \delta_1, \delta_2 \gamma =(\gamma\gamma_{\scriptscriptstyle{H}}) \delta_2$, that is, $\gamma \sim \gamma_{\scriptscriptstyle{H}}, \gamma \sim \gamma\gamma_{\scriptscriptstyle{H}} $ in $E_7$, we have the following proposition.
\vspace{1mm}

\noindent {\bf Proposition 4.17.2-1}
{ \it The group $(E_7)^\gamma$ is isomorphic to both of the groups $(E_7)^{\gamma_{{}_{\scriptscriptstyle {H}}}}$ and $(E_7)^{\gamma \gamma_{{}_{\scriptscriptstyle {H}}}}${\rm:} $(E_7)^\gamma \cong (E_7)^{\gamma_{{}_{\scriptscriptstyle {H}}}} \cong (E_7)^{\gamma \gamma_{{}_{\scriptscriptstyle {H}}}}$.}
\vspace{1mm}

From the result of type EVI in Table 2 and Proposition 4.17.2-1, we have the following theorem.
\vspace{1mm}

\noindent {\bf Theorem 4.17.2-2}
{\it \,For $\mathbb{Z}_2 \times \mathbb{Z}_2=\{1,\gamma \} \times \{1, \gamma_{{}_{\scriptscriptstyle{H}}} \}$, the $\mathbb{Z}_2 \times \mathbb{Z}_2$-symmetric space is of type $(E_7/(E_7)^{\gamma}, E_7/(E_7)^{\gamma_{{}_{\scriptscriptstyle{H}}}}, E_7/(E_7)^{\gamma\gamma_{{}_{\scriptscriptstyle{H}}}})=(E_7/(E_7)^{\gamma}, E_7/(E_7)^{\gamma}, E_7/(E_7)^{\gamma})$, that is, type {\rm (EVI, EVI, EVI)}, abbreviated as {\rm  EVI-VI-VI}.}
\vspace{2mm}

Here, we prove proposition needed and make some preparations for the theorem \vspace{1mm}below.

First, using identifying 
$(\mathfrak{J}_{{}_{\sC}})^C \oplus M(3, \C)^C$ with 
$\mathfrak{J}^C$, we identify $(\mathfrak{P}_{{}_{\sC}})^C \oplus (M(3, \C)^C \oplus M(3, \C)^C)$ with $\mathfrak{P}^C$  by
$$
   (X, Y, \xi, \eta) + (M, N) = (X + M, Y + N, \xi, \eta), 
$$
where $(\mathfrak{P}_{{}_{\sC}})^C=(\mathfrak{J}_{{}_{\sC}})^C \oplus (\mathfrak{J}_{{}_{\sC}})^C \oplus C \oplus C$. (As for identifying $(\mathfrak{J}_{{}_{\sC}})^C \oplus M(3, \C)^C$ with $\mathfrak{J}^C$ and the definition of $(\mathfrak{J}_{{}_{\sC}})^C$, see \cite[Sections 2.2, 2.3]{M.02}.)

\noindent We often denote any element of $M(3, \C)^C$ by $(\m_1, \m_2, \m_3)$, where  $\m_k \in (\C^3)^C, k=1, 2, 3$, moreover denote any element of $(\mathfrak{P}_{{}_{\sC}})^C$ by $(X, Y, \xi, \eta)$ as above. (Remark. we often denote any element of $\mathfrak{P}^C$\vspace{1mm} by same one.)
\vspace{1mm}

We define a $C$-linear transformation $w$ of
$(\mathfrak{P}_{{}_{\sC}})^C \oplus (M(3, \C)^C \oplus M(3, \C)^C)=\mathfrak{P}^C$ by
\begin{eqnarray*}
&& w((X, Y, \xi, \eta)+(M,N)) = (X, Y, \xi, \eta)+(\omega_1 M, \omega_1 N),
\\[1mm]
&& \hspace*{10mm}(X, Y, \xi, \eta)+(M,N) \in (\mathfrak{P}_{{}_{\sC}})^C \oplus (M(3, \C)^C \oplus M(3, \C)^C)=\mathfrak{P}^C, 
\end{eqnarray*}
where $\omega_1 M=(\omega_1 \m_1, \omega_1 \m_2, \omega_1 \m_3),\omega_1 \in \C, {\omega_1}^3=1,\omega_1 \ne 1 $, so is $\omega_1 N$. 

\noindent Besides, $w$ is defined as the $C$-linear transformation of $\C \oplus \C^3=\mathfrak{C}$ as follows:
$$
    w\,(a+\m)=a+\omega_1 \m, \,\, a+\m \in \C \oplus \C^3=\mathfrak{C}.
$$ 
Then we have that $w \in G_2, w^3=1$. Hence, using the inclusion $G_2 \subset F_4 \subset E_6 \subset E_7$, $w$ induces inner automorphism $\tilde{w}$ of order 3 in $E_7$: $\tilde{w}\,(\alpha)=w^{-1}\alpha w, \alpha \in E_7$. 
\vspace{2mm}

\noindent {\bf Proposition 4.17.2-3}
{\it \, We have the following isomorphism{\rm :}$(E_7)^w \cong (S\!U(3) \times S\!U(6))/\Z_3, \Z_3 $ $= \{(E, E), (\omega_1E, \omega_1E),({\omega_1}^2E, {\omega_1}^2E)\}$.}
\begin{proof}
We define a mapping $\varphi_{w} : S\!U(3) \times S\!U(6) \to (E_7)^w$ by
$$
    \varphi_{w}(D, A)P = f^{-1}((D, A)(fP)), \,\, P \in \mathfrak{P}^C. 
$$
Then this mapping induces the required isomorphism (see \cite[Section 4.13]{Yokotaichiro0} in detail).
\end{proof}
By identifying $(\mathfrak{J}_{\sC})^C \oplus M(3, \C)^C$ with $\mathfrak{J}^C$, the $C$-linear transformation $\gamma,  \gamma_{{}_{\scriptscriptstyle{H}}}$ and $ \gamma_{{}_{\scriptscriptstyle{C}}}$ of $\mathfrak{J}^C$ naturally act on $(\mathfrak{J}_{\sC})^C \oplus M(3, \C)^C=\mathfrak{J}^C$.
Hence,  using the inclusion $E_6 \subset E_7$, the $C$-linear transformations $\gamma, \gamma_{{}_{\scriptscriptstyle{H}}}$ and $\gamma_{{}_{\scriptscriptstyle{C}}}$ of $(\mathfrak{J}_{\sC})^C \oplus M(3, \C)^C = \mathfrak{J}^C$ are naturally extended to the $C$-linear transformations of $(\mathfrak{P}_{\sC})^C \oplus (M(3, \C)^C \oplus M(3, \C)^C) = \mathfrak{P}^C$ as follows:
\begin{eqnarray*}
   \gamma((X, Y, \xi, \eta)+(M,N)) \!\!\!&=&\!\!\! (X, Y, \xi, \eta)+(\gamma M, \gamma N), 
\\[1mm]
  \gamma_{{}_{\scriptscriptstyle{H}}} ((X, Y, \xi, \eta)+(M,N)) \!\!\!&=&\!\!\! (X, Y, \xi, \eta)+(\gamma_{{}_{\scriptscriptstyle{H}}} M, \gamma_{{}_{\scriptscriptstyle{H}}} N),
\\[1mm]
  \gamma_{{}_{\scriptscriptstyle{C}}}((X, Y, \xi, \eta)+(M,N)) \!\!\!&=&\!\!\! (\ov{X}, \ov{Y}, \xi, \eta)+(\ov{M}, \ov{N}),
\end{eqnarray*}
where $\gamma M=\gamma (\m_1, \m_2, \m_3)=(\gamma \m_1, \gamma \m_2, \gamma \m_3)$, and so is $\gamma_{\scriptscriptstyle{H}} M$. In addition, for $\m=(m_1, m_2, m_3) \in \C^3$, $\gamma \m$ and $ \gamma_{\scriptscriptstyle{H}} \m$ are defined by $(m_1, -m_2, -m_3)$ and $(-m_1, m_2, -m_3)$: 
$
\gamma \m=(m_1, -m_2, -m_3), \gamma_{\scriptscriptstyle{H}} \m=(-m_1, m_2, -m_3)
$, respectively.
\vspace{1mm}

Consider a group $\mathcal{Z}_2 = \{1, \gamma_{{}_{\scriptscriptstyle{C}}}\}$. Then the group $\mathcal{Z}_2 = \{1, \gamma_{{}_{\scriptscriptstyle{C}}} \}$ acts on the group $U(1) \times U(1) \times S\!U(6)$ by
$$
   \gamma_{{}_{\scriptscriptstyle{C}}}(p, q, A) = (\ov{p}, \ov{q}, \ov{(\mbox{Ad}J_3)A}) 
$$
and let $(U(1) \times U(1) \times S\!U(6)) \rtimes \mathcal{Z}_2$ be the semi-direct product of $U(1) \times U(1) \times S\!U(6)$ and $\mathcal{Z}_2$ with this action.
\vspace{2mm}

Now, we determine the structure of the group $(E_7)^\gamma \cap (E_7)^{\gamma_{{}_{\scriptscriptstyle{H}}}}$.
\vspace{2mm}

\noindent {\bf Theorem 4.17.2-4}
{\it We have that $(E_7)^{\gamma} \cap (E_7)^{\gamma_{{}_{\scriptscriptstyle{H}}}} \cong (U(1) \times U(1) \times S\!U(6))/\Z_3  \rtimes \mathcal{Z}_2, \Z_3 =\{(1, 1, E) , (\omega_1,\omega_1, \omega_1 E), ({\omega_1}^2, {\omega_1}^2, {\omega_1}^2 E) \} ,   \mathcal{Z}_2=\{1,  \gamma_{\scriptscriptstyle{C}} \}$, where $\omega_1 \in \C, {\omega_1}^3=1, \omega_1 \ne 1$.}
\begin{proof}
We define a mapping $\varphi_{4172-4}: (U(1) \times U(1) \times S\!U(6)) \rtimes \{1,  \gamma_{{}_{\scriptscriptstyle{C}}} \} \to (E_7)^{\gamma} \cap (E_7)^{\gamma_{{}_{\scriptscriptstyle{H}}}}$ by
\begin{eqnarray*}
     \varphi_{4172-4}((p, q, A), 1)\!\!\!&=&\!\!\! \varphi_{w}(D(p, q), A), 
\\[1mm]  
     \varphi_{4172-4}((p, q, A), \gamma_{{}_{\scriptscriptstyle{C}}}) \!\!\!&=&\!\!\! \varphi_{w}(D(p, q), A)\gamma_{{}_{\scriptscriptstyle{C}}}
,
\end{eqnarray*} 
where $D(p, q) = \diag(p, q, \ov{pq}) \in S\!U(3)$ and $\varphi_w$ is defined in Proposition 4.17.2-2.

\noindent Then this mapping induces the required isomorphism (see \cite[Theorem 2.4.3]{M.02} in detail).
\end{proof}
\vspace{1mm}

\subsection{Type EVI-VII-VII}

In this section, we give a pair of involutive inner automorphisms $\tilde{-\sigma}$ and $\tilde{\iota}$.
\vspace{1mm}

\begin{lem}
In $E_7$, $\iota$ is conjugate to $-\sigma\iota${\rm :} $\iota \sim -\sigma\iota$.
\end{lem}
\begin{proof}
We define a $C$-linear transformation $\delta_{10}$ of $\mathfrak{P}^C$ by 
$$
\delta_{10}
\begin{pmatrix}
                                                 X 
\\
                                                 Y
\\
                                               \xi
\\
                                               \eta
\end{pmatrix}
                         =
\begin{pmatrix}
                         (1-p_1) X-2 E_1 \times Y +\eta E_1
\\
                         2 E_1 \times X+(1-p_1)Y+ \xi  E_1 
\\
                          (E_1, Y)
\\
                          (- E_1, X)  
\end{pmatrix}, \,\begin{pmatrix}
                                                 X 
\\
                                                 Y
\\
                                               \xi
\\
                                               \eta
\end{pmatrix} \in \mathfrak{P}^C,
$$
where $p_1$ is defined by $p_1(X)=(X, E_1)+4E_1 \times (E_1 \times X), X \in \mathfrak{J}^C$.
Then, by straightforward computation, we have that $\delta_{10} \in E_7, \delta_{10} \iota=(-\iota\sigma)\delta_{10}$, that is, $\iota \sim -\sigma\iota$ in $E_7$. 
\end{proof}
We have the following proposition which is the direct result of Lemma 4.18.1.

\begin{prop}
The group $(E_7)^{\iota}$ is isomorphic to the group $(E_7)^{-\sigma\iota}${\rm :} $(E_7)^{\iota} \cong (E_7)^{-\sigma\iota}$.
\end{prop}
From the result of types EVI, EVII in Table 2 and Proposition 4.18.2, we have the following theorem.
 \begin{thm}
For $\mathbb{Z}_2 \times \mathbb{Z}_2=\{1,-\sigma \} \times \{1, \iota \}$, the $\mathbb{Z}_2 \times \mathbb{Z}_2$-symmetric space is of type $(E_7/(E_7)^{-\sigma}, E_7/(E_7)^{\iota}, E_7/(E_7)^{(-\sigma)(\iota)})=(E_7/(E_7)^{\gamma}, E_7/(E_7)^{\iota}, E_7/(E_7)^{\iota})$, that is, type {\rm (EVI, EVII, EVII)}, abbreviated as {\rm  EVI-VII-VII}.
\end{thm}

Now, we determine the structure of the group $(E_7)^{-\sigma} \cap (E_7)^{\iota}$.

\begin{thm}
We have that $(E_7)^{-\sigma} \cap (E_7)^{\iota} \cong (U(1) \times U(1) \times S\!pin(10))/(\Z_{4} \times \Z_{3}), \Z_{4}= \{(1,1,1), \,(1,-1,-\sigma), \,(1,-i,\sigma\phi_1(i)), \,(1, -i, \phi_1(i))  \}, \,\Z_3 =\{(1,1,1),\, (\omega, \omega, 1),\, (\omega^2, \omega^2,1)  \}$,  where $\omega \in C, \omega^3  =1, \omega \ne 1$.
\end{thm}
\begin{proof}
Let $U(1) =\{a \in C \,|\, (\tau a) a =1  \}$ and $ S\!pin(10) \cong (E_6)_{E_1}=\{\alpha \in E_6\,|\,\alpha E_1 =E_1  \}$. 
We define a mapping $\varphi_{4184}: U(1) \times U(1) \times S\!pin(10) \to (E_7)^{-\sigma} \cap (E_7)^{\iota}=(E_7)^{\sigma} \cap (E_7)^{\iota}$ by 
$$
\varphi_{4184}(\theta, a , \delta)=\phi(\theta) \phi_1(a)\delta,
$$
where $\phi, \phi_1$ are  defined in Theorems 3.4.3, 3.3.3, respectively. 
It is clear that $\varphi_{4184}(\theta, a, \delta) \in (E_7)^\iota$, moreover since $\sigma \phi(\theta)\sigma=\phi(\theta)$ and $\phi_1(a)\delta \in (E_6)^\sigma$, 
it is easily to verify that $\varphi_{4184}(\theta, a, \delta) \in (E_7)^\sigma = (E_7)^{-\sigma}$.
Hence $\varphi_{4184}$ is well-defined. Since $\phi(\theta)$ commutes with $\phi_1 (a)$ and $\delta$ each other and moreover $\phi_1 (a)$ commutes with $\delta$, we easily see that $\varphi_{4184}$ a homomorphism. 

We shall show that $\varphi_{4184}$ is surjection. Let $\alpha \in (E_7)^{-\sigma} \cap (E_7)^{\iota}$. Since  $(E_7)^{-\sigma} \cap (E_7)^{\iota} \subset (E_7)^{\iota}$, there exist $\theta \in U(1)$ and $\beta \in E_6$ such that $\alpha=\varphi_{{}_\text{E7}}(\theta, \beta)$ (Theorem 3.4.3). Moreover, from $\alpha \in (E_7)^{-\sigma}=(E_7)^{\sigma}$, that is, $ \sigma\varphi_{{}_\text{E7}}(\theta, \beta)\sigma=\varphi_{{}_\text{E7}}(\theta, \beta)$, we have  $\varphi_{{}_\text{E7}}(\theta, \sigma \beta \sigma)=\varphi_{{}_\text{E7}}(\theta, \beta)$. Hence, it follows that 
$$
\text{(i)}\,\left \{
         \begin{array}{l}
                {\theta}= \theta
                         \vspace{3mm}\\
                \sigma\beta\sigma = \beta,
         \end{array}\right.\quad 
\text{(ii)}\,\left \{         
          \begin{array}{l}
         {\theta}= \omega\theta
                         \vspace{3mm}\\
                \sigma\beta\sigma = \phi(\omega^2)\beta,
         \end{array}\right.\quad 
 \text{(iii)}\,\left \{         
          \begin{array}{l}
        {\theta}= \omega^2 \theta
                         \vspace{3mm}\\
                \sigma\beta\sigma = \phi(\omega)\beta.
         \end{array}\right. 
$$
Then we can easily  confirm that (ii) and (iii) are impossible because of $\theta=0$. In the case (i), we have that $\beta \in (E_6)^\sigma \cong (U(1) \times S\!pin(10))/\Z_4$. Hence, from Theorem 3.3.3, there exist $a \in U(1)$ and $\delta \in S\!pin(10)$ such that $\beta=\varphi_{{}_\text{E3}}(a, \delta )=\phi_1 (a) \delta$. Thus $\varphi_{4184}$ is surjection. 

Finally, we shall determine $\Ker \, \varphi_{4184}$. From $\Ker \,\varphi_{{}_\text{E3}}=\{(1, \phi(1)), (-1, \phi(-1)), (i, \phi(-i)), $ $(-i,\phi(i)) \}$, we can easily obtain that
\begin{eqnarray*}
\Ker\,\varphi_{4184}\!\!\!&=&\!\!\!\{(\theta, a, \delta) \in  U(1) \times U(1) \times S\!pin(10)\,|\,\varphi_{4184} (\theta, a, \delta)=1 \}
\\
\!\!\!&=&\!\!\! \{(\theta, a, \delta) \in  U(1) \times U(1) \times S\!pin(10)\,|\, \phi(\theta) \phi_1(a)\delta=1\}
\\
\!\!\!&=&\!\!\!\{(\theta, a, \delta) \in  U(1) \times U(1) \times S\!pin(10)\,|\, \theta^3=1, a^4=1, \delta=\phi_1(a^{-1})\}
\\
\!\!\!&=&\!\!\!\{(1,1,1), (1,-1,-\sigma), (1,-i,\sigma\phi_1(i)), (1, -i, \phi_1(i)),
\\
\!\!\!&& \!\!\! \quad (\omega, \omega^\frac{1}{4} , \phi_1(i)), (\omega, \omega^{-\frac{1}{4}} , \sigma\phi_1(i)),(\omega, i\omega^\frac{1}{4} , 1) , (\omega, -i\omega^\frac{1}{4} , \sigma), 
\\
\!\!\!&& \!\!\! \quad (\omega^2, \omega^\frac{1}{2} , \sigma), (\omega^2, -\omega^{-\frac{1}{2}} , 1),(\omega^2, i\omega^\frac{1}{2} , \phi_1(i)) , (\omega^2, -i\omega^\frac{1}{4} , \sigma\phi_1(i))\}
\\
\!\!\!&=&\!\!\!\{ (1,1,1), (1,-1,-\sigma), (1,-i,\sigma\phi_1(i)), (1, -i, \phi_1(i)) \}
\\
\!\!\!&& \!\!\!\qquad\qquad\qquad\qquad \times \{(1,1,1), (\omega, \omega, 1), (\omega^2, \omega^2,1)  \} \cong \Z_4 \times \Z_3.
\end{eqnarray*}

Therefore we have the required isomorphism 
$$
(E_7)^{-\sigma} \cap (E_7)^{\iota} \cong (U(1) \times U(1) \times S\!pin(10))/(\Z_{4} \times \Z_{3}).
$$
\end{proof}

\subsection{Type EVII-VII-VII}

In this section, we give a pair of involutive inner automorphisms $\tilde{\iota}$ and $\tilde{\lambda}$.
\vspace{1mm}



\begin{prop}
The group $(E_7)^{\iota}$ is isomorphic to both of the groups $(E_7)^\lambda$ and $(E_7)^{\iota\lambda}{\rm:}\\(E_7)^\iota \cong (E_7)^\lambda \cong (E_7)^{\iota\lambda}$.
\end{prop}
\begin{proof}
First, we have $(E_7)^\iota \cong (E_7)^\lambda$ as the direct result of Lemma 4.13.1.

Next, we define a mapping $g: (E_7)^{\iota\lambda} \to (E_7)^\iota$ by
 $$
      g(\alpha)=({\delta_\lambda}^{-1}{\delta_\iota}^{-1})\alpha(\delta_\iota\delta_\lambda),
 $$
where both of $\delta_\lambda$ and $\delta_\iota$ are defined in Section 4.13.
In order to prove this proposition, it is sufficient to show that the mapping $g$ is well-defined. Indeed, it follows from $(\lambda\iota)\delta_\iota=\delta_\iota(\lambda)$
and $\lambda \delta_\lambda=\delta_\lambda \iota$ that
\begin{eqnarray*}
    \iota g(\alpha)\!\!\!&=&\!\!\! \iota(({\delta_\lambda}^{-1}{\delta_\iota}^{-1})\alpha(\delta_\iota\delta_\lambda)))={\delta_\lambda}^{-1}((\lambda{\delta_\iota}^{-1})\alpha(\delta_\iota\delta_\lambda))=({\delta_\lambda}^{-1}{\delta_\iota}^{-1})((-\lambda\iota)\alpha(\delta_\iota\delta_\lambda))
    \\
    \!\!\!&=&\!\!\!({\delta_\lambda}^{-1}{\delta_\iota}^{-1})(\alpha (-\lambda\iota)(\delta_\iota\delta_\lambda))=({\delta_\lambda}^{-1}{\delta_\iota}^{-1})(\alpha\delta_\iota \lambda \delta_\lambda)=(({\delta_\lambda}^{-1}{\delta_\iota}^{-1})\alpha(\delta_\iota \delta_\lambda)) \iota
    \\
     \!\!\!&=&\!\!\!g(\alpha) \iota,
\end{eqnarray*}
that is, $g(\alpha) \in  (E_7)^\iota$.
\end{proof}
\vspace{1mm}

From the result of type  EVII  in Table 2 and Proposition 4.19.1, we have the following theorem
\begin{thm}
For $\mathbb{Z}_2 \times \mathbb{Z}_2=\{1,\iota \} \times \{1, \lambda \}$, the $\mathbb{Z}_2 \times \mathbb{Z}_2$-symmetric space is of type $(E_7/(E_7)^{\iota}, E_7/(E_7)^{\lambda}, E_7/(E_7)^{\iota\lambda})=(E_7/(E_7)^{\iota}, E_7/(E_7)^{\iota}, E_7/(E_7)^{\iota})$, that is, type {\rm (EVII, EVII, EVII)}, abbreviated as {\rm  EVII-VII-VII}.
\end{thm}
\vspace{1mm}

Now, we determine the structure of the group $(E_7)^{\iota} \cap (E_7)^{\lambda}$. 
\begin{thm}
We have that $(E_7)^{\iota} \cap (E_7)^{\lambda} \cong F_4 \times \mathcal{Z}_2, \mathcal{Z}_2 =\{1, -1  \}$.
\end{thm}
\begin{proof}
We define a mapping $\varphi_{4193}: F_4 \times \{1, -1  \}$ by 
\begin{eqnarray*}
  \varphi_{4193}(\alpha, 1)\!\!\!&=&\!\!\!  \varphi_{{}_{\rm E7}}(1, \alpha), \\
  \varphi_{4193}(\alpha, -1)\!\!\!&=&\!\!\! \varphi_{{}_{\rm E7}}(-1, \alpha),
\end{eqnarray*}
where $\varphi_{{}_{\rm E7}}$ is defined in Theorem 3.4.3.
Since the mapping $\varphi_{4193}$ is the restriction of the mapping $\varphi_{{}_{\rm E7}}$, it is clear that $ \varphi_{4193}$ is well-defined and a homomorphism. 

We shall show that $\varphi_{4193}$ is surjection. Let $\alpha \in (E_7)^{\iota} \cap (E_7)^{\lambda}$. From $(E_7)^{\iota} \cap (E_7)^{\lambda} \subset (E_7)^{\iota}$, there exist $\theta \in U(1)$ and $\beta \in E_6$ such that $\alpha = \varphi_{{}_{\rm E7}}(\theta, \beta)=\phi(\theta)\beta$ (Theorem 3.4.3). Moreover, from $\alpha \in (E_7)^{\lambda}$, that is, $\lambda \varphi_{{}_{\rm E7}}(\theta, \beta)\lambda^{-1}=\varphi_{{}_{\rm E7}}(\theta, \beta)$, using $\lambda\phi(\theta)\lambda^{-1}=\phi(\theta^{-1})$ (Lemma 4.15.3), we have that $\varphi_{{}_{\rm E7}}({\theta}^{-1}, \lambda\beta\lambda^{-1})=\varphi_{{}_{\rm E7}}(\theta, \beta)$. Hence, it follows  that
$$
\text{(i)}\,\left \{
         \begin{array}{l}
                {\theta}^{-1}= \theta
                         \vspace{3mm}\\
                \lambda\beta\lambda^{-1} = \beta,
         \end{array}\right.\quad 
\text{(ii)}\,\left \{         
          \begin{array}{l}
         {\theta}^{-1}= \omega\theta
                         \vspace{3mm}\\
                \lambda\beta\lambda^{-1} =\phi(\omega^2) \beta,
         \end{array}\right.\quad 
 \text{(iii)}\,\left \{         
          \begin{array}{l}
        {\theta}^{-1}= \omega^2 \theta
                         \vspace{3mm}\\
                \lambda\beta\lambda^{-1} =\phi(\omega) \beta.
         \end{array}\right. 
$$

Case (i). We see that $\theta=\pm 1$ and $\beta \in F_4 \cong (E_6)^\lambda$. Hence, in the case of $\theta =1$,  there exist $1 \in U(1)$ and $\beta \in F_4$ such that $\alpha=\phi(1)\beta=\beta$, that is, $\alpha=\varphi_\text{E7}(1, \alpha)=\varphi_{4193}(\alpha, 1)$. Similarly, in the case of $\theta=-1$, we have that $\alpha=\varphi_{{}_{\rm E7}}(-1, \alpha)=\varphi_{4193}(\alpha, -1)$.

Case (ii). We see that $\theta=\pm \omega$ and $\beta =\phi(\omega^2) \beta', \beta' \in F_4$.  Hence, in the case of $\theta =\omega$,  there exist $\omega \in U(1)$ and $\beta =\phi(\omega^2)\,\beta'$ such that $\alpha=\phi(\omega)\phi(\omega^2) \,\beta'=\beta'$, that is, $\alpha=\varphi_\text{E7}(1, \alpha)=\varphi_{4193}(\alpha, 1)$. Similarly, in the case of $\theta=-\omega$, we have that $\alpha=\varphi_\text{E7}(-1, \alpha)=\varphi_{4193}(\alpha, -1)$. Thus this case is reduced to\vspace{1mm} Case (i).

Case (iii). We see that $\theta=\pm \omega^2$ and $\beta =\phi(\omega) \beta', \beta' \in F_4$. As in Case (ii), this case is also reduced to \vspace{1mm}Case (i)

Finally, we shall determine the $\Ker\, \varphi_{4193}$, however it is easily obtained that $\Ker\, \varphi_{4193}=( \{ 1\}, 1)$.

Therefore we have the required isomorphism 
$$
(E_7)^{\iota} \cap (E_7)^{\lambda} \cong F_4 \times \mathcal{Z}_2.
$$ 
\end{proof}
\vspace{2mm}

$\bullet$\,\,{\boldmath [$E_8$]}\,\,We study four types in here.
\vspace{1mm}

\subsection{Type EVIII-VIII-VIII}

In this section, we give a pair of involutive inner automorphisms $\tilde{\sigma}$ and $\tilde{\sigma}'$, where $C$-linear transformations $\sigma, \sigma'$ of ${\mathfrak{e}_8}^C$ are defined  below.
\vspace{1mm}

\vspace{1mm}

We define $C$-linear transformations $\sigma, \sigma'$ of ${\mathfrak{e}_8}^C$ by 
\begin{eqnarray*}
\sigma(\varPhi, P, Q, r, s, t)\!\!\!&=&\!\!\!(\sigma\varPhi\sigma, \sigma P, \sigma Q, r, s, t),
\\[1mm]
\sigma' (\varPhi, P, Q, r, s, t)\!\!\!&=&\!\!\!(\sigma' \varPhi\sigma', \sigma' P, \sigma' Q, r, s, t), \,\,(\varPhi, P, Q, r, s, t) \in {\mathfrak{e}_8}^C, 
\end{eqnarray*}
where $\sigma, \sigma'$ of the right hand side are same ones as $\sigma, \sigma' \in F_4 \subset E_6 \subset  E_7$.
Then we have that that $\sigma, \sigma' \in E_8, \sigma^2={\sigma'}^2=1$. Hence $\sigma, \sigma'$ induce involutive inner automorphisms $\tilde{\sigma}, {\tilde{\sigma}}'$ of $E_8$: $\tilde{\sigma}(\alpha)=\sigma\alpha\sigma, {\tilde{\sigma}}'(\alpha)=\sigma' \alpha\sigma', \alpha \in E_8$.


\begin{lem}
{\rm (1)}\,The Lie algebra $(\mathfrak{e}_8)^\sigma$ of the group $(E_8)^\sigma$  is given by 
$$
(\mathfrak{e}_8)^\sigma =\{(\varPhi, \tau\lambda Q, Q, r, s, -\tau s) \in \mathfrak{e}_8 \,|\,   \varPhi \in (\mathfrak{e}_7)^\sigma  \!\! \cong \mathfrak{su}(2) \oplus \mathfrak{so}(12), Q \in (\mathfrak{P}^C)_{\sigma}, r \in i\R, s \in C \}.
$$
{\rm (2)}\, The Lie algebra $(\mathfrak{e}_8)^{\lambda_\omega \gamma}$ of the group $(E_8)^{\lambda_\omega \gamma}$  is given by 
$$
(\mathfrak{e}_8)^{\lambda_\omega \gamma}=\{(\varPhi, \lambda\gamma Q, Q, 0, s, - s) \in \mathfrak{e}_8 \,|\, \varPhi \in (\mathfrak{e}_7)^{\lambda\gamma} \!= (\mathfrak{e}_7)^{\tau\gamma} \cong \mathfrak{su}(8), Q \in (\mathfrak{P}^C)_{\tau\gamma}, s \in \R \}.
$$

In particular, we have that
\begin{eqnarray*}
\dim((\mathfrak{e}_8)^\sigma)\!\!\!&=&\!\!\! (3+66)+((3+8)\times 2+2 )\times 2+1+1 \times 2=120
\\
\!\!\!&=&\!\!\! 63+(3+(4 \times 3) \times 2)+2 +1=\dim ((\mathfrak{e}_8)^{\lambda_\omega \gamma}).
\end{eqnarray*}
\end{lem}
\begin{proof}
By straightforward computation, we can easily prove this lemma.
\end{proof}
\vspace{1mm}

From Lemma 4.20.1 and \cite[Lemma 5.3.3]{Yokotaichiro3}, we have the following proposition.

\begin{prop}
The group $(E_8)^\sigma$ is isomorphic to the group $(E_8)^{\lambda_\omega \gamma}${\rm :} $(E_8)^\sigma \cong (E_8)^{\lambda_\omega \gamma}$ $(\cong S\!s(16))$.
\end{prop}
\noindent {\bf Remark.}\,The author can not find any element $\delta \in E_8$ which gives the conjugation: $\delta \sigma=({\lambda_\omega \gamma})\delta$.
\vspace{1mm}

Here, using the inclusion $F_4 \subset E_6 \subset E_7 \subset E_8$, the $C$-linear transformations $\delta_6, \delta_7$ defined in the proof of Lemma 4.4.1 are naturally extended to the $C$-linear transformations of ${\mathfrak{e}_8}^C$. Hence,  
as in $E_6$, since we easily see that $\delta_6 \sigma=\sigma' \delta_6, \delta_7 \sigma=(\sigma\sigma')\delta_7$ as $\delta_6, \delta_7 \in F_4 \subset  E_6 \subset E_7 \subset E_8$, that is,  
$\sigma \sim \sigma', \sigma \sim \sigma\sigma'$ in $E_8$, we have the following proposition.

\begin{prop}
The group $(E_8)^\sigma$ is isomorphic to both of the groups  $(E_8)^{\sigma'}$ \!\!and $(E_8)^{\sigma\sigma'}${\rm :} $(E_8)^\sigma \cong (E_8)^{\sigma'} \cong (E_8)^{\sigma\sigma'}$. 
\end{prop}
From the result of type EVIII in Table 2 and Propositions 4.20.2, 4.20.3, we have the following theorem.

\begin{thm}
For $\mathbb{Z}_2 \times \mathbb{Z}_2=\{1,\sigma \} \times \{1, \sigma' \}$, the $\mathbb{Z}_2 \times \mathbb{Z}_2$-symmetric space is of type $(E_8/(E_8)^{\sigma}, E_8/(E_8)^{\sigma'}, E_8/(E_8)^{\sigma\sigma'})=(E_8/(E_8)^{\sigma}, E_8/(E_8)^{\sigma}, E_8/(E_8)^{\sigma})$, that is, type {\rm (EVIII, EVIII, EVIII)}, abbreviated as {\rm  EVIII-VIII-VIII}.
\end{thm}
\vspace{2mm}
 
 Now, we determine the structure of the group $(E_8)^\sigma \cap (E_8)^{\sigma'}$.
 
\begin{thm}
We have that $(E_8)^\sigma \cap (E_8)^{\sigma'} \cong (S\!pin(8) \times S\!pin(8))/(\Z_2 \times \Z_2), \Z_2 =\{(1, 1), (\sigma, \sigma)  \}, \Z_2=\{(1, 1), (\sigma', \sigma')  \}$.
\end{thm}
\begin{proof}
This proof is omitted (see \cite[Theorem 7.1]{Miyashitatoshikazu}. The purpose of \cite{Miyashitatoshikazu} is to prove the this theorem ).
\end{proof}

\subsection{Type EVIII-VIII-IX}

In this section, we use a pair of involutive inner automorphisms $\tilde{\lambda_\omega \gamma}$ and $\tilde{\lambda_\omega \gamma\upsilon}$.
\vspace{1mm}

\begin{lem}
In $E_8$, $\lambda_\omega \gamma\upsilon$ is conjugate to $\lambda_\omega \gamma${\rm :} $\lambda_\omega \gamma\upsilon \sim \lambda_\omega \gamma$.
\end{lem} 
\begin{proof}
We define a $C$-linear transformation $\delta_{\upsilon}$ of ${\mathfrak{e}_8}^C$ by\vspace{-1mm}
$$
\delta_{\upsilon}(\varPhi, P, Q, r, s, t)=(\varPhi, iP, -iQ, r, -s,- t), (\varPhi, P, Q, r, s, t) \in {\mathfrak{e}_8}^C.
$$
Then we have that $\delta_{\upsilon} \in E_8, {\delta_{\upsilon}}^2=\upsilon$ and $\delta_{\upsilon} (\lambda_\omega \gamma\upsilon)=(\lambda_\omega \gamma)\delta_{\upsilon}$: $\lambda_\omega \gamma\upsilon \sim \lambda_\omega \gamma$ in $E_8$, moreover that $\delta_\upsilon\lambda=\lambda \delta_\upsilon$.
\end{proof}
\vspace{1mm}

We have the following proposition which is the direct result of Lemma 4.21.1.

\begin{prop}
The group $(E_8)^{\lambda_\omega \gamma\upsilon}$  is isomorphic to the group $(E_8)^{\lambda_\omega \gamma}${\rm :}  
 $(E_8)^{\lambda_\omega \gamma\upsilon} \cong (E_8)^{\lambda_\omega \gamma}$.
\end{prop}
From the results of types EXIII, EIX in Table 2 and Proposition 4.21.1, we have the following theorem.
\begin{thm}
For $\mathbb{Z}_2 \times \mathbb{Z}_2=\{1, \lambda_\omega \gamma \} \times \{1, \lambda_\omega \gamma\upsilon \}$, the $\mathbb{Z}_2 \times \mathbb{Z}_2$-symmetric space is of type $(E_8/(E_8)^{\lambda_\omega \gamma}, E_8/(E_8)^{\lambda_\omega \gamma\upsilon}, E_8/(E_8)^{(\lambda_\omega \gamma)(\lambda_\omega \gamma\upsilon)})=(E_8/(E_8)^{\lambda_\omega \gamma}, E_8/(E_8)^{\lambda_\omega \gamma}, E_8/(E_8)^{\upsilon})$, that is, type {\rm (EVIII, EVIII, EIX)}, abbreviated as \vspace{-1mm}{\rm  EVIII-VIII-IX}.
\end{thm}
Here, we prove lemma needed in theorem below.
\begin{lem}
The mapping $\phi_3 : S\!U(2) \to E_8$ of Theorem 3.5.2 satisfies 
\begin{eqnarray*}
&&{\hspace*{-5mm}}{\rm (1)}\,\delta_\upsilon =\phi_3 (iI).
\\
&&{\hspace*{-5mm}}{\rm (2)}\,\lambda_\omega \phi_3 (A) {\lambda_\omega }^{-1}=\iota_\omega \phi_3 (A) {\iota_\omega }^{-1}=\phi_3 ({}^t A^{-1}). 
\\
&&{\hspace*{-5mm}}{\rm (3)}\,\sigma \phi_3 (A) \sigma=\phi_3 (A), \,\,
\gamma \phi_3 (A) \gamma=\phi_3 (A),
\end{eqnarray*}
where $iI=\diag(i, -i) \in S\!U(2)$.
\end{lem}
\begin{proof}
By straightforward computation, we can easily prove this lemma. (The $C$-linear transformation $\iota_\omega$ of $(\mathfrak{e}_8)^C$ is defined in Section 4.23. As for the definition of the mapping $\phi_3$, see \cite[Theorem 5.7.4]{Yokotaichiro0}.)
\end{proof}

Consider a group $\mathcal{Z}_2=\{1, \rho_\upsilon \}$, where $\rho_\upsilon=\delta_\upsilon \iota$.
Then the group $\mathcal{Z}_2$ acts on the group $S\!O(2) \times S\!U(8)$ by
$$
  \rho_\upsilon(A, B)=((i I) A\,(i I)^{-1}, J\ov{B}J^{-1}),
$$
where $J=\diag(J_1, J_1, J_1, J_1) \in M(8, \R), J_1 =\begin{pmatrix} 0 & 1 \\
                                            -1 & 0 
                     \end{pmatrix}$, 
\vspace{1mm}and let $(S\!O(2) $ $ \times \,S\!U(8)) \rtimes \mathcal{Z}_2$ be the semi-direct product $S\!O(2) \times S\!U(8)$  and $\mathcal{Z}_2$ with this action.
\vspace{1mm}

Now, we determine the structure of the group $(E_8)^{\lambda_\omega \gamma} \cap (E_8)^{\lambda_\omega \gamma\upsilon}$.
\begin{thm}
We have that $(E_8)^{\lambda_\omega \gamma} \cap (E_8)^{\lambda_\omega \gamma\upsilon} \cong (S\!O(2) \times S\!U(8))/\Z_4 \rtimes \mathcal{Z}_2, \Z_4=\{(E,E), (E, -E), (-E, e_1 E), (-E, -e_1 E)  \}, \mathcal{Z}_2=\{1, \rho_\upsilon \}$.
\end{thm}
\begin{proof}
We define a mapping $\varphi_{4215}: (S\!O(2) \times S\!U(8)) \rtimes \{1, \rho_\upsilon \} \to (E_8)^{\lambda_\omega \gamma} \cap (E_8)^{\lambda_\omega \gamma\upsilon}$ 
by 
\begin{eqnarray*}
&& \varphi_{4215}((A, B), 1)=\varphi_{{}_\text{E9}}(A, \varphi_{{}_\text{E5}}(B)),
\\
&& \varphi_{4215}((A, B), \rho)=\varphi_{{}_\text{E9}}(A, \varphi_{{}_\text{E5}}(B))\,\rho_\upsilon, 
\end{eqnarray*}
where $\varphi_{{}_\text{E9}}, \varphi_{{}_\text{E5}}$ are  defined in Theorems 3.5.2, 3.4.1, respectively. 
From $\lambda_\omega \gamma \phi_3(A)\gamma {\lambda_\omega}^{-1}=\phi(A), A \in S\!O(2)$ (Lemma 4.21.4 (3)) and $\varphi_{{}_\text{E5}}(B) \in (E_7)^{\lambda\gamma}$ (Theorem 3.4.1), it is easily to verify that $ \varphi_{4215}$ is well-defined. By straightforward computation, we can confirm that $ \varphi_{4215}$ is a homomorphism. Indeed, we show that the case of $\varphi_{4215}((A, B),\rho_\upsilon) \varphi_{4215}((C, D), 1) $ $=\varphi_{4215}((A, B), \rho_\upsilon)((C, D), 1))$ as example. For the left hand side of this equality, we have that \vspace{-3mm}
\begin{eqnarray*}
\varphi_{4215}((A, B), \rho_\upsilon) \varphi_{4215}((C, D), 1)
\!\!\!&=&\!\!\! \varphi_{{}_\text{E9}}(A, \varphi_{{}_\text{E5}}(B))\,\rho_\upsilon\,\varphi_{{}_\text{E9}}(C, \varphi_{{}_\text{E5}}(D))
\\
\!\!\!&=&\!\!\! \phi_3(A)\varphi_{{}_\text{E5}}(B)\,\rho_\upsilon\,\phi_3(C)\varphi_{{}_\text{E5}}(D). 
\end{eqnarray*}
On the other hand, for the right hand side of same one, using $\delta_\upsilon=\phi_3(iI), \iota\varphi_{{}_\text{E5}}(A) \iota^{-1}=\varphi_{{}_\text{E5}}(J\ov{A}J)$ (Lemmas 4.21.4 (1), 4.13.6 (2)), we have \vspace{-2mm}that
\begin{eqnarray*}
\varphi_{4215}((A, B), \rho_\upsilon)((C, D), 1))\!\!\!&=&\!\!\!\varphi_{4215}(((A, B) \rho(C, D))), \rho_\upsilon)
\\
\!\!\!&=&\!\!\!\varphi_{4215}(((A, B) ((iI)C(iI)^{-1},J\ov{D}J^{-1})), \rho_\upsilon)
\\
\!\!\!&=&\!\!\!\varphi_{4215}((A \,(iI)C(iI)^{-1}, B \,J\ov{D}J^{-1}), \rho_\upsilon)
\\
\!\!\!&=&\!\!\! \phi_3 (A \,(iI)C(iI)^{-1})\,{\varphi_{{}_\text{E5}}}(B \, J \ov{D}J^{-1}))\rho_\upsilon
\\
\!\!\!&=&\!\!\!\phi_3 (A \,(iI)C(iI)^{-1})\,{\varphi_{{}_\text{E5}}}(B \, J \ov{D}J^{-1}))(\delta_\upsilon \iota)(\gets\,J^{-1}=-J)
\\
\!\!\!&=&\!\!\!\phi_3 (A) \delta_\upsilon \phi_3 (C) {\delta_\upsilon}^{-1}\,{\varphi_{{}_\text{E5}}}(B) \iota {\varphi_{{}_\text{E5}}}(D) \iota^{-1}(\delta_\upsilon \iota)
\\
\!\!\!&=&\!\!\! \phi_3 (A) \delta_\upsilon \phi_3 (C) {\delta_\upsilon}^{-1}\,{\varphi_{{}_\text{E5}}}(B) \iota {\varphi_{{}_\text{E5}}}(D) \delta_\upsilon 
\\
\!\!\!&=&\!\!\! \phi_3 (A) \delta_\upsilon \phi_3 (C) \,{\varphi_{{}_\text{E5}}}(B) \iota {\varphi_{{}_\text{E5}}}(D) 
\\
\!\!\!&=&\!\!\!\phi_3 (A) \delta_\upsilon  \,{\varphi_{{}_\text{E5}}}(B) \phi_3 (C)\iota {\varphi_{{}_\text{E5}}}(D)(\gets \, \varphi_{{}_\text{E5}}(B), \iota \in E_7 )
\\
\!\!\!&=&\!\!\!\phi_3(A){\varphi_{{}_\text{E5}}(B)}(\delta_\upsilon \iota )\,\phi_3(C){\varphi_{{}_\text{E5}}(D)}
\\
\!\!\!&=&\!\!\! \phi_3(A)\varphi_{{}_\text{E5}}(B)\,\rho_\upsilon\,\phi_3(C)\varphi_{{}_\text{E5}}(D).
\end{eqnarray*}
Similarly, the other cases are shown.

We shall show that $\varphi_{4215}$ is surjection. Let $\alpha \in (E_8)^{\lambda_\omega \gamma} \cap (E_8)^{\lambda_\omega \gamma\upsilon}$. From $(E_8)^{\lambda_\omega \gamma} \cap (E_8)^{\lambda_\omega \gamma\upsilon} \subset (E_8)^{(\lambda_\omega \gamma)(\lambda_\omega \gamma\upsilon)}=(E_8)^\upsilon$, there exist $A \in S\!U(2)$ and $\beta \in E_7$ such that $\alpha=\varphi_{{}_\text{E9}}(A, \beta)$ (Theorem 3.5.2). Moreover, from $\alpha\!=\!\varphi_{{}_\text{E9}}(A, \beta) \in (E_8)^{\lambda_\omega \gamma}$, that is, $\lambda_\omega \gamma \varphi_{{}_\text{E9}}(A, \beta)\gamma {\lambda_\omega }^{-1}=\varphi_{{}_\text{E9}}(A, \beta)$,  using $\lambda_\omega \gamma \phi_3(A)\gamma {\lambda_\omega }^{-1}\!=\!\phi_3({{}^t A^{-1}}\!)$ (Lemma 4.21.4 (2)), we have that $\varphi_{{}_\text{E9}}({}^t A^{-1}\!,$ $ \lambda\gamma\beta\gamma{\lambda}^{-1})=\varphi_{{}_\text{E9}}(A, \beta)$. (Remark. For $\alpha \in (E_8)^\upsilon$, that  $\alpha \in (E_8)^{\lambda_\omega \gamma}$ implies that $\alpha \in (E_8)^{\lambda_\omega \gamma\upsilon}$.) Hence, it follows that
$$
\left \{
         \begin{array}{l}
                {}^t A^{-1}= A
                         \vspace{3mm}\\
                \lambda\gamma\beta\gamma{\lambda}^{-1} = \beta
         \end{array}\right.\qquad
 \text{or}\qquad 
\left \{
         \begin{array}{l}
                {}^t A^{-1}=-A
                         \vspace{3mm}\\
                \lambda\gamma\beta\gamma{\lambda}^{-1} = -\beta.
         \end{array}\right. 
$$
\noindent In the former case, we see that $A \in S\!O(2)$ and $\beta \in (E_7)^{\lambda\gamma} \cong S\!U(8)/\Z_2$. Hence, there exists $B \in S\!U(8)$ such that $\beta=\varphi_{{}_\text{E5}}(B)$ (Theorem 3.4.1). Thus we have that $\alpha=\varphi_{{}_\text{E9}}(A, \beta)=\varphi_{{}_\text{E9}}(A, \varphi_{{}_\text{E5}}(B))=\varphi_{4215}((A, B), 1)$.
In the latter case, we see that $A=A'(iI), A' \in S\!O(2)$ and $\beta=\beta' \iota, \beta' \in (E_7)^{\lambda\gamma}$. Hence, in a similar way as the former case,  we have that 
\begin{eqnarray*}
\alpha \!\!\!&=&\!\!\! \varphi_{{}_\text{E9}}(A, \beta)=\varphi_{{}_\text{E9}}(A'(iI), \beta' \iota)=\phi_3(A'(iI))(\beta'\iota)=\phi_3(A')\phi_3(iI)(\beta'\iota)
\\
\!\!\!&=&\!\!\! \phi_3(A')\phi_3(iI)\beta'\iota=\phi_3(A')\beta'(\phi_3(iI)\iota)=
\varphi_{{}_\text{E9}}(A', \beta')\,(\delta_\upsilon \iota)=\varphi_{{}_\text{E9}}(A',\varphi_{{}_\text{E5}}(B') )\,\rho_\upsilon
\\
\!\!\!&=&\!\!\!\varphi_{4215}((A', B'), \rho_\upsilon).
\end{eqnarray*}
Thus  $\varphi_{4214}$ is surjection.

Finally, we shall determine $\Ker \,\varphi_{4215}$. From the definition of kernel, it is as follows:
\begin{eqnarray*}
\Ker\,\varphi_{4215}\!\!\!&=&\!\!\!\{((A, B),1)\,|\,\varphi_{4215}((A, B), 1)=1 \} \cup \{((A, B),\rho_\upsilon)\,|\,\varphi_{4215}((A, B),\rho_\upsilon) =1\}
\\
\!\!\!&=&\!\!\!\{((A, B),1)\,|\,\varphi_{{}_\text{E9}}(A, \varphi_{{}_\text{E5}}(B))=1\} \cup \{((A, B),\rho)\,|\,\varphi_{{}_\text{E9}}(A, \varphi_{{}_\text{E5}}(B))\rho_\upsilon=1\}.
\end{eqnarray*}
Here, for the left hand side case , we have that 
\begin{eqnarray*}
&& \!\!\!\{((A, B),1)\,|\,\varphi_{{}_\text{E9}}(A, \varphi_{{}_\text{E5}}(B))=1\}
\\
\!\!\!&=&\!\!\!\{((A, B),1)\,|\,A=\pm E, \varphi_{{}_\text{E5}}(B))=\pm1\}
\\
\!\!\!&=&\!\!\!\{((E, E),1),( (E, -E),1), ((-E, -e_1 E),1), ((-E, e_1 E),1)  \}.
\end{eqnarray*}
On the other hand, for the right hand side case, since $\varphi_{{}_\text{E9}}(A, \varphi_{{}_\text{E5}}(B))\rho_\upsilon=1$,  we suppose that 
$$
\varphi_{{}_\text{E9}}(A, \varphi_{{}_\text{E5}}(B))\rho_\upsilon  (0,0,0,0,1,0)=(0,0,0,0,1,0), \,\,{\rm {where}} \,\,(0,0,0,0,1,0) \in {\mathfrak{e}_8}^C.
$$ 
Then since we have that $\phi_3(A)(0,0,0,0,i,0)=(0,0,0,0,1,0)$, there exist no $A \in S\!O(2)$ such that $\varphi_{{}_\text{E9}}(A, \varphi_{{}_\text{E5}}(B))\rho_\upsilon=1$.
Hence, the right hand case is impossible. Thus we have that 
$$
\Ker\, \varphi_{4215}=\{((E, E),1),( (E, -E),1), ((-E, -e_1 E),1), ((-E, e_1 E),1)  \}\cong (\Z_4, 1).
$$

Therefore we have the required isomorphism 
$$
(E_8)^{\lambda_\omega \gamma} \cap (E_8)^{\lambda_\omega \gamma\upsilon} \cong (S\!O(2) \times S\!U(8))/\Z_4 \rtimes \mathcal{Z}_2 .$$
\end{proof} 

\subsection{Type EVIII-IX-IX}
 
 In this section, we use a pair of involutive inner automorphisms $\tilde{\sigma}$ and $\tilde{\upsilon}$.
\vspace{1mm}

 
\begin{lem}
{\rm (1)} The Lie algebra $(\mathfrak{e}_8)^\upsilon$ of the group $(E_8)^\upsilon$ is given by
$$
(\mathfrak{e}_8)^\upsilon=\{(\varPhi, 0, 0, r, s, -\tau s) \,|\varPhi \in \mathfrak{e}_7, r \in i\R, s \in C \}.
$$
{\rm (2)} The Lie algebra $(\mathfrak{e}_8)^{\upsilon\sigma}$ of the group $(E_8)^{\upsilon\sigma}$ is given by
$$
(\mathfrak{e}_8)^{\upsilon\sigma}=\{(\varPhi, \tau\lambda Q, Q, r, s, -\tau s) \,|\varPhi \in (\mathfrak{e}_7)^\sigma \cong \mathfrak{su}(2) \oplus \mathfrak{so}(12), Q \in (\mathfrak{P}^C)_{-\sigma} , r \in i\R, s \in C \},
$$ 
where $(\mathfrak{P}^C)_{-\sigma}=\{P \in  \mathfrak{P}^C \,|\, \sigma P=-P \} \vspace{1mm}$.

In particular, we have that
$$
\dim((\mathfrak{e}_8)^\upsilon))=133+1+2=136=(3+66)+(8+8) \times 2 \times 2+1+2=\dim((\mathfrak{e}_8)^{\upsilon\sigma}).
$$
\end{lem} 
\begin{proof}
By straightforward computation, we can easily prove this lemma.
\end{proof}
\vspace{2mm}

From Lemma 4.22.1 and \cite[Lemma 5.3.3]{Yokotaichiro3}, we have the following proposition.

\begin{prop}
The group $(E_8)^{\upsilon}$ is isomorphic to the group $(E_8)^{\upsilon\sigma}${\rm :} $(E_8)^{\upsilon} \cong (E_8)^{\upsilon\sigma}$.
\end{prop}
\noindent {\bf Remark.}\,The author can not find any element $\delta \in E_8$ which gives the conjugation: $\upsilon\delta=\delta(\upsilon\sigma)$.
\vspace{2mm}

From the results of types EVII, EIX in Table 2 and Propositions 4.20.2, 4.22.2. we have the following theorem.

\begin{thm}
For $\mathbb{Z}_2 \times \mathbb{Z}_2=\{1,\sigma \} \times \{1, \upsilon \}$, the $\mathbb{Z}_2 \times \mathbb{Z}_2$-symmetric space is of type $(E_8/(E_8)^{\sigma}, E_8/(E_8)^{\upsilon}, E_8/(E_8)^{\upsilon\sigma})
=(E_8/(E_8)^{\lambda_{\omega}}, E_8/(E_8)^{\upsilon}, E_8/(E_8)^{\upsilon})$, that is, type {\rm (EVIII, EIX, EIX)}, abbreviated as {\rm  EVIII-IX-IX}.
\end{thm}
\vspace{1mm}

Now, we determine the structure of the group $(E_8)^{\sigma} \cap (E_8)^{\upsilon}$.

\begin{thm}
We have that $(E_8)^{\sigma} \cap (E_8)^{\upsilon} \cong (S\!U(2) \times S\!U(2) \times S\!pin(12))/(\Z_2 \times \Z_2), \Z_2=\{(E,E,1), (-E,E,-1)  \}, \Z_2=\{(E,E,1), (E-,E,-\sigma)  \}$.
\end{thm}
\begin{proof}
We define a mapping $\varphi_{4224}: S\!U(2) \times S\!U(2) \times S\!pin(12) \to (E_8)^{\sigma} \cap (E_8)^{\upsilon}$ by
$$
\varphi_{4224}(A, B, \beta)=\phi_3(A) \phi_2(B) \beta,
$$
where $\phi_3, \phi_2$ are defined in Theorems 3.5.2, 3.4.2, respectively.
From $\sigma \phi_3(A)\sigma=\phi_3(A)$ (Lemma 4.21.4 (3)) and $\phi_2 (B)\beta \in (E_7)^\sigma$ (Theorem 3.4.2), it is easily to verify that $ \varphi_{4224}$ is well-defined.
Since $\phi_3(A)$ commutes with $\phi_2(B)$ and $\beta$ each other (see [8, Theorem 5.7.6] in detail), moreover $\phi_2(B)$ commutes with $\beta$ in $E_7 \subset E_8$ (see \cite[Theorem 4.11.15]{Yokotaichiro0} in detail), we see that $\varphi_{4224}$ is a homomorphism. 

We shall show that $\varphi_{4224}$ is surjection. Let $\alpha \in (E_8)^{\sigma} \cap (E_8)^{\upsilon}$. From $(E_8)^{\sigma} \cap (E_8)^{\upsilon} \subset (E_8)^{\upsilon}$, there exist $A \in S\!U(2)$ and $\delta \in E_7$ such that $\alpha=\varphi_{{}_\text{E9}}(A, \delta)$ (Theorem 3.5.2). Moreover, from $\alpha=\varphi_{{}_\text{E9}}(A, \delta) \in  (E_8)^{\sigma}$, that is, $\sigma\varphi_{{}_\text{E9}}(A, \delta)\sigma=\varphi_{{}_\text{E9}}(A, \delta)$, again using $ \sigma \phi_3(A)\sigma=\phi_3(A)$, we have that $\varphi_{{}_\text{E9}}(A, \sigma\delta\sigma)= \varphi_{{}_\text{E9}}(A, \delta) $. 
Hence, it follows that
$$
\left \{
         \begin{array}{l}
                 A= A
                         \vspace{3mm}\\
                \sigma\delta\sigma = \delta
         \end{array}\right.\qquad
 \text{or}\qquad \left \{
         \begin{array}{l}
                A=-A
                         \vspace{3mm}\\
                \sigma\delta\sigma = -\delta.
         \end{array}\right. 
$$
In the latter case, this case is impossible because of $A=0$. 
In the former case, we see that $\delta \in (E_7)^\sigma \cong (S\!U(2) \times S\!pin(12))/\Z_2$. Hence, there exist $B \in S\!U(2)$ and $\beta \in S\!pin(12)$ such that $\delta=\varphi_{{}_\text{E6}}(B, \beta)=\phi_2(B)\beta$ (Theorem 3.4.2). Thus $\varphi_{4224}$ is surjection. 

Finally, we shall determine $\Ker \, \varphi_{4224}$. From $\Ker \,\varphi_{{}_\text{E6}}=\{(E,1), (-E, -\sigma) \}$, we have that
\begin{eqnarray*}
   \Ker \, \varphi_{4224}\!\!\!&=&\!\!\!\{ (A, B, \beta) \in S\!U(2) \times S\!U(2) \times S\!pin(12) \,|\,  \varphi_{4224}(A, B, \beta)=1\}
\\
\!\!\!&=&\!\!\! \{ (A, B, \beta) \in S\!U(2) \times S\!U(2) \times S\!pin(12) \,|\,  \phi_3(A) \phi_2(B) \beta=1 \}
\\
\!\!\!&=&\!\!\!\{ (A, B, \beta) \in S\!U(2) \times S\!U(2) \times S\!pin(12) \,|\,A=\pm E, \phi_2(B) \beta=\pm 1 \}
\\
\!\!\!&=&\!\!\! \{(E, E, 1), (E, -E, -\sigma), (-E,E,-1), (-E, -E, \sigma)  \}
\\
\!\!\!&=&\!\!\! \{(E, E, 1),   (-E,E,-1) \} \times \{(E, E, 1), (E, -E, -\sigma)  \} \cong \Z_2 \times \Z_2.
\end{eqnarray*}

Therefore we have the required isomorphism 
$$
(E_8)^{\sigma} \cap (E_8)^{\upsilon} \cong (S\!U(2) \times S\!U(2) \times S\!pin(12))/(\Z_2 \times \Z_2).
$$
\end{proof}


\subsection{Type EIX-IX-IX}

 In this section, we use a pair of involutive inner automorphisms $\tilde{\upsilon}$ and $\tilde{\iota_{\omega}}$.
\vspace{1mm}

We define $C$-linear transformations $\iota_{\omega}, \upsilon\iota_{\omega}$ of ${\mathfrak{e}_8}^C$ by
\begin{eqnarray*}
   \iota_{\omega}(\varPhi, P, Q, r, s, t)\!\!\!&=&\!\!\!(\iota\varPhi\iota^{-1}, \iota Q, -\iota P, -r, -t, -s),
 \\
  \upsilon\iota_{\omega}(\varPhi, P, Q, r, s, t)\!\!\!&=&\!\!\!(\iota\varPhi\iota^{-1}, -\iota Q, \iota P, -r, -t, -s),\,(\varPhi, P, Q, r, s, t) \in {\mathfrak{e}_8}^C,
\end{eqnarray*}
where $\iota$ of the right hand side is same one as $\iota \in E_7$. Then we see that $\iota_{\omega}, \upsilon\iota_{\omega} \in E_8, {\iota_{\omega}}^2={\upsilon\iota_{\omega}}^2=1$. Hence  $\iota_{\omega}, \upsilon \iota_{\omega}$ induce involutive inner automorphisms $\tilde{\iota_{\omega}}, \tilde{\upsilon\iota_{\omega}}$ of $E_8$:
$\tilde{\iota_{\omega}}(\alpha)=\iota_{\omega}\alpha \iota_{\omega}, \tilde{\upsilon\iota_{\omega}}(\alpha)=(\upsilon\iota_{\omega})\alpha(\iota_{\omega}\upsilon), \alpha \in E_8$.
 \begin{lem}
{\rm (1)} The Lie algebra $(\mathfrak{e}_8)^{\iota_{\omega}}$ of the group $(E_8)^{\iota_{\omega}}$ is given by
$$
(\mathfrak{e}_8)^{\iota_{\omega}}\!\!=\!\!\Biggl \{(\varPhi, \tau\lambda Q, Q, 0, s, -s) \,\Biggm|\,\begin{array}{l}
\varPhi \in ({\mathfrak{e}_7})^\iota \cong \mathfrak{u}(1) \oplus \mathfrak{e}_6, Q= (X, i\tau X, \xi, i\tau \xi), \\
X \in \mathfrak{J}^C, \xi \in C, s \in \R
                            \end{array} \Biggr\}.
$$
{\rm (2)} The Lie algebra $(\mathfrak{e}_8)^{\upsilon{\iota_{\omega}}}$ of the group $(E_8)^{\upsilon{\iota_{\omega}}}$ is given by
$$
(\mathfrak{e}_8)^{\upsilon{\iota_{\omega}}}\!\!=\!\!\Biggl\{(\varPhi, \tau\lambda Q, Q, 0, s, -s) \,\Biggm|\,\begin{array}{l}
\varPhi \in ({\mathfrak{e}_7})^\iota \cong \mathfrak{u}(1) \oplus \mathfrak{e}_6,  Q= (X,- i\tau X, \xi, -i\tau \xi), \\
X \in \mathfrak{J}^C, \xi \in C,  s \in \R  
                                 \end{array}\Biggr\}.
$$
In particular,\,
\begin{eqnarray*}
 \dim((\mathfrak{e}_8)^{\iota_{\omega}})\!\!\!&=&\!\!\!(1+78)+(27+1) \times 2 +1\vspace{1mm}=136
\\
\!\!\!&=&\!\!\!(1+78)+(27+1) \times 2 +1=\dim((\mathfrak{e}_8)^{\upsilon{\iota_{\omega}}}).
\end{eqnarray*}
\end{lem}
\begin{proof}
By straightforward computation, we can easily prove this lemma.
\end{proof} 
From Lemmas 4.22.1 (1), 4.23.1 above and \cite[Lemma 5.3.3]{Yokotaichiro3}, we have the following proposition. 

\begin{prop}
The group $(E_8)^\upsilon $ is isomorphic to both of the groups $(E_8)^{\iota_{\omega}}$ and $(E_8)^{\upsilon {\iota_{\omega}}}${\rm :} $(E_8)^\upsilon \cong (E_8)^{\iota_{\omega}} \cong (E_8)^{\upsilon {\iota_{\omega}}}$.
\end{prop}

\noindent {\bf Remark.}\,The author can not find any element $\delta, \delta' \in E_8$ which give the conjugations: $\upsilon\delta=\delta{\iota_{\omega}}, \,{\iota_{\omega}}\delta'=\delta' \upsilon {\iota_{\omega}}$.
\vspace{1mm} 

From the result of type EIX in Table 2 and Proposition 4.23.2, we have the following theorem.

\begin{thm}
For $\mathbb{Z}_2 \times \mathbb{Z}_2=\{1,\upsilon \} \times \{1,  {\iota_{\omega}} \}$, the $\mathbb{Z}_2 \times \mathbb{Z}_2$-symmetric space is of type $(E_8/(E_8)^{\upsilon}, E_8/(E_8)^{\iota_{\omega}}, E_8/(E_8)^{\upsilon{\iota_{\omega}}})
=(E_8/(E_8)^{\upsilon}, E_8/(E_8)^{\upsilon}, E_8/(E_8)^{\upsilon})$, that is, type {\rm (EIX, EIX, EIX)}, \vspace{1mm}abbreviated as {\rm  EIX-IX-IX}.
\end{thm} 

Consider a group $\mathcal{Z}_2 =\{ 1, \nu \}$, where $\nu=\delta_\upsilon \lambda$ ($\delta_\upsilon$ and $\lambda$ are defined in Section 4.21 and Section 3.4, respectively).
Then the group $\mathcal{Z}_2$ acts on the group $S\!O(2) \times U(1) \times E_6$ by
$$
  \nu(A, \theta, \beta)=((iI)A(iI)^{-1}, \theta^{-1}, \tau \beta \tau),
$$
and let $(S\!O(2) \times U(1) \times E_6) \rtimes \mathcal{Z}_2$ be the semi-direct product $S\!O(2) \times U(1) \times E_6$ and $\mathcal{Z}_2$ with this action.

Now, we determine the structure of the group $(E_8)^{\upsilon} \cap (E_8)^{\iota_{\omega}}$.  

\begin{thm}
We have that $(E_8)^{\upsilon} \cap (E_8)^{\iota_{\omega}} \cong (S\!O(2) \times U(1) \times E_6)/(\Z_2 \times \Z_3) \rtimes \mathcal{Z}_2, \Z_2=\{(E,1,1), (-E-1,1)  \}, \Z_3= \{(E,1,1), (E, \omega, \phi_2(\omega^2)),(E, \omega^2, \phi_2(\omega)  \}, \mathcal{Z}_2=\{ 1, \nu \}$.
\end{thm}
\begin{proof}
We define a mapping $\varphi_{4234}: (S\!O(2) \times U(1) \times E_6) \rtimes \mathbb{Z}_2  \to (E_8)^{\upsilon} \cap (E_8)^{\iota_{\omega}}$ by
\begin{eqnarray*}
&& \varphi_{4234}((A, \theta, \beta), 1)=\varphi_{{}_\text{E9}}(A, \varphi_{{}_\text{E7}}(\theta, \beta)),
\\
&& \varphi_{4234}((A, \theta, \beta), \nu)=\varphi_{{}_\text{E9}}(A, \varphi_{{}_\text{E7}}(\theta, \beta))\,\nu, 
\end{eqnarray*}
where $\varphi_{{}_\text{E7}}$ are  defined in Theorem 3.4.3. 
From $\varphi_{{}_\text{E7}}(\theta, \beta) \in (E_7)^\iota$ and $\upsilon\nu=\nu\upsilon$, it is clear that $\varphi_{4234}((A, \theta, \beta), 1), \varphi_{4234}((A, \theta, \beta), \nu) \in (E_8)^{\upsilon}$, moreover from $\varphi_{{}_\text{E9}}(A, \varphi_{{}_\text{E7}}(\theta, \beta))=\phi_3(A)\phi(\theta)\beta$ and $\iota_\omega \nu=\nu\iota_\omega$, it is easily to verify that 
$\varphi_{4234}((A, \theta, \beta), 1), \varphi_{4234}((A, \theta, \beta), \nu) \in (E_8)^{\iota_\omega}$. Hence $\varphi_{4234}$ is well-defined. By straightforward computation, we can confirm that $ \varphi_{4234}$ is a homomorphism. Indeed, we show that the case of $\varphi_{4234}((A, \theta, \beta), \nu) \,\varphi_{4234}((B, \zeta, \kappa ), 1) =\varphi_{4234}(((A, \theta, \beta), \nu)$ $((B, \zeta, \kappa ), 1))$ as example. For the left hand side of this equality, we have that
\begin{eqnarray*}
\varphi_{4234}((A, \theta, \beta), \nu) \,\varphi_{4234}((B, \zeta, \kappa ), 1)
\!\!\!&=&\!\!\! \varphi_{{}_\text{E9}}(A, \varphi_{{}_\text{E7}}(\theta, \beta))\,\nu\,\varphi_{{}_\text{E9}}(B, \varphi_{{}_\text{E7}}(\zeta, \kappa))
\\
\!\!\!&=&\!\!\! \phi_3(A)\varphi_{{}_\text{E7}}(\theta, \beta)\,\nu\,\phi_3(B)\varphi_{{}_\text{E7}}(\zeta, \kappa).
\end{eqnarray*}
On the other hand, for the right hand side of same one, using $\delta_\upsilon=\phi_3(iI), \delta_\upsilon \lambda= \lambda  \delta_\upsilon$ (Lemmas 4.21.4 (1), 4.21.1) and $\tau \kappa \tau=\lambda \kappa \lambda^{-1}$, that is, $(\tau \kappa \tau)\lambda=\lambda \kappa$ as $\kappa \in E_6 \subset E_7$, 
we have that
\begin{eqnarray*}
\varphi_{4234}(((A, \theta, \beta), \nu)((B, \zeta, \kappa ), 1) )\!\!\!&=&\!\!\!\varphi_{4234}(((A, \theta, \beta) (\nu (B, \zeta, \kappa))), \nu)
\\[0.5mm]
\!\!\!&=&\!\!\!\varphi_{4234}(((A,  \theta, \beta) ((iI)B(iI)^{-1},\zeta^{-1}, \tau\kappa\tau)), \nu)
\\
\!\!\!&=&\!\!\!\varphi_{4234}((A \,(iI)B(iI)^{-1}, \theta \zeta^{-1}, \beta \,\tau\kappa\tau), \nu)
\\[1mm]
\!\!\!&=&\!\!\!\varphi_{{}_\text{E9}}((A \,(iI)B(iI)^{-1}, \varphi_{{}_\text{E7}}(\theta \zeta^{-1}, \beta \,\tau\kappa\tau), \nu)
\\[1mm]
\!\!\!&=&\!\!\! \phi_3 (A \,(iI)B(iI)^{-1})\,{\varphi_{{}_\text{E7}}}(\theta \zeta^{-1}, \beta\, \tau\kappa\tau))\nu
\\[1mm]
\!\!\!&=&\!\!\! \phi_3 (A \,(iI)B(iI)^{-1})\,{\varphi_{{}_\text{E7}}}(\theta \zeta^{-1}, \beta \,\tau\kappa\tau)) (\delta_\upsilon \lambda)
\\
\!\!\!&=&\!\!\! \phi_3 (A) (\delta_\upsilon \phi_3 (B) {\delta_\upsilon}^{-1})\, (\phi(\theta) \phi({\zeta}^{-1})(\beta \,\tau\kappa\tau)(\delta_\upsilon \lambda)
\\
\!\!\!&=&\!\!\! \phi_3 (A) (\delta_\upsilon \phi_3 (B) {\delta_\upsilon}^{-1})\, \phi(\theta) \beta \phi({\zeta}^{-1})\tau\kappa\tau (\lambda \delta_\upsilon )
\\[1mm]
\!\!\!&=&\!\!\! \phi_3 (A) (\delta_\upsilon \phi_3 (B) {\delta_\upsilon}^{-1})\, \phi(\theta) \beta (\lambda \phi(\zeta) \lambda^{-1}) \lambda \kappa \delta_\upsilon
\\[1mm]
\!\!\!&=&\!\!\! \phi_3 (A) (\delta_\upsilon \phi_3 (B) {\delta_\upsilon}^{-1})\, \phi(\theta) \beta \lambda \phi(\zeta)  \kappa \delta_\upsilon 
\\[1mm]
\!\!\!&=&\!\!\! \phi_3 (A) (\delta_\upsilon \phi_3 (B) {\delta_\upsilon}^{-1})\, \phi(\theta) \beta (\lambda  \delta_\upsilon) \phi(\zeta)  \kappa  
\\[1mm]
\!\!\!&=&\!\!\! \phi_3 (A)\phi(\theta) \beta (\delta_\upsilon \phi_3 (B) {\delta_\upsilon}^{-1})\,  (\delta_\upsilon \lambda) \phi(\zeta)  \kappa 
\\
\!\!\!&=&\!\!\! \phi_3 (A)\phi(\theta) \beta \delta_\upsilon \phi_3 (B)  \lambda \phi(\zeta)  \kappa 
\\[1mm]
\!\!\!&=&\!\!\! \phi_3 (A)\phi(\theta) \beta (\delta_\upsilon \lambda) \phi_3 (B)  \phi(\zeta)  \kappa
\\[1mm]
\!\!\!&=&\!\!\! \phi_3 (A)  \varphi_{{}_\text{E7}}(\theta, \beta) \,\nu \phi_3 (B)  \varphi_{{}_\text{E7}}(\zeta, \kappa),
\end{eqnarray*}
where $\phi$ is defined in Theorem 3.4.3. Similarly, the other cases are shown.

We shall show that $\varphi_{4234}$ is surjection. Let $\alpha \in (E_8)^{\upsilon} \cap (E_8)^{\iota_{\omega}}$. From $(E_8)^{\upsilon} \cap (E_8)^{\iota_{\omega}} \subset (E_8)^{\upsilon}$, there exist $A  \in S\!U(2)$ and $\delta \in E_7$ such that $\alpha=\varphi_{{}_\text{E9}}(A, \delta)$ (Theorem 3.5.2). Moreover, since $\alpha=\varphi_{{}_\text{E9}}(A, \delta) \in (E_8)^{\iota_\omega}$, that is, $\iota_\omega \varphi_{{}_\text{E9}}(A, \delta) {\iota_\omega}^{-1}=\varphi_{{}_\text{E9}}(A, \delta)$, using $\iota_\omega \phi_3(A) {\iota_\omega}^{-1}=\phi_3({}^t A^{-1})$ (Lemma 4.21.4 (2)), we have that $\varphi_{{}_\text{E9}}({}^t A^{-1}, \iota\delta\iota^{-1})=\varphi_{{}_\text{E9}}(A, \delta)$. Hence, it follows that
$$
\left \{
         \begin{array}{l}
                {}^t A^{-1}= A
                         \vspace{3mm}\\
                \iota\delta\gamma{\iota}^{-1} = \delta
         \end{array}\right.\qquad
 \text{or}\qquad \left \{
         \begin{array}{l}
                {}^t A^{-1}=-A
                         \vspace{3mm}\\
                \iota\delta\gamma{\iota}^{-1} = -\delta.
         \end{array}\right. 
$$
In the former case, we see that $A \in S\!O(2)$ and $\delta \in (E_7)^\iota \cong (U(1) \times E_6)/\Z_3$. Hence, there exist $\theta \in U(1)$ and $\beta \in E_6$ such that $\delta =\varphi_{{}_\text{E7}}(\theta, \beta)$ (Theorem 3.4.3). Thus we have that $\alpha=\varphi_{{}_\text{E9}}(A, \delta)=\varphi_{{}_\text{E9}}(A, \varphi_{{}_\text{E7}}(\theta, \beta))=\varphi_{4234}((A, \theta, \beta), 1)$.
In the latter case, we see that $A=A'(iI), A' \in S\!O(2)$ and $\delta=\delta' \lambda, \delta' \in (E_7)^\iota$. Hence,  as in the former case, we have that
\begin{eqnarray*}
\alpha \!\!\!&=&\!\!\!\varphi_{{}_\text{E9}}(A, \delta)=\varphi_{{}_\text{E9}}(A(iI), \delta' \lambda) =\phi_3((A(iI))(\delta' \lambda)=\phi_3(A')\phi_3(iI)(\delta' \lambda)
\\
\!\!\!&=&\!\!\!\phi_3(A')\delta' (\phi_3(iI) \lambda)=\phi_3(A')\delta'(\delta_\upsilon \lambda)=\varphi_{{}_\text{E9}}(A', \delta')(\delta_\upsilon \lambda)
\\
\!\!\!&=&\!\!\!\varphi_{{}_\text{E9}}(A',\varphi_{{}_\text{E7}}(\theta', \beta') )\nu=\varphi_{4234}((A', \theta', \beta'), \nu).
\end{eqnarray*}
Thus $\varphi_{4234}$ is surjection. 

Finally, we shall determine $\Ker\,\varphi_{4234} $. From the definition of kernel, it is as follows:
\begin{eqnarray*}
\Ker\,\varphi_{4234}\!\!\!\!\!&=&\!\!\!\!\!\{((A, \theta, \beta),1)\,|\,\varphi_{4234}((A, \theta,\beta), 1)=1 \} \cup \{((A, \theta, \beta),\nu)\,|\,\varphi_{4234}((A, \theta,\beta), \nu)=1 \}
\\
\!\!\!\!\!&=&\!\!\!\!\!\{((A, \theta, \beta),1)\,|\,\varphi_{{}_\text{E9}}(A, \varphi_{{}_\text{E7}}(\theta, \beta))=1\} \cup \{((A, \theta, \beta),\nu)\,|\,\varphi_{{}_\text{E9}}(A, \varphi_{{}_\text{E7}}(\theta, \beta))\nu=1\}.
\end{eqnarray*}
Here, for the left hand side case, we have that
\begin{eqnarray*}
&& \{((A, \theta, \beta),1)\,|\,\varphi_{{}_\text{E9}}(A, \varphi_{{}_\text{E7}}(\theta, \beta))=1\}
\\
\!\!\!&=&\!\!\! \{((A, \theta, \beta),1)\,|\,A=\pm E, \varphi_{{}_\text{E7}}(\theta, \beta)=\pm 1 \}
\\
\!\!\!&=&\!\!\! \{((A, \theta, \beta),1)\,|\,A=\pm E, \phi(\theta)\beta=\pm 1\}
\\
\!\!\!&=&\!\!\! \{((E,1,1), 1),  (E, \omega, \phi_2(\omega^2)),(E, \omega^2, \phi_2(\omega),
\\
&& \qquad ((-E-1,1), 1), ((-E, -\omega, \phi_2(\omega^2)), 1), ((-E, -\omega^2, \phi_2(\omega), 1)  \}
\\
\!\!\!&=&\!\!\! \{(E,1,1), (-E-1,1)  \} \times \{(E,1,1), (E, \omega, \phi_2(\omega^2)), (E, \omega^2, \phi_2(\omega)  \}.
\end{eqnarray*}
For the right hand case, in a similar way as the argument of kernel in Theorem 4.21.5, we have that $\{((A, \theta, \beta),\nu)\,|\,\varphi_{{}_\text{E9}}(A, \varphi_{{}_\text{E7}}(\theta, \beta))\nu=1\}=\phi$. Thus we can obtain that 
$$
\Ker\,\varphi_{4234}=\{(E,1,1), (-E-1,1)  \} \times \{(E,1,1), (E, \omega, \phi_2(\omega^2)), (E, \omega^2, \phi_2(\omega)  \} \cong \Z_2 \times \Z_3.
$$

Therefore we have the required isomorphism 
$$
(E_8)^{\upsilon} \cap (E_8)^{\iota_{\omega}} \cong (S\!O(2) \times U(1) \times E_6)/(\Z_2 \times \Z_3) \rtimes \mathcal{Z}_2.
$$
\end{proof}


\vspace{5mm}

\begin{flushright}

\begin{tabular}{l}
 Toshikazu Miyashita \\
Ueda-Higashi High School \\
Ueda City, Nagano, 386-8683,\\
Japan \\
e-mail: anarchybin@gmail.com 
\end{tabular}

\end{flushright}



\begin{thebibliography}{99}
\bibitem{Bahturin} Y. Bahturin and M. Goze, $\mathbb{Z}_2 \times \mathbb{Z}_2$-symmetric spaces, Pacific Journal of Math. 236-1(2008), 1-21.
\bibitem{Jing-Song Huang}Jing-Song Huang and Jun Yu, Klein four subgroups of Lie algebra automorphisms, Pacific Journal of Math. 262-2(2013), 397-420.
\bibitem{Andreas} A. Kollross, Exceptional $\mathbb{Z}_2 \times \mathbb{Z}_2$-symmetric spaces, Pacific Journal of Math. 242-1(2009), 113-130.
\bibitem{Lutz} R. Lutz, Sur lag{\'e}om{\'e}trie des espaces $\varGamma$-sym{\'e}triques, C. R. Acad. Sci. Paris S{\'e}r. I Math. 293, no. 1,
55-58 (1981)
\bibitem{M.Y.01}T. Miyashita  and I. Yokota, 
Fixed points subgroups $G^{\sigma, \sigma'}$ by two involutive automorphisms $\sigma, \sigma' $ of compact exceptional Lie group $G =F_4, E_6 $ and  $E_7$, Math. J. Toyama Univ. 24(2001), 135-149.
\bibitem{M.01}T. Miyashita, Fixed points subgroups by two involutive automorphisms $\sigma, \gamma $ of compact exceptional Lie group $G =F_4, E_6 $ and  $E_7$, Tsukuba J. Math. 27-1(2003), 199-215.
\bibitem{M.02}T. Miyashita, Fixed points subgroups by two involutive automorphisms $\gamma, \gamma' $ of compact exceptional Lie group $G =G_2, F_4, E_6 $ and  $E_7$, Yokohama Math. J.  53(2006), 9-38.
\bibitem{Miyashitatoshikazu} T. Miyashita, Fixed points subgroups $G^{\sigma, \sigma'}$ by two involutive automorphisms $\sigma, \sigma'$ of exceptional compact Lie group $G$. Part II:$G=E_8$, Kyushu Journal of Math. 64-2(2010), 221-237.
\bibitem{MS} Mitsuo Sugiura, Lie group theory (in Japanese), Kyoritsu-shuppan, Tokyo, 2000.
\bibitem{Yokotaichiro0} I. Yokota, Exceptional simple Lie groups, arXiv:0902.0431vl. 
\bibitem{Yokotaichiro1} I. Yokota, Realizations of involutive automorphisms $\sigma$ of exceptional Lie groups $G$, Part I, $G = G_2, F_4, E_6$, Tsukuba J.
Math. 14(1990), 185-223.
\bibitem{Yokotaichiro2} I. Yokota, Realizations of involutive automorphisms $\sigma$ of exceptional Lie groups $G$, Part II, $G = E_7$, Tsukuba J.
Math. 14(1990), 379-404.
\bibitem{Yokotaichiro3} I. Yokota, Realizations of involutive automorphisms $\sigma$ of exceptional Lie groups $G$, Part III, $G = E_8$, Tsukuba J.
Math. 15(1991), 301-314.
\end{thebibliography}
\end{document}